\documentclass[10pt]{article}
\usepackage{amsmath, amssymb,amsxtra}
\usepackage{courier}
\usepackage{color}
\usepackage{pstricks, multido}

\definecolor{dkgreen}{rgb}{0,0.6,0}
\definecolor{gray}{rgb}{0.5,0.5,0.5}

\newtheorem{thm}{{\sc Theorem}}[section]
\newtheorem{prop}[thm]{{\sc Proposition}}
\newtheorem{cor}[thm]{{\sc Corollary}}
\newtheorem{lem}[thm]{{\sc Lemma}}
\newtheorem{df}[thm]{{\sc Definition}}

\newenvironment{proof}{\begin{sc}\noindent Proof: \end{sc}}{
     \hbox to 2em{}\nobreak\hfill$\blacksquare$\par\medskip}

\newcommand{\LR}{\hbox{Little\-wood-Richard\-son}}

\catcode`\@=11 \font\linear=line10 scaled \magstep5
\def\slant{{\linear ,}}
\def\young@lign{\everycr{}\tabskip0pt\halign}
\def\Mathstrut@{\setbox\z@\hbox{$($}\setbox\tw@\null\ht\tw@\ht\z@\dp\tw@\dp\z@
 \box\tw@}
\def\b@m#1{$\m@th\underline{#1}$}
\def\t@p#1{$\m@th\overline{#1}$}

\def\young{
\setbox\strutbox=\hbox{\vrule height3pt depth3.5pt width\z@}
 \offinterlineskip%
 {}\,\vcenter
\bgroup
     \def\1{{}}
     \def\2{\1&\1}
     \def\3{\1&\1&\1}
     \def\4{\1&\1&\1&\1}
\let\\=\cr
 \tabskip0pt\baselineskip0pt\m@th
 \young@lign
 \bgroup\vrule\b@m{\hbox to .75 em{\strut\hfil$##$\hfil}}\vrule %
   &&\b@m{\hbox to .75em{\strut\hfil$##$\hfil}}\vrule\crcr
   \noalign{\hrule}
  }

\def\endyoung{\egroup \egroup\,}
\def\slyoung{
   \setbox\strutbox=\hbox{\vrule height10pt depth2pt width\z@}
   \offinterlineskip%
   {}\,\vcenter
\bgroup
     \def\1{{}}
     \def\2{\1&\1}
     \def\3{\1&\1&\1}
     \def\4{\1&\1&\1&\1}
\let\\=\cr
 \tabskip0pt\baselineskip0pt\m@th
 \young@lign
 \bgroup\lower2pt\hbox{\slant}\kern-25pt\b@m{\hbox to 2.5em{\strut\hfil$##\;\;$}}\vrule
   &&\b@m{\hbox to 2.5em{\strut\hfil$##$\hfil}}\vrule\crcr
   \noalign{\hrule}
  }

\def\endslyoung{\egroup \egroup\,}
\def\frame #1#2#3#4{\vbox{\hrule height #1pt
 \hbox{\vrule width #1pt\kern #2pt
 \vbox{\kern #2pt
 \vbox{\hsize #3\noindent #4}
 \kern #2pt}
 \kern #2pt\vrule width #1pt}
 \hrule height0pt depth #1pt}}

\def\nframe #1#2#3#4{\vbox{
 \hrule height #1pt width0pt
 \hbox{\vrule height0pt width #1pt\kern #2pt
 \vbox{\kern #2pt
 \vbox{\hsize #3\noindent #4}
 \kern #2pt}
 \kern #2pt\vrule width #1pt height0pt}
 \hrule height0pt width0pt}}

\def\leftwall #1#2#3#4{\vbox{
 \hrule height #1pt width0pt
 \hbox{\vrule width #1pt\kern #2pt
 \vbox{\kern #2pt
 \vbox{\hsize #3\noindent #4}
 \kern #2pt}
 \kern #2pt\vrule width #1pt height0pt}
 \hrule height0pt width0pt}}

\def\rightwall #1#2#3#4{\vbox{
 \hrule height #1pt width0pt
 \hbox{\vrule height0pt width #1pt\kern #2pt
 \vbox{\kern #2pt
 \vbox{\hsize #3\noindent #4}
 \kern #2pt}
 \kern #2pt\vrule width #1pt }
 \hrule height0pt width0pt}}

\definecolor{light}{gray}{.75}

\def\muone #1{\frame{.3}{2}{140pt}{\centerline{#1}\vphantom{(}}}

\def\mutwo #1{\frame{.3}{2.4}{80pt}{\centerline{#1}\vphantom{(}}}
\def\muthr #1{\frame{.3}{2.4}{40pt}{\centerline{#1}\vphantom{$k_{11}$(}}}
\def\mufour #1{\frame{.3}{2}{18pt}{\centerline{#1}\vphantom{$k_{11}$(}}}
\def\koneone #1{\frame{.3}{2}{80pt}{\centerline{#1}\vphantom{(}}}
\def\konetwo #1{\frame{.3}{2.4}{50pt}{\centerline{#1}\vphantom{(}}}
\def\konethr #1{\frame{.3}{2.4}{30pt}{\centerline{#1}\vphantom{(}}}
\def\konefour #1{\frame{.3}{2}{15pt}{\centerline{#1}\vphantom{(}}}
\def\ktwotwo #1{\frame{.3}{2.4}{70pt}{\centerline{#1}\vphantom{(}}}
\def\ktwothr #1{\frame{.3}{2.4}{40pt}{\centerline{#1}\vphantom{(}}}
\def\ktwofour #1{\frame{.3}{2}{20pt}{\centerline{#1}\vphantom{(}}}
\def\kthrthr #1{\frame{.3}{2.4}{35pt}{\centerline{#1}\vphantom{()}}}
\def\kthrfour #1{\frame{.3}{2}{50pt}{\centerline{#1}\vphantom{(}}}
\def\kfourfour #1{\frame{.3}{2}{25pt}{\centerline{#1}\vphantom{(}}}

\def\nl{\hfill\break} 

\begin{document}
\title{Matrix Pairs over Valuation Rings and ${\mathbb
R}$-Valued \LR\ Fillings}
\author{Glenn D.\ Appleby, Tamsen Whitehead\\
Department of Mathematics\\
Santa Clara University\\
Santa Clara, CA 95053\\
{\tt gappleby@scu.edu, tmcginley@scu.edu}}
\date{ }

\noindent {\bf Matrix Pairs over Valuation Rings and ${\mathbb
R}$-Valued \LR\ Fillings}

\medskip

\noindent
Glenn D. Appleby\\
 Tamsen Whitehead\\
{\em Department of Mathematics\\
and Computer Science,\\
Santa Clara University\\
Santa Clara,  CA 95053}\\
gappleby@scu.edu, tmcginley@scu.edu


\noindent \mbox{} \hrulefill \mbox{}
\section{Introduction}
In this paper we will present some significant extensions of results
relating invariants of matrices over valuation rings, and the
combinatorics of \LR\ fillings and sequences, first obtained in
\cite{lrcon} in the case of discrete valuation rings of
characteristic zero. There, a construction was made to associate
\LR\ fillings to the orbit of a pair of matrices $(M,N)$ over a
discrete valuation ring $R$ under a natural notion of paired matrix
equivalence. Here, we shall first extend these results to power
series rings with ${\mathbb R}$-valued exponents. This will allow us
to extend also the classical definition of \LR\ fillings of skew
shapes that not only goes beyond integer fillings, but will actually
allow fillings of arbitrary real-valued length, including rows of
``negative'' length.

With these extensions, we will be able to consider continuous
deformations of matrix parameters yielding invariants of our matrix
pairs, resulting in similar continuous deformations of associated
real-valued \LR\ fillings (called here ``L${\mathbb R}$-fillings").
These deformation results allow us to calculate effectively the
``dynamics" of L${\mathbb R}$ fillings in way not obtainable when
restricted to the case of non-negative integer fillings.  That is,
we can construct continuously parameterized families of L${\mathbb
R}$-fillings that interpolate, for example, between classical
integer-valued fillings. We will use these results to describe, in
matrix terms, a ``combinatorial core" of a filling, and the
relationship between it and the dynamics mentioned above.

In our previous work \cite{lrcon}, we showed that a matrix pair
$(M,N)$ actually yields {\em two} fillings, one a filling of
$\lambda / \mu$ with content $\nu$, and another of skew-shape
$\lambda / \nu$ with content $\mu$. In this paper we establish that
these fillings are in bijection, and do not depend on the particular
matrix realization of the filling. Further, we shall show that this
bijection (proved here for L${\mathbb R}$ fillings) extends the
previously established bijection between \LR\ fillings described in
\cite{kerber}, and generalized in the paper of \cite{Sottile}.  A
central feature of the proof will be an ``ordering lemma" which
describes the growth of L${\mathbb R}$-fillings under a continuous
deformation of the partitions associated to a matrix pair.  We will
use these results to quickly prove some results describing the
relationship between \LR\ and L${\mathbb R}$ fillings between
different shapes. We hope this will add new insight into these
classical bijections, and that the invariants of matrices over
valuation rings will become a new and effective way to investigate
the combinatorics of \LR\, and L${\mathbb R}$, fillings.

\section{\LR\ and L${\mathbb R}$ Fillings}
In this section we will define the notions of ${\mathbb R}$-valued
\LR\ fillings (``L${\mathbb R}$-fillings"). Before we can
demonstrate that these generalizations warrant attention, we will
need to lay further groundwork in the sequel where we will develop
the machinery to associate L${\mathbb R}$-fillings to invariants of
matrix pairs over valuation rings. For now, we will try to emphasize
the connections to the classical definitions, and the important
differences the ${\mathbb R}$-valued setting affords.

\begin{df} We will say a sequence $\mu = (\mu_1, \mu_2, \dots,
\mu_r)$ is an {\em ${\mathbb R}$-partition}, provided each $\mu_{i}
\in {\mathbb R}$, $1 \leq i \leq r$, and
\[ \mu_1 \geq \mu_2 \geq \dots \geq \mu_r. \]
We will say, in this case, the ${\mathbb R}$-partition $\mu$ has
{\em length} $r$.  Given an ${\mathbb R}$-partition $\mu= (\mu_1,
\mu_2, \dots, \mu_r)$, we shall use the notation $| \mu |$ to denote
\[  | \mu | = \mu_1 + \mu_2 + \cdots + \mu_r. \]
Given two ${\mathbb R}$-partitions $\mu$ and $\lambda$ of the same
length $r$, we shall let
\[ \mu \subseteq \lambda \]
to denote the condition that $\mu_i \leq \lambda_i$, for $1 \leq i
\leq r.$
\end{df}

Note that entries of ${\mathbb R}$-partitions may certainly be
negative, as, indeed, may the sum $| \mu |$.  Also, while many
authors use the term ``length'' to denote the number of non-zero
entries in an (integer-valued) partition, we will use ``length'' to
denote the number of entries in the sequence, in order to make
precise statements regarding the sizes of matrices associated to
such ${\mathbb R}$-partitions.

Below is our central combinatorial definition.  It generalizes the
classical definition of \LR\ sequences found, for example, in
\cite{mac}.

\begin{df}  Let $(\mu, \nu, \lambda)$ be a triple of ${\mathbb R}$-partitions such
that $|\mu | + |\nu| = | \lambda|$.  Let us define an indexed set of
real numbers $\{ k_{ij} \}$ (called the {\em parts} of the filling),
for $1 \leq i \leq r$, $i \leq j \leq r$, an {\em ${\mathbb
R}$-valued \LR\ filling of type $(\mu, \nu; \lambda)$},
(abbreviated: L${\mathbb R}$ filling) if

\begin{enumerate}
 \item (L${\mathbb R}$1) (${\mathbb R}$-Sums) For all $1 \leq j \leq
r$, and $1 \leq i \leq r$,
\[  \mu_{j}+\sum_{s=1}^{j} k_{sj} = \lambda_{j},
\quad \hbox{and} \quad \sum_{s=i}^{r}k_{is} = \nu_{i}. \]
\item (L${\mathbb R}$2)  (${\mathbb R}$-Non-negativity) For all $i$
and $j$, where $i < j$, we have $k_{ij} \geq 0$.  There is no
condition on the parts $k_{ii}$ except $k_{ii} \in {\mathbb R}$.
\item (L${\mathbb R}$3) (${\mathbb R}$-Column Strictness) For each
$j$, for $2 \leq j \leq r$ and $1 \leq i \leq j$ we require
\[  \mu_{j}+ k_{1j} + \cdots k_{ij} \leq \,  \mu_{(j-1)} +k_{1,(j-1)} +
 \cdots + k_{(i-1),(j-1)} . \label{R-cs}
\]
\item (L${\mathbb R}$4)  (${\mathbb R}$-Word Condition) For all $1
\leq i \leq r-1$, $i \leq j \leq r-1$,
\[   \sum_{s=i+1}^{j+1} k_{(i+1),s} \  \leq \  \sum_{s=i}^{j}k_{is}. \] \label{R-wd}
\end{enumerate}
Let the set of all L${\mathbb R}$-fillings of type $(\mu, \nu;
\lambda)$ be denoted $L{\mathbb R}(\mu, \nu; \lambda)$.
\end{df}

We emphasize again that, in particular, it is possible for terms
$k_{ii}$ in an L${\mathbb R}$ filling to be negative (as it is also
possible for the parts of the partitions $\mu$, $\nu$ and
$\lambda$).

The definition above extends the classical definition of \LR\
fillings of skew shapes.  In classical \LR\ fillings, a non-negative
integer $k_{ij}$ denotes the {\em number} of $i$'s appearing in row
$j$ of a skew shape. For example, with $\lambda =(15,10,8,6)$, $\mu
=(9,5,2,1)$, and $\nu =(11,6,3,2)$ a \LR\ filling of the skew shape
$\lambda / \mu$ with content $\nu$ could be depicted:

\vspace{0.3in} \psset{unit=0.7cm} \hspace{1in}
\begin{pspicture}(5,4) \psframe(0,0)(1,1) \psframe(0,1)(1,2)
\psframe(0,2)(1,3) \psframe(0,3)(1,4)

\psframe(1,0)(2,1) \psframe(1,1)(2,2)
\rput{*0}(1.5,0.5){$1$}\psframe(1,2)(2,3) \psframe(1,3)(2,4)

\psframe(2,0)(3,1) \psframe(2,1)(3,2)\rput{*0}(2.5,0.5){$2$}
\psframe(2,2)(3,3)\rput{*0}(2.5,1.5){$1$} \psframe(2,3)(3,4)

\psframe(3,0)(4,1) \psframe(3,1)(4,2)\rput{*0}(3.5,0.5){$3$}
\psframe(3,2)(4,3) \rput{*0}(3.5,1.5){$1$} \psframe(3,3)(4,4)

\psframe(4,0)(5,1) \psframe(4,1)(5,2)
\rput{*0}(4.5,0.5){$4$}\psframe(4,2)(5,3) \rput{*0}(4.5,
1.5){$2$}\psframe(4,3)(5,4)

\psframe(5,0)(6,1) \psframe(5,1)(6,2)
\rput{*0}(5.5,0.5){$4$}\psframe(5,2)(6,3)
\psframe(5,3)(6,4)\rput{*0}(5.5,1.5){$2$} \rput{*0}(5.5,2.5){$1$}

\psframe(6,1)(7,2) \psframe(6,2)(7,3) \psframe(6,3)(7,4)
\rput{*0}(6.5,1.5){$3$} \rput{*0}(6.5,2.5){$1$}

\psframe(7,1)(8,2) \psframe(7,2)(8,3) \psframe(7,3)(8,4)
\rput{*0}(7.5,1.5){$3$} \rput{*0}(7.5,2.5){$2$}

\psframe(8,2)(9,3) \psframe(8,3)(9,4) \rput{*0}(8.5,2.5){$2$}

\psframe(9,2)(10,3) \psframe(9,3)(10,4)  \rput{*0}(9.5,2.5){$2$}
\rput{*0}(9.5,3.5){$1$}

 \psframe(10,3)(11,4) \rput{*0}(10.5,3.5){$1$}

  \psframe(11,3)(12,4)\rput{*0}(11.5,3.5){$1$}

   \psframe(12,3)(13,4)\rput{*0}(12.5,3.5){$1$}

    \psframe(13,3)(14,4)\rput{*0}(13.5,3.5){$1$}

     \psframe(14,3)(15,4)\rput{*0}(14.5,3.5){$1$}
\end{pspicture}

\vspace{0.2in}

In our generalized setting, we will now interpret $k_{ij}$ as
denoting the {\em length} of a portion of the skew shape.
Pictorially, the above filling will be thought of as:

\vspace{0.3in} \psset{unit=0.7cm} \hspace{1in}
\begin{pspicture}(5,4)\psframe(0,0)(1,1)
\psframe(1,0)(2,1)\rput{*0}(1.5,0.5){$k_{14}$}
\rput{*0}(2.5,0.5){$k_{24}$} \rput{*0}(3.5,0.5){$k_{34}$}
\psframe(2,0)(3,1) \psframe(3,0)(4,1) \psframe(2,1)(4,2)
\rput{*0}(3,1.5){$k_{13}$} \rput{*0}(0.5,0.5){$\mu_4$}
 \psframe(0,1)(2,2) \psframe(0,2)(5,3)
\rput{*0}(1,1.5){$\mu_{3}$} \rput{*0}(2.5,2.5){$\mu_2$}
\psframe(0,3)(9,4) \rput{*0}(4.5,3.5){$\mu_1$} \psframe(9,3)(15,4)
\rput{*0}(12,3.5){$k_{11}$} \psframe(5,2)(7,3)
\rput{*0}(6,2.5){$k_{12}$} \psframe(7,2)(10,3)
\rput{*0}(8.5,2.5){$k_{22}$} \psframe(4,0)(6,1)
\rput{*0}(5,0.5){$k_{44}$} \psframe(4,1)(6,2)
\rput{*0}(5,1.5){$k_{23}$} \psframe(6,1)(8,2)
\rput{*0}(7,1.5){$k_{33}$}
\end{pspicture}

\vspace{0.2in}

\medskip

 Note that in our generalized,
${\mathbb R}$-valued \LR\ filling, $k_{ij}$ is a {\em real number}.
We shall call the parts $k_{ij}$ of a L${\mathbb R}$ filling such
that $i < j$ the {\em interior parts}, which must be non-negative by
definition. Parts of the form $k_{ii}$ we will call the {\em edge
parts}.  These are, in our formulation, allowed to be an arbitrary
real. Thus, a ${\mathbb R}$-valued \LR\ ``filling" of a skew shape
made out of non-negative ${\mathbb R}$-partitions, with all parts
non-negative, might be depicted as:

\bigskip
\hspace{1.4in} \vbox{ \offinterlineskip \openup-0pt \nl
\muone{$\mu_1 \vphantom{k_{11}}$}\koneone{$k_{11}$} \nl
\mutwo{$\mu_2
\vphantom{k_{22}}$}\konetwo{$k_{12}\vphantom{\mu_{2}}$}\ktwotwo{$k_{22}$}\nl
\muthr{$\mu_3$}\konethr{$k_{13}\vphantom{\mu_3}$}\ktwothr{$k_{23}$}\kthrthr{$k_{33}$}\nl
\mufour{$\mu_4$}\konefour{$k_{14}$}\ktwofour{$k_{24}$}\kthrfour{$k_{34}$}\kfourfour{$k_{44}$}\nl},

However, we will have to be more creative when drawing fillings for
which some $k_{ii} < 0$, or for ${\mathbb R}$-partitions with
negative parts. First, if $\mu = (\mu_1, \mu_2, \dots, \mu_r)$, and
we have some $\mu_i < 0$, we will depict this with a fixed choice of
center line, denoting the (horizontal distance from) the origin.
Further, negative ``boxes" will be shaded, with a smaller height.
 Thus, if $\mu = (\mu_1, \mu_2, \mu_3, \mu_4) = (7,3,-2, -4)$, we will draw this as

\vspace{0.3in} \psset{unit=0.7cm} 
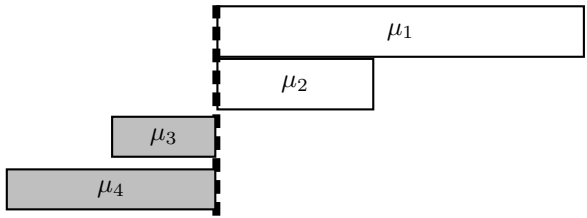
\hspace{2in}
\begin{pspicture}(11,4)
\psline[linewidth=3pt, linestyle=dashed](0,0)(0,4)
\psframe(0,3)(7,4) \rput{*0}(3.5,3.5){$\mu_1$} \psframe(0,2)(3,3)
\rput{*0}(1.5,2.5){$\mu_2$}
\psframe[fillstyle=solid,fillcolor=lightgray](-2,1.1)(0,1.9)
\rput{*0}(-1,1.5){$\mu_3$}
\psframe[fillstyle=solid,fillcolor=lightgray](-4,0.1)(0,0.9)
\rput{*0}(-2,0.5){$\mu_4$}
\end{pspicture}

\vspace{0.2in} In these pictures, we will construct the diagram by
beginning at the "origin", and first attach the boxes for $\mu$, as
shown above.  When the part of $\mu$ is positive, the end of the row
moves to the right, and when negative it moves to the left.  We then
add the (necessarily non-negative) interior parts of the filling to
the end of each row, moving to right.  We obtain a picture such as:

\vspace{0.3in} \psset{unit=0.7cm} 
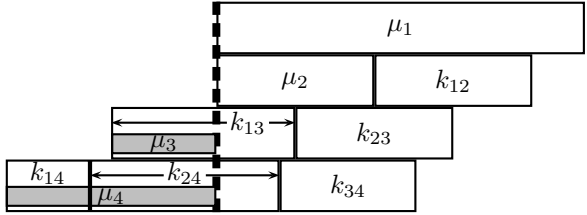
\hspace{2in}
\begin{pspicture}(11,4)
\psline[linewidth=3pt, linestyle=dashed](0,0)(0,4)
\psframe(0,3)(7,4) \rput{*0}(3.5,3.5){$\mu_1$} \psframe(0,2)(3,3)
\rput{*0}(1.5,2.5){$\mu_2$}
\psframe[fillstyle=solid,fillcolor=lightgray](-2,1.1)(0,1.5)
\rput{*0}(-1,1.3){$\mu_3$}
\psframe[fillstyle=solid,fillcolor=lightgray](-4,0.1)(0,0.5)
\rput{*0}(-2,0.3){$\mu_4$}
\psframe(3,2)(6,3)\psline{->}(0.2,1.7)(-2,1.7)
\psline{->}(1,1.7)(1.5,1.7) \rput{*0}(4.5,2.5){$k_{12}$}
\psframe(-2,1)(1.5,2) \rput{*0}(0.6,1.7){$k_{13}$}
\psframe(1.5,1)(4.5,2) \rput{*0}(3,1.5){$k_{23}$}
\psframe(-4,0)(-2.4,1) \rput{*0}(-3.2,0.75){$k_{14}$}
\psframe(-2.4,0)(1.2,1)\rput{*0}(-0.6,0.75){$k_{24}$}
\psline{->}(-1,0.7)(-2.4,0.7)\psline{->}(-0.2,0.7)(1.2,0.7)
\psframe(1.2,0)(3.8,1) \rput{*0}(2.5,0.5){$k_{34}$}
\end{pspicture}

\vspace{0.3in} To contend with possibly negative values for the edge
parts $k_{ii}$, we again represent them with shaded, thinner parts,
and attach them to the ends of the rows (moving to the left for
negative edge parts, and to the right for positive ones):

\vspace{0.3in} \psset{unit=0.7cm} 
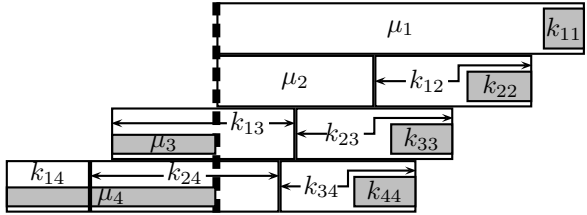
\hspace{2in}
\begin{pspicture}(11,4)
\psline[linewidth=3pt, linestyle=dashed](0,0)(0,4)
\psframe(0,3)(7,4) \rput{*0}(3.5,3.5){$\mu_1$} \psframe(0,2)(3,3)
\rput{*0}(1.5,2.5){$\mu_2$}
\psframe[fillstyle=solid,fillcolor=lightgray](-2,1.1)(0,1.5)
\rput{*0}(-1,1.3){$\mu_3$}
\psframe[fillstyle=solid,fillcolor=lightgray](-4,0.1)(0,0.5)
\rput{*0}(-2,0.3){$\mu_4$}
\psframe(3,2)(6,3)\psline{->}(0.2,1.7)(-2,1.7)
\psline{->}(1,1.7)(1.5,1.7) \rput{*0}(4,2.5){$k_{12}$}
\psframe(-2,1)(1.5,2) \rput{*0}(0.6,1.7){$k_{13}$}
\psframe(1.5,1)(4.5,2) \rput{*0}(2.4,1.5){$k_{23}$}
\psline{->}(2.8,1.5)(3, 1.5)(3,1.8)(4.5,1.8)
\psline{->}(2.1,1.5)(1.5, 1.5) \psframe(-4,0)(-2.4,1)
\rput{*0}(-3.2,0.75){$k_{14}$}
\psframe(-2.4,0)(1.2,1)\rput{*0}(-0.6,0.75){$k_{24}$}
\psline{->}(-1,0.7)(-2.4,0.7)\psline{->}(-0.2,0.7)(1.2,0.7)
\psframe(1.2,0)(3.8,1) \rput{*0}(2,0.5){$k_{34}$}
\psframe[fillstyle=solid,fillcolor=lightgray](6.2,3.1)(7,3.9)
\psline{->}(4.4,2.5)(4.6, 2.5)(4.6,2.8)(6,2.8)
\rput{*0}(6.6,3.5){$k_{11}$}
\psframe[fillstyle=solid,fillcolor=lightgray](4.75,2.1)(6,2.7)
\rput{*0}(5.4,2.4){$k_{22}$} \psline{->}(3.6,2.5)(3,2.5)
\psframe[fillstyle=solid,fillcolor=lightgray](3.3,1.1)(4.5,1.7)
\rput{*0}(3.9,1.4){$k_{33}$}
\psframe[fillstyle=solid,fillcolor=lightgray](2.6,0.1)(3.8,0.7)
\rput{*0}(3.2,0.4){$k_{44}$} \psline{->}(2.3,0.5)(2.5,
0.5)(2.5,0.8)(3.8,0.8) \psline{->}(1.6,0.5)(1.2, 0.5)
\end{pspicture}

\vspace{0.3in}

In the above diagram, $\lambda$ is determined in each row as the
signed distance from the origin to the end of the edge part in that
row.  We also see that row $i$ of the ``skew shape" $\lambda / \mu$
has has size $\lambda_{i} - \mu_i$ (which may be negative), and that
the sum of the sizes of the rows of $\lambda / \mu$ is $| \nu |$.
 The meaning, visually, of the column-strictness (L${\mathbb R}$3)
and word conditions (L${\mathbb R}$4) may be unfamiliar, but can be
reconciled with the linear inequalities in their definitions, and
our conventions for drawing diagrams.  Note, in particular, that by
L${\mathbb R}$3, the Word Condition, we must have \[ k_{11} \geq
k_{22} \geq \dots \geq k_{rr}. \]  Thus if, for any $i$, $1 \leq i
\leq r$, we have $k_{ii} < 0$, then $k_{(i+\kappa)(i+\kappa)} < 0$
for all integers $\kappa \geq 0$.

For future reference, in any L${\mathbb R}$-filling, the skew shape
formed by the parts $k_{ij}$ for a fixed $i$ will be called ``the
$i$-strip" of the filling.

Alternately, one may define L${\mathbb R}$ fillings in terms of {\em
${\mathbb R}$-valued \LR\ sequences} ({\em L${\mathbb R}$
sequences}). That is, as a sequence of partitions $\lambda^{(0)}
\subseteq \lambda^{(1)} \subseteq \dots \subseteq \lambda^{(r)}$
such that the $i$-strip of the filling is formed by the skew shape
$\lambda^{(i)} / \lambda^{(i-1)}$.  This way, if we set $k_{ij}$ to
be the length of the portion of the skew shape $\lambda^{(i)} /
\lambda^{(i-1)}$ that appears in row $j$, the collection $\{ k_{ij}
\}$ forms an L${\mathbb R}$ filling.

Of course, one is at liberty to invent any number of rules for
constructing ${\mathbb R}$-valued diagrams, but doing so does not
make such rules necessary or of interest.  We feel, however, that
this particular choice of generalized \LR\ filling is the ``right"
one, insofar as, on the basis of the work to follow, the L${\mathbb
R}$-fillings we have defined here {\em actually occur} as invariants
for matrix pairs over valuation rings (with an ${\mathbb R}$
valuation) in a way that directly generalizes previous results for
pairs over discrete valuation rings.

This interpretation was anticipated by the results of \cite{pak},
who established simple linear bijections between classical \LR\
fillings and non-negative integer valued hives (an alternate
combinatorial object of much recent interest). However, hives are
definable over ${\mathbb R}$ generally, so the inverse of their
linear map need not have been restricted to integer-valued hives.
When their maps are applied to the set of all hives, the image is
our L${\mathbb R}$-fillings.  The requirement that $k_{ij} \geq 0$
when $i < j$ for interior parts comes out as a natural requirement
imposed by the so-called ``rhombus inequalities" used in the
definition of hives. The requirement that $k_{ii} \geq 0$ for the
edge parts of classical \LR\ fillings is the result of an imposed
boundary condition on the hive, which restricts the partitions
$\mu$, $\nu$, and $\lambda$ to be non-negative and integer-valued.
Dropping this constraint, we obtain the L${\mathbb R}$-fillings
defined above.  There are, in fact, numerous other fruitful
connections between hives, L${\mathbb R}$-fillings, and matrix pairs
over valuation rings, obtainable by utilizing concepts that arise in
one of these three contexts, and then translating them to the other
two.  This work will be forthcoming \cite{hive-symm}.

So, while the linear map in \cite{pak} could define a proposed
${\mathbb R}$-valued \LR\ filling, it would not establish that such
fillings are realized as invariants for actual mathematical objects.
Correcting this omission is the work of the rest of this paper.

\section{Matrices over ${\mathbb R}$-Valuation Rings}

\begin{df}
Let ${\cal F}$ denote the field of formal ${\mathbb R}$-valuated
Laurent series with sparse exponents.  That is, ${\cal F}$ is the
field of formal series in the variable $t$ such that:
\begin{enumerate}
\item Every element $a \in F$
may be written $a=\sum_{i=0}^{\infty}a_{i}t^{\alpha_{i}}$, where
$a_{i}$ is in a fixed field of characteristic zero, and $\alpha_{i}
\in {\mathbb R}$.
\item For a given $a \in {\cal F}$, $a = \sum_{i=0}^{\infty}a_{i}t^{\alpha_{i}}$,
the sequence of exponents $(\alpha_{0}, \alpha_{1}, \dots)$, is
strictly increasing, and, in particular, bounded below. \item For a
given $a \in {\cal F}$, $a=\sum_{i=0}^{\infty}a_{i}t^{\alpha_{i}}$,
the set of exponents $\{ \alpha_{i} : i \in {\mathbb N} \}$ has no
limit points. In particular, any finite interval of the real line
may contain only finitely many non-zero exponents $\alpha_i$
appearing in any element $a \in {\cal F}$.
\end{enumerate}
\end{df}

These hypotheses on ${\cal F}$ ensure that defining multiplication
of elements of ${\cal F}$ via
\[ a \cdot b = \left(\sum_{i=0}^{\infty}a_{i}t^{\alpha_{i}}\right) \cdot \left(\sum_{j=0}^{\infty}
b_{j}t^{\beta_{j}} \right) = \sum_{k=0}^{\infty}
c_{k}t^{\kappa_{k}},
\] where $\alpha_{0} + \beta_{0} = \kappa_{0} < \kappa_1 < \cdots$,
and
\[ c_{k} = \sum_{\alpha_{i} + \beta_{j} =
\kappa_{k}}a_{i} b_{j}, \] is well-defined since the sparseness
condition on the exponents guarantees that there are only finitely
many solutions to the equation $\alpha_{i} + \beta_{j} = \kappa_{k}$
among the exponents of $\alpha_{i}$ and $\beta_{j}$.

The field ${\cal F}$ comes with a natural valuation, or ``norm",
over ${\mathbb R}$ by defining, for all $a \in {\cal F}$, $a \neq
0$:
\[ \|a \| = \alpha_{0}, \quad \hbox{if} \quad a =
\sum_{i=0}^{\infty}a_{i}t^{\alpha_{i}}. \]
 This norm clearly satisfies:
 \[ \| a \cdot b \| = \| a \| + \| b \|, \quad
 \| a \|  +  \| b \| \leq \| a + b \|, \]
and, in particular, $\| t^s \cdot a \| = s + \| a \|$.  We shall
sometimes refer to the value $\|a \|$ for $a \in {\cal F}$ as the
{\em order} of $a$.

 With this definition of ${\cal F}$, we may now define the subring $R \subseteq {\cal F}$ by
 \[ R = \{a \in {\cal F} : \| a \| \geq 0 \}. \]

Let $\| a \| = \alpha.$  Then we may write $a$ uniquely in the form:
\[ a = t^{\alpha} \cdot u, \quad \hbox{where} \quad \| u \| = 0.
\]
In particular, $\| u \| =  0$ implies $u$ is a {\em unit} in $R$
(and not just in ${\cal F}$). The only issue in computing the
multiplicative inverse of a unit $u =
\sum_{i=0}^{\infty}\sigma_{i}t^{s_i}$ is to note that if we formally
set
\[ u^{-1} = \sum_{i=0}^{\infty} w_i t^{v_i}, \]
and try to solve for the coefficients $w_{i}$, we need only first
require $v_i= s_i$ (that is, the exponents appearing in the
expansions of $u$ and $u^{-1}$ are the same). With this proviso, one
can recursively solve for the coefficients $w_i$ as in the case of
formal power series. Let $R^{\times}$ denote the units of $R$.
 If $a = t^{\alpha} \cdot u$, for $u \in
 R^{\times},$ then $a^{-1} = t^{-\alpha}\cdot u^{-1}$.

Let $M_{r}({\cal F})$ denote the set of $r \times r$ invertible
matrices over ${\cal F}$, and let $GL_{r}(R)$ denote the group of
invertible $r \times r$ matrices over $R$ (that is, matrices with an
inverse over $R$). Since matrix equivalence over the field $F$ is
not very interesting, we shall view the columns (or sometimes rows)
of a matrix $M \in M_{r}({\cal F})$ as spanning a module over $R$,
and will extend the classical theory of invariant factors to this
case, at least for finitely generated modules (finitely generated
$R$-submodules of ${\cal F}$ are principal, so for such modules the
classical theory applies). Indeed, most of what we do here could
probably be formulated in the context of ${\mathbb R}$-valuated
lattices, extending parts of the theory of buildings over discrete
valuation rings.  It will be enough for our purposes to work over
matrices, however, and we will be able to develop explicit
computational formulas in this context.

All we need to establish at this point is the existence of invariant
factors, viewed as $GL_{r}(R)$ invariants of full-rank matrices over
${\cal F}$.  Indeed, using only elementary row and column operations
over $R$, acting on a matrix $M \in M_{r}({\cal F})$, the following
theorem is easy to prove:

\begin{thm}  Let $M \in M_{r}({\cal F})$.  Then there exist $P,Q \in
GL_{r}(R)$ such that
\[ PMQ^{-1} = diag(t^{\mu_1}, t^{\mu_2}, \dots ,
t^{\mu_r}), \] where $\mu_1 \geq \mu_2 \geq \dots \geq \mu_r$.  The
exponents $\mu = (\mu_1, \mu_2, \dots , \mu_r)$ are uniquely
determined by $M$ and are an invariant of the orbit under the action
of matrix equivalence.\label{inv-fac}
\end{thm}

\begin{df}
Given $M \in M_{r}({\cal F})$, if there are $P,Q \in GL_{r}(R)$ such
that \[ PMQ^{-1} = D_{\mu} = diag(t^{\mu_1}, t^{\mu_2}, \dots ,
t^{\mu_r}), \] where $\mu_1 \geq \mu_2 \geq \dots \geq \mu_r$, we
shall call the ${\mathbb R}$-partition $\mu = (\mu_1, \mu_2, \dots,
\mu_r)$ the {\em invariant parition} of $M$, and denote this by
\[ inv(M) = \mu = (\mu_1, \mu_2, \dots,
\mu_r). \]
\end{df}

When $R$ is a principal ideal domain, the invariant factors of a
matrix are elements of $R$ with certain divisibility relations.  In
the case of a {\em discrete} valuation ring, such elements may be
written in the form $t^{\mu_{i}}$, where $\mu_{i}$ is a non-negative
integer.  Over the field ${\cal F}$, of course, we have the $\mu_{i}
\in {\mathbb R}$.

 Much of the classical theory of invariant factors, in the
 case of a discrete valuation ring, goes over to the ${\mathbb
 R}$-valuation ring presented here.  In~\cite{carlson-sa} the theory
 of {\em $s$-spaces} was developed to present a unified development
 of the theory of invariant factors over various rings, and also
 eigenvalues of hermitian matrices.  Let ${\cal F}^r$ be viewed as a free, finitely
 generated module over $R$.  Let us
 set \[ {\cal S}_{i} = \{ \hbox{finitely generated $R$-submodules of
 ${\cal F}^r$ of rank $i$} \}. \]
Given a finitely generated $R$-submodule ${\cal M}$ of ${\cal
 F}^r$, generated by $\{ a_{1}, \dots, a_k \},$ (we may assume this set is linearly
 independent over $R$) then let us define, given some
 $x \in {\cal M}$,
 \[ \| x \|_{\cal M} = \min_{1 \leq s \leq k} \| x_{i} \|, \quad \hbox{where}
 \quad x = x_{1}a_{1} + \dots + x_{k} a_{k}. \]
 Finally, given some $R$-module homomorphism $\phi : {\cal F}^r
 \rightarrow {\cal F}^r$, let us define
 \[ \psi : {\cal F}^r/ \{ 0 \} \rightarrow {\mathbb R}, \]
 by
 \[ \psi(x) = \frac{ \| \phi(x) \|_{\phi({\cal F}^r)} }{\| x \|_{{\cal F}^r}}. \]
where $\phi({\cal F}^r)$ is the $R$-submodule formed by the image of
${\cal F}^r$ under the map $\phi$.  It is then easily checked that
these definitions allow us to give
 $({\cal F}^r, \{ {\cal S}_{i} \} )$ an $s$-space structure such
 that $\psi$ is, as defined in~\cite{carlson-sa},
 $\ell$-diagonalizable (that is, the definition of $\psi$ allows one to define a
 generalized Rayleigh quotient for finitely generated submodules).  Once this has been determined, the
 following is obtained as a consequence:

\begin{thm}[Interlacing] Let $M \in M_{r}({\cal F})$, and let $H$ be a
square submatrix of $M$ of size $s$.  If the invariant partition of
$M$ is $inv(M) = \mu = (\mu_1, \mu_2, \dots , \mu_r)$, and that of
$H$ is $inv(H) = \sigma = (\sigma_1, \dots , \sigma_s)$, then
\[ \mu_{i} \geq \sigma_i \geq \mu_{i +r-s}, \quad 1 \leq i \leq s.
\] \label{interleaving}
\end{thm}

In particular, if in the above Theorem, $H$ is an $(r-1) \times
(r-1)$ submatrix of $M$, then
\[ \mu_1 \geq \sigma_1 \geq \mu_2 \geq \sigma_2 \geq \dots \geq
\mu_{r-1} \geq \sigma_{r-1} \geq \mu_r. \]

 We will begin by recording some definitions, notation, and
preliminary lemmas that will be used throughout the paper.

\begin{df}  Let $\mu$ and $\nu$ be ${\mathbb R}$-partitions.
Define
\[ D_{\mu} = diag(t^{\mu_1}, t^{\mu_2}, \dots, t^{\mu_r}), \]
that is, the diagonal matrix with the entries $t^{\mu_1}, t^{\mu_2},
\dots, t^{\mu_r}$ on the diagonal.

Similarly, set
\[ D_{\widehat{\nu}} = diag(t^{\nu_r}, t^{\nu_{(r-1)}}, \dots, t^{\nu_1}). \]
\end{df}

\begin{df}
Given two pairs of matrices $(M,N), (A,B) \in M_{r}({\cal F})^2$, we
say a $(M,N)$ is {\em pair equivalent} to $(A,B)$ if there are
matrices $P,Q,T \in GL_{r}(R)$ such that
\[ (A,B) = (P,Q,T) \cdot (M,N) = (PMQ^{-1}, QNT^{-1}). \]
\end{df}

That is, though the pairs $(M,N)$ are defined over the field ${\cal
F}$, we will restrict the conjugations to invertible matrices over
the subring $R$.  If $(M,N)$ is pair equivalent to $(A,B)$, then
$inv(M)=inv(A)$, $inv(N) = inv(B)$, and also $inv(MN) = inv(AB)$.

   Note that by Theorem~\ref{inv-fac} there are
$P,Q \in GL_{r}(R)$ such that
\[ (P,Q, 1) \cdot (M,N) = (PMQ^{-1}, QN) = (D_{\mu}, N'), \]
where $\mu = inv(M).$  Thus, when calculating invariants of the
orbit of a matrix pair, we may assume that one entry in the pair is
diagonalized.  We shall look to put $N'$ in a form from which a
L${\mathbb R}$ filling may be determined.  In order to accomplish
this we will restrict our paired conjugations to those of the
following form:

\begin{df}  Let $\mu$ be an ${\mathbb R}$-partition.  We will say that
a square matrix $Q \in GL_{r}(R)$ is {\em $\mu$-admissible} iff
\[ D_{\mu}Q D_{\mu}^{-1} \in GL_{r}(R). \]  Note that if $Q$ is
$\mu$-admissible, and $q_{ij}$ is the $(i,j)$-entry of $Q$, then
\[ \| q_{ij} \| \geq \mu_{j} - \mu_{i}. \]
\end{df}

This condition only imposes constraints on entries below the
diagonal, where $q_{ij}$ will have to have a sufficiently large
order.  This condition is not only necessary, but, as is easy to
check, sufficient.

\begin{lem} Suppose $Q \in GL_{r}(R)$.  Suppose, given some ${\mathbb R}$-partition
$\mu$, the $(i,j)$ entry of $Q$ is $q_{ij}$, and that we have
\[ \| q_{ij} \| \geq \mu_{j} - \mu_{i}. \]
Then $Q$ is $\mu$-admissible.
\end{lem}
In particular, any upper-triangular matrix is $\mu$-admissible, for
any ${\mathbb R}$-partition $\mu$.

Given a matrix pair $(D_{\mu}, N')$, if $Q$ is $\mu$-admissible,
then $D_{\mu}QD_{\mu}^{-1}$ is invertible, so $( D_{\mu}Q
D_{\mu}^{-1}, Q, T)$ is a triple of invertible matrices which may
act on $(D_{\mu}, N')$.  We obtain:
\[ ( D_{\mu}Q D_{\mu}^{-1}, Q, T) \cdot (D_{\mu}, N') =
(D_{\mu}Q D_{\mu}^{-1} \cdot D_{\mu} \cdot Q^{-1}, QN' T^{-1}) =
(D_{\mu}, QN' T^{-1}).
\]

That is, the action above using a $\mu$-admissible matrix fixes the
first term $D_{\mu}$.  Thus, in order to compute invariants of the
orbit of a matrix pair $(M,N)$, it is sufficient to reduce the
problem to the orbit of pairs $(D_{\mu}, N')$ where $D_{\mu}$ is
fixed.  That is, we need only study the orbit of the action on the
right term $N'$:
\[ (Q,T) \cdot N' = QN'T^{-1} \]
where $Q$ is a $\mu$-admissible matrix.

Our plan is as follows. Since, given an arbitrary matrix pair $(M,N)
\in M_{r}({\cal F})^2$ we may find a pair $(D_{\mu}, N')$, in the
orbit of $(M,N)$, we will then find, in the orbit of $N'$ under the
$\mu$-admissible action above, a special matrix form $N^*$ called
{\em $\mu$-generic}, from which we may determine an ${\mathbb
R}$-valued \LR\ filling. We shall provide a formula, based on the
orders of determinants of the $\mu$-generic matrix $N^*$ to compute
the filling.  In order to make these definitions precise, we will
require the following notation and definitions:

\begin{df}
 Let $I$, $J$, and $H$ be subsets of $ \{ 1,2,\ldots, r \}$ of length
$k$, written as $I = (i_{1}, i_{2}, \ldots , i_{k})$,  where $1 \leq
i_{1} < i_{2} < \cdots < i_{k} \leq r$, and similarly for $J$ and
$H$. We call such sets {\em index sets}.  (Note: $I,J$, and $H$ do
{\em not} denote partitions.) Let $I \subseteq H$ denote the
condition that $i_{s} \leq h_{s}$ for $1 \leq s \leq k$. Given an $r
\times r$ matrix $W$, let \label{det-def}
\[ W_{IJ}=W\left( \begin{array}{cccc} i_{i} & i_{2} & \cdots & i_{k} \\
j_{1} & j_{2} & \cdots & j_{k} \end{array} \right) \] denote the $k
\times k$ minor of $W$ using rows $I$ and columns $J$ (that is, the
determinant of this submatrix).  \label{mat-def} Let us extend the
definition of $\|a\|$ to square matrices, so that if $B$ is any
square matrix, $\| B\|$ will denote $\| \det(B) \|$. Also, given an
${\mathbb R}$-partition $\mu = (\mu_{1}, \ldots , \mu_{r})$, let
$\mu_{I}$ denote the ${\mathbb R}$-partition $\mu_{I}=(\mu_{i_{1}},
\mu_{i_{2}}, \ldots, \mu_{i_{k}})$, and let $|\mu_{I}|= \mu_{i_{1}}
+ \mu_{i_{2}} + \cdots + \mu_{i_{k}}$.
\end{df}

In \cite{lrcon} the theory of $\mu$-generic matrices was presented
in the context of discrete valuation rings.  While most of what was
developed there works for ${\mathbb R}$-valuation rings with few, if
any, changes, we will need more general results in later sections.
So, we present here the following technical lemma:

\begin{lem} (a) Given any lower-triangular matrix $Q_{L} \in
GL_{r}(R)$, there exists an upper-triangular matrix $Q_{U} \in
GL_{r}(R)$ such that we may find a lower-triangular
$\widehat{Q_{L}}$ and an upper-triangular $\widehat{Q_{U}}$ such
that
\[ Q_{U} Q_{L} = \widehat{Q_{L}}\widehat{Q_{U}} = Q. \]

\medskip

(b) Given any upper-triangular matrix $\widehat{Q_{U}} \in
GL_{r}(R)$ there exists a lower-triangular matrix $\widehat{Q_{L}}
\in GL_{r}(R)$ such that we may find an upper-triangular $Q_{U}$ and
a lower-triangular $Q_{L}$ such that
\[ \widehat{Q_{L}}\widehat{Q_{U}} =  Q_{U} Q_{L} =Q. \]
We may choose $\widehat{Q_{L}}$ to be $\mu$-admissible, in which
case the product $Q$ will be $\mu$-admissible as well.

\medskip

(c) Suppose that $S \in M_{r}({\cal F})$ of full rank is given.
Then, we may choose $Q \in GL_{r}(R)$, obtained from either (a) or
(b) above, so that, for any index sets $I,H,J$ of length $k$, $1
\leq k \leq r$, such that $I \subseteq H \subseteq J$, we have
\begin{equation} \| (QS)_{IJ} \| \leq \| (QS)_{HJ} \|, \label{row} \end{equation}
 and, in the case
that $Q$ is $\mu$-admissible, \begin{equation}
 \| (QS)_{HJ} \| \leq
\| (QS)_{IJ} \| + | \mu_{I} | - | \mu_{H} |. \label{mu-gap}
\end{equation} \label{row-ineq}
\end{lem}
\begin{proof}
We begin with (a).  Suppose we are given a lower-triangular $Q_{L}
\in GL_{r}(R)$.  For any upper-triangular $Q_{U}$, the necessary
condition that the product
\[ Q= Q_{U} Q_{L} \]
possess a ```$UL$" decomposition (with entries over ${\cal F}$, the
quotient field of the ring $R$) of the form
$\widehat{Q_{L}}\widehat{Q_{U}}$ is that for all principal minors
$(Q_{U}Q_{L})_{J_{k}J_{k}}$, we have \[ (Q_{U}Q_{L})_{J_{k}J_{k}}
\neq 0,
\] where $J_{k}$ is an index set of the form $J_{k} = (1, \dots, k)$, for
each $k$, $1 \leq k \leq r$ (see~\cite{Gant}, pp.\ 35-36). We can,
in fact, choose $Q_{U}$ so that not only is this condition met, but
so that both matrices $\widehat{Q_{L}}$ and $\widehat{Q_{U}}$ are in
$GL_{r}(R)$. This is accomplished, again by appeal to the result
in~\cite{Gant}, by noting that the entries in $\widehat{Q_{L}}$ and
$\widehat{Q_{U}}$ are quotients of determinants in the product
$Q=Q_{U}Q_{L}$, whose denominators are principal minors of $Q$.
Choosing $Q_{U}$ so that the principal minors are not only non-zero,
but are actually units in $R$ themselves (that is, so that $ \|
(Q_{U}Q_{L})_{J_{k}J_{k}} \| = 0$) amounts to choosing $Q_{U}$ so
that the images in the residue field of ${\cal F}$ of entries of
$Q_{U}$ lie {\em outside} a variety defined by certain polynomials
in the images in the residue field of entries $Q_{U}$.  Thus, the
existence of $\widehat{Q_{L}}$ and $\widehat{Q_{U}}$ is proved.  The
matrix $\widehat{Q_{U}}$ is, like any upper-triangular matrix,
necessarily $\mu$-admissible, and since the product $Q_{U}Q_{L}$ is
$\mu$-admissible, so must be $\widehat{Q_{L}}$.

Case (b) is proved in essentially the same way.  Given any matrix $Q
= \widehat{Q_{L}}\widehat{Q_{U}}$, the necessary condition that an
atypical ``UL" decomposition exist and be defined over $R$ is only
that the principal minors be units in $R$ (this statement requires
only certain formal modifications of the proof of the ``$LU$"
decomposition in~\cite{Gant}).  We may, in fact, choose
$\widehat{Q_{L}}$ to be $\mu$-admissible, since the existence of the
decomposition will be implied by an appropriate choice of units on
the diagonal of $\widehat{Q_{L}}$, which are unaffected by the
requirement of $\mu$-admissibility.

The upshot of this is that given matrices $Q_{L}$ or
$\widehat{Q_{U}}$, which we may assume to be $\mu$-admissible if we
choose, we can multiply either on the left so that the resulting
product ($Q_{U}Q_{L}$ or $\widehat{Q_{L}}\widehat{Q_{U}}$) possess
both a ``$LU$" and a ``$UL$" decomposition, and such that the
product is is $\mu$-admissible, if we choose. To prove (c), then,
let us choose a matrix $S \in M_{r}(R)$ of full rank, and let us fix
a choice of index sets $I,H,J$ of length $k$, where $1 \leq k \leq
r$. For each such choice, we will argue that whenever $I \subseteq H
\subseteq J$, implying Inequalities~\ref{row} and~\ref{mu-gap}
amounts to choosing $Q$, possessing ``$UL$" and ``$LU$"
decompositions, to be sufficiently generic so that the image in the
residue field of its entries lies outside of a variety defined by
the images of entries of $S$.  In particular, we may ensure we can
choose $Q$ so that $Q = Q_{U}Q_{L}$, and the upper-triangular
$Q_{U}$ satisfies
\[ \| (Q_{U})_{IW} \| = 0,\]
whenever $I \subseteq W$ (that is, $I \subseteq W$ will imply the
minor is a {\em unit} in $R$). Note that
\[ I \nsubseteq W \quad \hbox{implies} \quad (Q_{U})_{IW}=0, \]
since $Q_{U}$ is upper-triangular.

Let us begin with Inequality~\ref{row}.  We have, by the
Cauchy-Binet formula:
\begin{align*} \| (QS)_{IJ} \| &=  \| (Q_{U}Q_{L}S)_{IJ} \| \\
& = \left\| \sum_{I \subset W}(Q_{U})_{IW}(Q_{L}S)_{WJ} \right\| \\
\intertext{By a generic choice of $Q_{U}$, there can be no
``catastrophic cancelation" in the sum, so the above must be:}
 \| (QS)_{IJ} \| &
= \left\| \min_{I \subseteq W}(Q_{U})_{IW}(Q_{L}S)_{WJ} \right\| \\
& =\left\| \min_{I \subseteq W}(Q_{L}S)_{WJ} \right\| \\
& \leq \left\| \min_{H \subseteq W}(Q_{L}S)_{WJ} \right\| \\
& = \left\| \min_{H \subseteq W}(Q_{U})_{HW}(Q_{L}S)_{WJ} \right\| \\
& = \| (QS)_{HJ} \|.
\end{align*}
So Inequality~\ref{row} is proved.  Inequality~\ref{mu-gap} is
proved similarly. First note that, given a $\mu$-admissible
$Q=\widehat{Q_{L}}\widehat{Q_{U}}$, that $\widehat{Q_{L}}$ is
$\mu$-admissible as well, so that
\[ D_{\mu}\widehat{Q_{L}}D_{\mu}^{-1} = \widehat{Q_{L}}_{0} \in
GL_{r}(R). \] This implies
\[\widehat{Q_{L}} =D_{\mu}^{-1}\widehat{Q_{L}}_{0}D_{\mu}. \]
We may choose $\widehat{Q_{L}}$, and hence $\widehat{Q_{L}}_{0}$, to
be sufficiently generic so that for all index sets $I$ and $W$ of
length $k$, such that whenever $W \subseteq H$, we have
\[ \| (\widehat{Q_{L}}_{0})_{HW} \| = 0. \]
But then
\[ \| (\widehat{Q_{L}})_{HW} \| =
\|(D_{\mu}^{-1}\widehat{Q_{L}}_{0}D_{\mu})_{HW} \| =
\|(D_{\mu}^{-1})_{HH}(\widehat{Q_{L}}_{0})_{HW}(D_{\mu})_{WW} \|  =
|\mu_{W} | - | \mu_{H} |.
\]

But then, again using a sufficiently generic choice of
$\widehat{Q_{L}}$ so there is no cancelation among sums of
determinants:
\begin{align*}\| (QS)_{HJ} \| & = \|
(\widehat{Q_{L}}\widehat{Q_{U}}S)_{HJ} \| \\
& = \left\| \sum_{W \subseteq H}(\widehat{Q_{L}})_{HW}(\widehat{Q_{U}}S)_{WJ} \right\| \\
& = \left\| \min_{W \subseteq H}(\widehat{Q_{L}})_{HW}(\widehat{Q_{U}}S)_{WJ} \right\| \\
& =\left\| \min_{W \subseteq H}(\widehat{Q_{L}}_{0})_{HW}(\widehat{Q_{U}}S)_{WJ} \right\| + |\mu_{W} | - | \mu_{H} |\\
& =\left\| \min_{W \subseteq H}(\widehat{Q_{U}}S)_{WJ} \right\| + |\mu_{W} | - | \mu_{H} |\\
& \leq \left\| \min_{W \subseteq I}(\widehat{Q_{U}}S)_{WJ} \right\|
+ |\mu_{W} | - | \mu_{H} |\\
& =\left( \left\| \min_{W \subseteq I}(\widehat{Q_{U}}S)_{WJ}
\right\|
+ |\mu_{W} | - | \mu_{I}| \right) +  |\mu_{I} | - | \mu_{H}|   \\
& = \|(QS)_{IJ} \|  +  |\mu_{I} | - | \mu_{H}| .  \\
\end{align*}
So the inequality is proved.
\end{proof}

Given an ${\mathbb R}$-partition $\nu= (\nu_1, \dots , \nu_r)$,
recall the notation
\[ \widehat{\nu} = (\widehat{\nu_1},\widehat{\nu_2},\dots, \widehat{\nu_r}) =(\nu_{r}, \nu_{r-1}, \dots, \nu_2, \nu_1), \]
and
\[ D_{\widehat{\nu}} = diag(t^{\nu_{r}}, t^{\nu_{r-1}}, \dots ,
t^{\nu_{1}}). \] If $J=(j_1, j_2, \dots , j_k)$ is an index set of
length $k \leq r$, then we will adopt the notation \[
\widehat{\nu}_{J} = (\widehat{\nu}_{j_{1}}, \widehat{\nu}_{j_{2}},
\dots, \widehat{\nu}_{j_{k}} ) = ( \nu_{r-j_{1}+1}, \nu_{r-j_{2}+1},
\dots , \nu_{r-j_{k}+1}). \]

\begin{df}
Given an ${\mathbb R}$-partition $\nu$, we will say a matrix $T \in
GL_{r}(R)$ is $\widehat{\nu}$-admissible iff
\[D_{\widehat{\nu}}^{-1} T D_{\widehat{\nu}} \in GL_{r}(R). \]
Note that if $T$ is $\widehat{\nu}$-admissible, and $\tau_{ij}$ is
the $(i,j)$-entry of $T$, then whenever $i \geq j$,
\[ \| \tau_{ij} \| \geq \widehat{\nu}_{i} - \widehat{\nu}_{j} =
\nu_{r-i+1} - \nu_{r-j+1}. \]
\end{df}

As before, it is elementary to show that, for invertible matrices,
the above necessary condition for $\widehat{\nu}$-admissibility is
also sufficient.

By the same reasoning as the lemma above, we prove the following:
\begin{cor} (a) Given any lower-triangular matrix $T_{L} \in
GL_{r}(R)$, there exists an upper-triangular matrix $T_{U} \in
GL_{r}(R)$ such that we may find a lower-triangular
$\widehat{T_{L}}$ and an upper-triangular $\widehat{T_{U}}$ such
that
\[ T_{L}T_{U} =\widehat{T_{U}}  \widehat{T_{L}}= T. \]

\medskip

(b) Given any upper-triangular matrix $\widehat{T_{U}} \in
GL_{r}(R)$ there exists a lower-triangular matrix $\widehat{T_{L}}
\in GL_{r}(R)$ such that we may find an upper-triangular $T_{U}$ and
a lower-triangular $T_{L}$ such that
\[ \widehat{T_{U}}\widehat{T_{L}} =  T_{L}T_{U}  =T. \]
We may choose $\widehat{T_{L}}$ to be $\widehat{\nu}$-admissible, in
which case the product $T$ will be $\widehat{\nu}$-admissible as
well.

\medskip

(c) Suppose $T \in GL_{r}(R)$ is  any matrix obtained from either
(a) or (b) above, and that $S \in M_{r}({\cal F})$ of full rank is
given.  Then, we may choose $T$ so that, for any index sets $I,H,J$
of length $k$ such that $I \subseteq H \subseteq J$, we have
\begin{equation} \| (ST)_{IJ} \| \leq \| (ST)_{IH} \|, \label{col} \end{equation}
 and, in the case
that $T$ is $\widehat{\nu}$-admissible, \begin{equation}
 \| (ST)_{IH} \| \leq
\| (ST)_{IJ} \| + | \widehat{\nu}_{J} | - | \widehat{\nu}_{H} |.
\label{nu-gap}
\end{equation}
\label{col-ineq}
\end{cor}

\begin{df}
 Let us call a matrix pair $(D_{\mu}, N^*)$ a {\em $\mu$-generic}
 matrix pair associated to $N \in M_{r}({\cal F})$ with respect to a
partition $\mu$ if $N^*$ is upper-triangular, and we can factor
$N^*$ as
\[ N^* = QNT^{-1},
\] where $Q= Q_{U}Q_{L} = \widehat{Q_{L}} \widehat{Q_{U}}$ is
$\mu$-admissible, $ Q_{U},Q_{L}$ are upper and lower triangular,
respectively, and $Q$ and $T^{-1}$ satisfy
Inequalities~\ref{row},~\ref{mu-gap} of Lemma~\ref{row-ineq} and
Inequality~\ref{col} of Corollary~\ref{col-ineq}.
  We shall simply say $N^*$ is {\em $\mu$-generic} if $N^*$ is a
$\mu$-generic matrix associated to some $N \in M_{r}({\cal F})$.
\end{df}


The proof of the corollary below may be found in \cite{lrcon}.
There, the results are stated for the case that $R$ denotes a
discrete valuation ring whose residue field is characteristic zero.
The results are based, however, on elementary matrix calculations
and carry over to our setting without alteration.

\begin{cor}  Suppose $N^*$ is a $\mu$-generic matrix such that
$inv\,(N^*)=(\nu_{1} \geq \nu_{2} \geq \dots \geq \nu_{r})$.  Then,
if $I_{s} = (1,2, \dots, s)$, $H_{(r-s)} = ((r-s+1), (r-s+2), \dots,
r)$, we have
\[ \| N_{I_{s}H_{(r-s)}}^* \| =
\nu_{r-s+1}+ \nu_{r-s+2} + \dots + \nu_{r} \quad \hbox{and} \quad \|
(D_{\mu}N^*)_{H_{(r-s)}H_{(r-s)}} \| =
\lambda_{r-s+1}+\lambda_{r-s+2}+\dots + \lambda_r,
\] where $\lambda=(\lambda_{1}, \dots, \lambda_r) = inv\,(D_{\mu}N^*)$.\label{row-inv}
\end{cor}

\begin{df}
Suppose that $(D_{\mu},N^*)$ is a fixed $\mu$-generic pair in the
orbit of the given pair $(M,N) \in {\cal M}_{r}({\cal F})^2$, both
of full rank. Let the symbols
\[\left\| \Bigl( i_{1}, \ldots , i_s \Bigr) \right\|, \quad \left\| \Bigl( (i_{1})^{\wedge},
(i_2)^{\wedge}, \ldots , (i_k)^{\wedge}\Bigr) \right\|,
 \] denote, respectively, the order of the minor
 of the $\mu$-generic matrix $N^*$ with rows $i_1, \dots, i_s$, and
 the right-most distinct columns possible.
Secondly, when using the ``\,$^{\wedge}$" symbol, the order of the
minor of $N^*$ whose rows include all rows $1$ through $r$ but with
the rows $i_1, i_2, \ldots, i_k$ omitted, again using the right-most
columns resulting in a square submatrix.
\end{df}

 We omit the dependence of the above notation on the fixed $\mu$-generic
matrix $N^*$.  Our central matrix result is the following
determinantal formula that calculates a L${\mathbb R}$ filling from
a matrix pair.  Most of the proof of this result follows exactly the
argument presented in \cite{lrcon} in the case of a discrete
valuation ring. We will sketch the proof and mention the few cases
where passage to the ${\mathbb R}$-valuation introduces changes.

\begin{thm}
Let $(M,N) \in M_{r}({\cal F})^2$ be a matrix pair, with both
matrices full rank.  Suppose that $N^*$ is a $\mu$-generic matrix
associated to $N$. Let us define a triangular array of integers $\{
k_{ij} \}$, for $1 \leq i \leq r$, and $i \leq j \leq r$, by
declaring
\begin{equation}
k_{1j} + k_{2j} + \dots + k_{ij} =\left\| \Bigl( (j-i)^{\wedge},
 \ldots ,(j-1)^{\wedge} \Bigr) \right\| - \left\| \Bigl(
(j-i+1)^{\wedge}, \ldots ,(j)^{\wedge} \Bigr)
\right\| .\\
\label{seq-def}
\end{equation}
Then, the set $F = \{ k_{ij} : 1 \leq i \leq r, \ i \leq j \leq r
\}$ is an L${\mathbb R}$ filling in L${\mathbb R}(\mu,\nu,
\lambda)$, where $inv\,(M)=\mu$, $inv\,(N) = \nu$, and $inv\,(MN)=
\lambda$. \label{kij-lr}
\end{thm}
\begin{proof}[{\em sketch}]
 Since our notation for omitted indices in determinants is only to be
used when removing a non-empty increasing sequence of indices, we
will adopt the convention that
\[ \| (p)^{\wedge} , \dots, (q)^{\wedge} \| = \| 1, 2, \dots, r \| \qquad \hbox{if $p>q$.} \]  With this, we can use Equation~\ref{seq-def} above to
define the individual entries $k_{ij}$, according to the formula

\begin{multline}
k_{ij} =\left\| \Bigl( (j-i)^{\wedge}, (j-i+1)^{\wedge}, \ldots ,
(j-1)^{\wedge} \Bigr) \right\| - \left\| \Bigl( (j-i+1)^{\wedge},
\ldots ,(j)^{\wedge} \Bigr)
\right\| \\
- \Biggl( \left\| \Bigl( (j-i+1)^{\wedge}, (j-i+2)^{\wedge}, \ldots
, (j-1)^{\wedge} \Bigr) \right\| - \left\| \Bigl( (j-i+2)^{\wedge},
\ldots , (j)^{\wedge} \Bigr) \right\| \Biggr).\label{kij-def}
\end{multline}

We need to verify that using Equation~\ref{kij-def} to define
$\{k_{ij} \}$ results in a set of real numbers satisfying the
conditions L${\mathbb R}1$, L${\mathbb R}2$, L${\mathbb R}3$, and
L${\mathbb R}4$.  Condition L${\mathbb R}1$ follows quickly from
Corollary~\ref{row-inv} above, and the others follow precisely as in
\cite{lrcon}.  However, it should be noted that in \cite{lrcon} it
was proved that given a $\mu$-generic matrix $N^*$, the
determinantal formulas above defined a \LR-filling associated to a
matrix pair $(M,N)$ over a {\em discrete} valuation ring. Thus, the
filling so constructed would satisfy $k_{ij} \geq 0$ for {\em all}
$i$ and $j$. However, in the case of a ring with ${\mathbb
R}$-valuation, the definition of L${\mathbb R}$-fillings only
requires $k_{ij} \geq 0$ for $i < j$. Let us briefly indicate how to
account for the difference one finds in working over an ${\mathbb
R}$-valuated ring.
 When computing $k_{ij}$, for $i < j$, we find
the matrix argument found in \cite{lrcon} goes through without any
change. It shows that $k_{ij}$ is determined as the {\em difference}
between the orders of two invariant factors of a single submatrix
(the larger minus the smaller).  So, this value must be non-negative
in all cases.

In contrast to this, in the case $i=j$, Equation~\ref{kij-def} above
becomes:
\begin{multline}
k_{ii} =\left\| \Bigl( (i-i)^{\wedge}, (i-i+1)^{\wedge}, \ldots ,
(i-1)^{\wedge} \Bigr) \right\| - \left\| \Bigl( (i-i+1)^{\wedge},
\ldots ,(i)^{\wedge} \Bigr)
\right\| \\
- \Biggl( \left\| \Bigl( (i-i+1)^{\wedge}, (i-i+2)^{\wedge}, \ldots
, (i-1)^{\wedge} \Bigr) \right\| - \left\| \Bigl( (i-i+2)^{\wedge},
\ldots , (i)^{\wedge} \Bigr) \right\| \Biggr) \\
= \left\| \Bigl( (0)^{\wedge}, (1)^{\wedge}, \ldots , (i-1)^{\wedge}
\Bigr) \right\| - \left\| \Bigl( (1)^{\wedge}, \ldots ,(i)^{\wedge}
\Bigr)
\right\| \\
- \Biggl( \left\| \Bigl( (1)^{\wedge}, (2)^{\wedge}, \ldots ,
(i-1)^{\wedge} \Bigr) \right\| - \left\| \Bigl( (2)^{\wedge}, \ldots
, (i)^{\wedge} \Bigr) \right\| \Biggr).
\end{multline}
 In the above expression there is cancelation of some of the terms
 that does not occur when $i < j$.  Consequently, the above, when
 reduced and written positively, becomes:

\[ k_{ii} = \| \left( 1, (i+1) , \dots , r \right) \| - \| \left(
(i+1), \dots, r \right) \|. \]

 In the above, there
is no {\em a priori} reason that this value cannot be negative, when
working over a ring $R$ with an ${\mathbb R}$-valuation.

In general, if the ring $R$ only possesses a discrete valuation,
then the parts $k_{ij}$, for $i < j$ will necessarily  be
non-negative integers. If for any $i$ we had $k_{ii} < 0$, then by
the word condition L${\mathbb R}4$ we would have $\nu_r = k_{rr}
\leq k_{ii} < 0$, so there would necessarily be a part of the
partition $\nu$ that was negative. Thus, if the partitions $\mu$ and
$\nu$, which appear as the invariant partitions of the matrices $M$
and $N$, are comprised of non-negative integers, then all the
$k_{ii} \geq 0$. Lastly, since
\[ k_{ii} + k_{i, i+1} + \dots + k_{ir} = \mu_{i}, \]
then $k_{ii}$ must also be an integer.

The proofs that conditions L${\mathbb R}$3 and L${\mathbb R}$4 hold
follow exactly as in \cite{lrcon}.  Note that the arguments in
\cite{lrcon} require the interleaving condition among invariant
factors of submatrices, which is the content of
Theorem~\ref{interleaving}.
\end{proof}

The following corollary, proposition, and theorem appeared in
\cite{lrcon}, and the proofs there apply here without alteration.

\begin{cor} With $k_{ij}$ defined by Equation~\ref{kij-def} for
all $1 \leq i \leq r$ and $i \leq j \leq r$, we have
\begin{enumerate}
\item $\sum_{\beta = j}^{l}  (k_{1 \beta} + k_{2 \beta} + \cdots + k_{i \beta}) =
\| (j-i)^{\wedge} \dots (j-1)^{\wedge} \| - \| (l - i + 1)^{\wedge}
\dots (l)^{\wedge} \|,$ for $j \leq l$.
\item $k_{ii} + k_{i,(i+1)} + \cdots + k_{ij} = \| (j-i+2)^{\wedge} \dots (j)^{\wedge} \| - \| (j-i+1)^{\wedge} \dots (j)^{\wedge} \|$.
    \end{enumerate}\label{word-lem}
\end{cor}

\begin{prop}
Let $I$ and $J$ be index sets of length $k$, and let $N$ and
$\widehat{N}$ be $r \times r$ $\mu$-generic matrices.  Suppose there
exist $\mu$-admissible matrices $Q$ and $T$ such that
\[ QNT^{-1} = \widehat{N} . \]
Then \label{big-prop}
\[ \|N_{IJ}\| = \|\widehat{N}_{IJ}\|.\]
\end{prop}

\begin{thm}[Uniqueness]If $(M,N)$ is pair equivalent to to $(M', N')$, then the
L${\mathbb R}$-fillings determined by both pairs are the same.  That
is, pairs in the same $GL_{r}(R)^3$ orbit yield identical L${\mathbb
R}$-fillings.
\end{thm}

 Finally, given a L${\mathbb R}$-filling of $\lambda / \mu$ with content $\nu$,
 where $\mu, \nu$, and $\lambda$ are ${\mathbb R}$-partitions, the following result allows one to {\em construct} a
 matrix pair $(M,N)$ such that $inv(M) = \mu$, $inv(N) = \nu$, and
 $inv(MN) = \lambda$.  The proof of this, again in the context of
 discrete valuation rings and classical \LR-fillings, first
 appeared in~\cite{me}, but the proofs there made only formal use of
 the conditions L${\mathbb R}$3 and L${\mathbb R}$4, and they may be
 used in our case here without alteration.  A version of this result over discrete
valuation rings, though based on conjugate sequences and a very
different construction, had first been obtained in \cite{AS}.

\begin{thm}[\cite{me}]
Let $F=\{ k_{ij}: 1 \leq i \leq r, \, i \leq j \leq r \}$ be a
L${\mathbb R}$-filling of $\lambda / \mu$ with content $\nu$. Define
$r \times r$ matrices $M,N_{1},N_{2}, \ldots N_{r}$ over $R$ by
\label{my-thm}
\begin{enumerate}
\item $M = diag(t^{\mu_{1}},t^{\mu_{2}}, \ldots , t^{\mu_{r}})$.
\item Define the block matrix $N_{i}$ by
\[ N_{i} = \left[ \begin{array}{c|c}
1_{i-1} & 0 \\
\hline 0 & T_{i}
\end{array} \right] \]
where $T_{i}$ is the $(r-i+1) \times (r-i+1)$ matrix:
\[ T_{i} = \left[ \begin{array}{ccccc}
t^{k_{i,i}}    & 1    & 0    & \cdots    &       0  \\
0        & t^{k_{i,i+1}} & 1 &  \ddots   &       \vdots     \\
0 &      0& t^{k_{i,i+2}} & \ddots  &        0  \\
 \vdots & \vdots  & \ddots  &  \ddots    & 1 \\
0  &  0  & \cdots  &  0 &    t^{k_{i,r}} \\

\end{array} \right] \]
and $1_{i-1}$ is an $(i-1) \times (i-1)$ identity matrix.
\end{enumerate}
Set $N=N_1 N_2 \cdots N_r$.  Then $inv(M)=\mu$, $inv(N)=\nu$, and
$inv(MN) = \lambda$. \label{lr_to_mat}
\end{thm}

(Note that in \cite{me}, Theorem~\ref{my-thm} was written so that
invariants were calculated in {\em increasing} order, and so the
matrices used in the factorization have a slightly different form.)

\medskip

\section{Left-Right Bijections} We now give an alternate method for
computing a L${\mathbb R}$ filling from a matrix pair $(M,N) \in
M_{r}({\cal F})$ which is conceptually simpler, and is more
```symmetric" in that {\em one} special matrix form can be used to
obtain {\em both} the ``right filling" of $\lambda / \mu$ with
content $\nu$, and the ``left filling" of $\lambda / \nu$ with
content $\mu$.  This method will also exhibit properties that
replicate all the fundamental features of the combinatorial
bijection $L{\mathbb R}(\mu, \nu, \lambda) \leftrightarrow L{\mathbb
R}(\nu, \mu, \lambda)$ in the context of {\em integer-valued}
fillings (see \cite{kerber}). These features will allow us to prove
that every left filling of a matrix pair determines a unique right
filling that depends only on the filling, and not on the particular
matrix realization of it.

Given an ${\mathbb R}$-partition $\nu= (\nu_1, \dots , \nu_r)$,
recall the notation
\[ \widehat{\nu} = (\widehat{\nu_1},\widehat{\nu_2},\dots, \widehat{\nu_r}) =(\nu_{r}, \nu_{r-1}, \dots, \nu_2, \nu_1), \]
and
\[ D_{\widehat{\nu}} = diag(t^{\nu_{r}}, t^{\nu_{r-1}}, \dots ,
t^{\nu_{1}}). \] In particular, if $J=(j_1, j_2, \dots , j_k)$ is an
index set of length $k \leq r$, then we will adopt the notation \[
\widehat{\nu}_{J} = (\widehat{\nu}_{j_{1}}, \widehat{\nu}_{j_{2}},
\dots, \widehat{\nu}_{j_{k}} ) = ( \nu_{r-j_{1}+1}, \nu_{r-j_{2}+1},
\dots , \nu_{r-j_{k}+1}). \]

Recall also that, given an ${\mathbb R}$-partition $\nu$, we will
say a matrix $T \in GL_{r}(R)$ is $\widehat{\nu}$-admissible
whenever
\[D_{\widehat{\nu}}^{-1} T D_{\widehat{\nu}} \in GL_{r}(R). \]
Note that if $T$ is $\widehat{\nu}$-admissible, and $\tau_{ij}$ is
the $(i,j)$-entry of $T$, then whenever $i \geq j$,
\[ \| \tau_{ij} \| \geq \widehat{\nu}_{i} - \widehat{\nu}_{j} =
\nu_{r-i+1} - \nu_{r-j+1}. \]

We will describe an alternative method to calculate L${\mathbb
R}$-fillings from matrix pairs.  Let us assume we have a pair
$(D_{\mu},N)$, where $D_{\mu} = diag(t^{\mu_1}, \dots, t^{\mu_r})$.
Let us factor $N$ according to its invariant factors as:
\[ N = S D_{\widehat{\nu}} V^{-1}, \]
where $S$ and $V$ are invertible over $R$.  By the definition of
paired-equivalence given above we see $(D_{\mu},SD_{\widehat{\nu}})$
is in the orbit of $(D_{\mu},N)$, so we will work with this simpler
matrix.  (Note that also $(D_{\mu}S, D_{\widehat{\nu}})$ is in the
orbit of $(D_{\mu},SD_{\widehat{\nu}})$ as well).  We will act on
$SD_{\widehat{\nu}}$ on the left using $\mu$-admissible matrices,
generic with respect to $S$, as in the previous section, but we will
also be using generic $\widehat{\nu}$-admissible matrices on the
right.  Suppose $Q$ is a given $\mu$-admissible matrix, and $T$ is
$\widehat{\nu}$-admissible.  Then, by definition,
\[ D_{\mu}QD_{\mu}^{-1}, D_{\widehat{\nu}}^{-1}TD_{\widehat{\nu}}
\in GL_{r}(R). \] Hence, we may act on the pair $(D_{\mu},
SD_{\widehat{\nu}})$ via \begin{align*} \big( D_{\mu}QD_{\mu}^{-1},
Q,D_{\widehat{\nu}}^{-1}TD_{\widehat{\nu}} \big) \cdot (D_{\mu},
SD_{\widehat{\nu}}) &= \Big( \big(D_{\mu}QD_{\mu}^{-1} \cdot D_{\mu}
\cdot Q^{-1}\big), \big( Q \cdot S D_{\widehat{\nu}} \cdot
D_{\widehat{\nu}}^{-1}TD_{\widehat{\nu}} \big) \Big) \\
& =\Big(  D_{\mu} ,  (Q \cdot S \cdot T) \cdot D_{\widehat{\nu}}
\Big).
\end{align*}

Thus, finding invariants for the orbit $(D_{\mu},
SD_{\widehat{\nu}})$ is equivalent to considering invariants of the
orbit $S \mapsto QST$ of the invertible matrix $S$ under the action
of $\mu$-admissible matrices $Q$ on the left, and
$\widehat{\nu}$-admissible matrices $T$ on the right.

\begin{df}
A lower-triangular matrix $L \in GL_{r}(R)$ which may be factored
\[ L = Q_{L}Q_{U}S T_{U}T_{L} \]
so that, for some $S \in GL_{r}(R)$ the matrices $Q_{L},Q_{U},T_{U}$
and $T_{L}$ satisfy Inequalities~\ref{row} and~\ref{mu-gap} of
Lemma~\ref{row-ineq}, and Inequalities~\ref{col} and~\ref{nu-gap} of
Corollary~\ref{col-ineq} will be called a {\em
$\mu$-$\widehat{\nu}$-generic matrix.}
\end{df}

 \begin{cor} Given any matrix pair $(M,N)$ such that $\mu = inv(M)$
 and $\nu = inv(N)$, there is a $\mu$-$\widehat{\nu}$-generic matrix
 $L$ such that $(D_{\mu}, L D_{\widehat{\nu}})$ is in the orbit of
 $(M,N)$.
 \end{cor}
 \begin{proof}
There is a matrix $N'$ such that $(D_{\mu}, N')$ is in the orbit of
$(M,N)$.  Let us write
\[ N' = SD_{\widehat{\nu}}W \]
for invertible $S,W \in GL_{r}(R)$.  Then we also have $(D_{\mu},
SD_{\widehat{\nu}})$ in the same orbit.  With appropriate
upper-triangular matrices $\widehat{Q_{U}}$ and $\widehat{T_{U}}$
(which are necessarily $\mu$-admissible and
$\widehat{\nu}$-admissible, respectively), we may ensure that the
product $\widehat{Q_{U}}S \widehat{T_{U}}$ is lower-triangular.  The
result now follows from Lemma~\ref{row-ineq} and
Corollary~\ref{col-ineq}.
\end{proof}

\begin{df}
Let us establish the following notation:
\[ J_{k} = (1,2,\dots, k), \quad \widehat{J}_{k} = (r-k+1, r-k+2,
\dots, r), \] and, more generally, if $U=(u_1, u_2, \dots, u_k)$ is
any index set (so $u_1 < u_2 < \dots < u_k$), then set
\[ \widehat{U}=(\widehat{u}_1, \dots, \widehat{u}_k) = \big( (r-u_k
+1), (r-u_{k-1}+1), \dots, (r-u_2 + 1), (r-u_1 + 1) \big). \]
\end{df}
\begin{prop}
Suppose we are given ${\mathbb R}$-partitions $\mu$ and $\nu$ of
length $r$.  Let $N$ be a $\mu$-generic, and $L$ be a
$\mu$-$\widehat{\nu}$-generic matrix such that $(D_{\mu}, N)$ and
$(D_{\mu}, LD_{\widehat{\nu}})$ are pair-equivalent. Then if $I$ is
any index set of length $k$, $1 \leq k \leq r$, we have
\label{mu=mu-nu}
\[  \left\|
\big(N \big)_{I\widehat{J}_{k}} \right\| = \left\|
\big(LD_{\widehat{\nu}}\big)_{IJ_{k}} \right\|. \]
\end{prop}
\begin{proof}  Since $L$ is $\mu$-$\widehat{\nu}$-generic, we may
write
\[ LD_{\widehat{\nu}} =Q_{L}Q_{U}\big(S T_{U}T_{L}D_{\widehat{\nu}}\big)=
\widehat{Q_{U}}\widehat{Q_{L}}\big(S
T_{U}T_{L}D_{\widehat{\nu}}\big). \] The left factors
$\widehat{Q_{U}}\widehat{Q_{L}}$ satisfy the conditions of
Lemma~\ref{row-ineq}, and so may serve as factors on the left to put
$S T_{U}T_{L}D_{\widehat{\nu}}$ into $\mu$-generic form. Note that,
since $L$ is $\mu$-$\widehat{\nu}$-generic, by
Corollary~\ref{col-ineq} the orders of entries in $L$ increase from
right to left in any given row, but upon multiplying $L$ by
$D_{\widehat{\nu}}$, the orders now increase from left to right.
Therefore, all that remains is to put $LD_{\widehat{\nu}}$ into
upper triangular form, and then multiply by a generic upper
triangular matrix.

Let $\Pi_{r}$ denote the $r \times r$ permutation matrix that takes
$\vec{e}_{i}$,  the $i$-th standard basis vector, to
$\vec{e}_{r-i+1}$.  That is, a matrix of zeros except for $1$'s on
the off-diagonal, running from the lower left corner to the upper
right.  Note that, given the notation above and any matrix $W \in
M_{r}({\cal F})$, that if $I$ is any index set of length $k$, $1
\leq k \leq r$, then
\[ \big( W \Pi_{r}\big)_{IJ_{k}} = \big(W\big)_{I\widehat{J}_{k}}. \]

 Let us first right multiply
$LD_{\widehat{\nu}}$ by $\Pi_{r}$, so that orders of entries in the
product $LD_{\widehat{\nu}}\Pi_{r}$ now increase from right to left.
Thus, we may put this product into upper triangular form by a lower
triangular matrix $T_{L}$, whereupon we will multiply
$LD_{\widehat{\nu}}\Pi_{r}T_{L}$ by an upper triangular matrix
$T_{U}$ satisfying the conditions of Corollary~\ref{col-ineq}.  Let
us set $T = T_{L}T_{U}$. The upshot of the above is that the product
\[ LD_{\widehat{\nu}}\Pi_{r}T \]
is a $\mu$-generic matrix such that $(D_{\mu}, LD_{\widehat{\nu}})$
is in the orbit of $(D_{\mu}, LD_{\widehat{\nu}}\Pi_{r}T)$.

First, note that by Proposition~\ref{big-prop}, that if $N^*$ is any
other $\mu$-generic matrix such that $(D_{\mu},N^*)$ is in the orbit
of $(D_{\mu}, LD_{\widehat{\nu}}\Pi_{r}T)$, then for any index sets
$I$ and $J$ of length $k$,
\[ \left\| \big( N^* \big)_{IJ} \right\| = \left\| \big(
LD_{\widehat{\nu}}\Pi_{r}T \big)_{IJ} \right\|, \] so, it is
sufficient to prove \begin{equation*} \left\|
\big(LD_{\widehat{\nu}}\big)_{IJ_{k}} \right\| = \left\| \big(
LD_{\widehat{\nu}}\Pi_{r}T \big)_{I\widehat{J}_{k}} \right\| =
\left\| \big( LD_{\widehat{\nu}}\Pi_{r}T_{L}T_{U}
\big)_{I\widehat{J}_{k}} \right\| \label{left-right}
\end{equation*}
for any index set $I$ of length $k$.  For this, we will expand
$LD_{\widehat{\nu}}\Pi_{r}T_{L}T_{U}$ according to the Cauchy-Binet
formula:
\[\big(LD_{\widehat{\nu}}\Pi_{r}T_{L}T_{U}\big)_{I\widehat{J}_{k}} =
\sum_{U,V,W}\big(LD_{\widehat{\nu}}\big)_{IU}\big(\Pi_{r}\big)_{UV}\big(T_{L}\big)_{VW}\big(T_{U}\big)_{W\widehat{J}_{k}}
\]

Note that $(\Pi_{R})_{UV} = 0$ unless $V = \widehat{U}$.  Further,
by the generic nature of the minors $(T_{L})_{W\widehat{J}_{k}}$, we
have
\[ \left\| \big(LD_{\widehat{\nu}}\Pi_{r}T_{L}T_{U}\big)_{I\widehat{J}_{k}} \right\| =
\min_{U,V,W} \left\|
\big(LD_{\widehat{\nu}}\big)_{IU}\big(\Pi_{r}\big)_{UV}\big(T_{L}\big)_{VW}\big(T_{U}\big)_{W\widehat{J}_{k}}
\right\|.
\]

Clearly this minimum is obtained when $\left\|
\big(LD_{\widehat{\nu}}\big)_{IU} \right\|$ is at a minimum, and the
other factors have order zero.  This can be accomplished, using
Corollary~\ref{col-ineq}, with the term
\[ \left\| \big(LD_{\widehat{\nu}}\big)_{IJ_k}\big(\Pi_{r}\big)_{J_{k}\widehat{J}_{k}}\big(T_{L}\big)_{\widehat{J}_{k}\widehat{J}_{k}}
\big(T_{U}\big)_{\widehat{J}_{k}\widehat{J}_{k}} \right\| = \left\|
\big(LD_{\widehat{\nu}}\big)_{IJ_k} \right\|, \] from which the
proposition is proved.
\end{proof}

Thus, we now have an new way to calculate the right L${\mathbb
R}$-filling of a pair $(M,N)$.  Our first method was to find a
$\mu$-generic $N^*$ such that $(D_{\mu}, N^*)$ was in the orbit of
$(M,N)$, and use the determinantal formulas to calculate the
filling.  We are now also able to find a
$\mu$-$\widehat{\nu}$-generic matrix $L$ such that $(D_{\mu},
LD_{\widehat{\nu}})$ is in the orbit as well, and use (almost) the
same determinantal formulas as in the $\mu$-generic case, except we
will use the leftmost columns instead of the rightmost.

Further, since $(D_{\mu}L, D_{\widehat{\nu}})$ is in the orbit of
$(D_{\mu}, LD_{\widehat{\nu}})$, we have an obvious way (switching
the roles of columns and rows) to generate determinantal formulas to
{\em also} compute the left filling of the pair, from the matrix
$D_{\mu}L$.

It might be argued, on the basis of the above, that we could
dispense with the $\mu$-generic form altogether.  However, though
the orders of appropriate minors of $LD_{\widehat{\nu}}$ coming from
the $\mu$-$\widehat{\nu}$-generic matrix $L$, equal the orders of
determinants of corresponding minors of the $\mu$-generic matrix
$N^*$, there does not seem to be a clear way to {\em prove} that the
filling determined by $L$ is necessarily a L${\mathbb R}$-filling.
So, our method is to use the $\mu$-generic matrix $N^*$ to define
the filling, and also to prove that it is a L${\mathbb R}$-filling,
and then pass to the $\mu$-$\widehat{\nu}$-generic form using
Proposition~\ref{mu=mu-nu} above.

More importantly, we will use the $\mu$-$\widehat{\nu}$-generic form
to dispense with determinantal formulas entirely, and use this form
to calculate L${\mathbb R}$-fillings directly from invariant factors
of matrices related to this form in a simple manner.

\begin{df}
Let us use the notation
\[ (\mu_{i}, \dots, \mu_1, H , \nu_1, \dots , \nu_j) \]
to denote the product of matrices
\[ diag(t^{\mu_1}, \dots t^{\mu_{i}}, t^0 \dots , t^0) \cdot H \cdot diag(t^0 \dots ,
t^0, t^{\nu_{j}}, \dots , t^{\nu_{1}}). \]

We will also let
\[( \mu H \nu_1, \dots , \nu_{j}) = (\mu_{r}, \dots,
\mu_1, H , \nu_1, \dots , \nu_j), \] and  \[(  H \nu_1, \dots ,
\nu_{j}) =   H \cdot diag(t^0 \dots , t^0, t^{\nu_{j}}, \dots ,
t^{\nu_{1}})   \] and similarly for $(\mu_{i} \dots , \mu_1 H \nu)$
and $(\mu_{i} \dots , \mu_1 H )$.
\end{df}

\begin{thm}
Let $(M,N) \in M_{r}({\cal F})^2$ be a matrix pair over ${\cal F}$,
with both matrices of full rank. Let $\{ k_{ij} \}$ be the right
filling associated to $(M,N)$ in $L{\mathbb R}(\mu, \nu, \lambda)$
and let $\{ m_{ij} \}$ be the left filling in $L{\mathbb R}(\nu,
\mu, \lambda)$. Let $(D_{\mu}, LD_{\widehat{\nu}})$ be a matrix pair
in the orbit of a pair $(M,N)$ where $L \in GL_{r}(R)$ is a lower
triangular $\mu$-$\widehat{\nu}$-generic matrix. Then the sequence
of invariants
\[ inv(\mu L),
inv(\mu L\nu_1),inv(\mu L \nu_1\nu_2),\dots inv(\mu L \nu_1, \dots
,\nu_r)
\] forms a L${\mathbb R}$ sequence for the right filling $\{ k_{ij}
\}$, and the sequence
\[ inv(L \nu), inv(\mu_1 L\nu),inv(\mu_2 \mu_1 L\nu),
\dots , inv(\mu_r, \dots, \mu_1 L\nu), \] is a L${\mathbb R}$
sequence associated to the left filling $\{ m_{ij}
\}$.\label{symm-fill}
\end{thm}

Before proving this result, let us show a brief example of the
phenomenon it implies.  We will work, for now, with integer-valued
fillings, though our results do not require this.  Given the
partitions
\[ \mu = (9,5,2,1),\  \nu=(11,6,3,2), \quad \hbox{and} \quad \lambda= (15,10,8,6)  \]
one may check that the following is a \LR, and hence L${\mathbb
R}$-filling, of $\lambda / \mu$ with content $\nu$:

\[ \begin{array}{cccc}
k_{11} = 6 &   & & \\
k_{12} = 2 &k_{22} = 3 &  \\
k_{13} = 2 & k_{23} = 2 & k_{33} = 2 & \\
 k_{14} = 1& k_{24} = 1   & k_{34} = 1 & k_{44} = 2 \\ \hline
 ( \ \ 11 & \ \ 6 & \ \ 3 & 2  \ \ )\end{array} \]

 Using the construction of Theorem~\ref{lr_to_mat} we may use this
 filling to build the matrices:
 \[
M :=  \left[ {\begin{array}{cccc}
t^{9} & 0 & 0 & 0 \\
0 & t^{5} & 0 & 0 \\
0 & 0 & t^{2} & 0 \\
0 & 0 & 0 & t
\end{array}}
 \right]
\]
and
\[
N_1 :=  \left[ {\begin{array}{cccc}
t^{6} & 1 & 0 & 0 \\
0 & t^{2} & 1 & 0 \\
0 & 0 & t^{2} & 1 \\
0 & 0 & 0 & t
\end{array}}
 \right]
, N_2 :=  \left[ {\begin{array}{rccc}
1 & 0 & 0 & 0 \\
0 & t^{3} & 1 & 0 \\
0 & 0 & t^{2} & 1 \\
0 & 0 & 0 & t
\end{array}}
 \right]
, N_3 :=  \left[ {\begin{array}{rrcc}
1 & 0 & 0 & 0 \\
0 & 1 & 0 & 0 \\
0 & 0 & t^{2} & 1 \\
0 & 0 & 0 & t
\end{array}}
 \right]
, N_4 :=  \left[ {\begin{array}{rrrc}
1 & 0 & 0 & 0 \\
0 & 1 & 0 & 0 \\
0 & 0 & 1 & 0 \\
0 & 0 & 0 & t^{2}
\end{array}}
 \right]
\]
So that
\[
N =N_1 N_2 N_3 N_4 =  \left[ {\begin{array}{cccc}
t^{6} & t^{3} & t^{2} & t^{2} \\
0 & t^{5} & 2\,t^{4} &  t^{3} + 2\,t^{4}  \\
0 & 0 & t^{6} & t^4 + t^5 + t^6 \\
0 & 0 & 0 & t^{5}
\end{array}}
 \right].
\]

By Theorem~\ref{lr_to_mat}, and as we can check, we have
\[ inv(M) = \mu= (9,5,2,1), \ inv(N) = \nu =(11,6,3,1), \quad \hbox{and} \quad
inv(MN) = \lambda = (15,10,8,6). \] In fact, the matrix $N$ is in
$\mu$-generic form, insofar as it is upper-triangular and satisfies
Inequalities~\ref{row} and~\ref{mu-gap} of Lemma~\ref{row-ineq}, and
Inequality~\ref{col} of Corollary~\ref{col-ineq}.  We can, in fact,
obtain a matrix in $\mu$-$\widehat{\nu}$-generic form from $N$ above
by adding multiples of column 4 to columns 3,2, and 1 to put zeros
in the top row to the left of the $(1,4)$ entry, and then adding
multiples of column 3 to columns 2 and 1, etc, and then finally
reversing the order of the columns and obtaining:

\[
  \left[ {\begin{array}{cccc}
  t^{2} &0 & 0 & 0\\
  2\,t^{4} + t^{3} &  - t^{3}& 0 & 0 \\
t^{6} + t^{5} + t^{4} &  - t^{5} - t^{4}& t^{6} &0  \\
 t^{5}&  - t^{5} & t^{7} & - t^{11}
\end{array}}
 \right]
\]

We note that in the above matrix, every entry in column $j$ has
order at least $\nu_{(4-j+1)}$.  So, we may factor the above as the
product (after multiply some columns by $-1$):
\[ LD_{\widehat{\nu}} \]
where
\[ L =\left[ {\begin{array}{cccr}
1 & 0 & 0 & 0 \\
2\,t^{2} + t & 1 & 0 & 0 \\
t^{4} + t^{3} + t^{2} & t^{2} + t & 1 & 0 \\
t^{3} & t^{2} & t & 1
\end{array}}
 \right].
\]
In this case, it can be shown that the above matrix is in
$\mu$-$\widehat{\nu}$-generic form, insofar as it is
lower-triangular and satisfies Inequalities~\ref{row}
and~\ref{mu-gap} of Lemma~\ref{row-ineq} and also
Inequalities~\ref{col} and~\ref{nu-gap} of Corollary~\ref{col-ineq}.

So, by Theorem~\ref{symm-fill} we may alternately calculate the
right filling of $\lambda / \mu$ with content $\nu$ by computing the
successive invariant factors (written vertically): \begin{align*}
inv(D_{\mu}L) & = inv\left( \left[ {\begin{array}{cccc}
t^{9} & 0 & 0 & 0 \\
0 & t^{5} & 0 & 0 \\
0 & 0 & t^{2} & 0 \\
0 & 0 & 0 & t
\end{array}}
 \right] \left[ {\begin{array}{cccc}
1 & 0 & 0 & 0 \\
2\,t^{2} + t & 1 & 0 & 0 \\
t^{4} + t^{3} + t^{2} & t^{2} + t & 1 & 0 \\
t^{3} & t^{2} & t & 1
\end{array}}
 \right] \right)\\
 & = inv\left(\left[ {\begin{array}{cccc}
t^9 & 0 & 0 & 0 \\
2\,t^{7} + t^6 & t^5 & 0 & 0 \\
t^{6} + t^{5} + t^{4} & t^{4} + t^3 & t^2 & 0 \\
t^{4} & t^{3} & t^2 & t
\end{array}}
 \right] \right) = \left( \begin{array}{c}9 \\ 5 \\ 2\\ 1
 \end{array} \right) = \mu. \end{align*}

\[ inv(D_{\mu}L \nu_1) = inv\left( \left[ {\begin{array}{cccc}
t^9 & 0 & 0 & 0 \\
2\,t^{7} + t^6 & t^5 & 0 & 0 \\
t^{6} + t^{5} + t^{4} & t^{4} + t^3 & t^2 & 0 \\
t^{4} & t^{3} & t^2 & t\cdot t^{11}
\end{array}}
 \right] \right) = \left( \begin{array}{l} 15 =9 +6 \\ 7 \ =5+2 \\ 4 \
 = 2+2 \\ 2 \ =1 + 1
 \end{array} \right). \]

\[ inv(D_{\mu}L \nu_1 \nu_2) = inv\left( \left[ {\begin{array}{cccc}
t^9 & 0 & 0 & 0 \\
2\,t^{7} + t^6 & t^5 & 0 & 0 \\
t^{6} + t^{5} + t^{4} & t^{4} + t^3 & t^2\cdot t^6 & 0 \\
t^{4} & t^{3} & t^2 \cdot t^6 & t\cdot t^{11}
\end{array}}
 \right]\right) = \left( \begin{array}{l} 15 =9 +6 \\ 10 =5+2 + 3 \\6  \
 = 2+2+2 \\ 3 \ =1 + 1 + 1
 \end{array} \right). \]

 \[ inv(D_{\mu}L \nu_1 \nu_2 \nu_3) = inv\left(\left[ {\begin{array}{cccr}
t^9& 0 & 0 & 0 \\
2\,t^{7} + t^6 & t^5\cdot t^3 & 0 & 0 \\
t^{6} + t^{5} + t^{4} & (t^{4} + t^3)\cdot t^3 & t^2\cdot t^6 & 0 \\
t^{4}& t^{3}\cdot t^3 & t^2 \cdot t^6 & t\cdot t^{11}
\end{array}}
 \right]\right) = \left( \begin{array}{l} 15 =9 +6 \\ 10 =5+2 + 3 \\8  \
 = 2+2+2 + 2 \\ 4\ =1 + 1 + 1 + 1
 \end{array} \right). \]

 \begin{align*} inv(D_{\mu}L \nu_1 \nu_2 \nu_3 \nu_4) =  inv & \left(  \left[  {\begin{array}{cccr}
t^9\cdot t & 0 & 0 & 0 \\
(2\,t^{7} + t^6)\cdot t & t^5\cdot t^3 & 0 & 0 \\
(t^{6} + t^{5} + t^{4}) \cdot t & (t^{4} + t^3)\cdot t^3 & t^2\cdot t^6 & 0 \\
t^{4}\cdot t & t^{3}\cdot t^3 & t^2 \cdot t^6 & t\cdot t^{11}
\end{array}}
 \right] \right) \\
 &  = \left( \begin{array}{l} 15 =9 +6 \\ 10 =5+2 + 3 \\8  \
 = 2+2+2 + 2 \\ 6\ =1 + 1 + 1 + 1 + 2
 \end{array} \right) = \lambda. \end{align*}

However, not only does decreasing the powers of $D_{\widehat{\nu}}$
appearing in the product allow us to calculate the left filling of
$\lambda / \mu$ with content $\nu$ associated to the pair
$(D_{\mu},LD_{\widehat{\nu}})$, but at the same time, decreasing the
powers appearing in $D_{\mu}$ allows us to find the {\em right}
filling of $\lambda / \nu$ with content $\nu$:

\[ inv(LD_{\widehat{\nu}}) =\left( \begin{array}{c}11 \\ 6 \\ 3\\2
 \end{array} \right) = \nu,  \ inv( \mu_1 LD_{\widehat{\nu}}) =
  \left( \begin{array}{l} 15 = 11 + 4 \\
 8 = 6 + 2  \\5 = 3 + 2 \\ 3 \ =2 + 1
 \end{array} \right), \ inv(\mu_2 \mu_1 L D_{\widehat{\nu}}) =
 \left( \begin{array}{l} 15 =11+4\\ 10 = 6 + 2  +2 \\7 \
 =  3 + 2 +2 \\ 4 \ =2 + 1 +1
 \end{array} \right), \]

 \[ inv(\mu_3 \mu_2 \mu_1 L D_{\widehat{\nu}}) =  \left( \begin{array}{l}
 15 =11+4\\ 10 = 6 + 2  +2 \\8 \
 =  3 + 2 +2 +1\\ 5 \ =2 + 1 +1+1
 \end{array} \right), \ inv(\mu_4 \mu_3 \mu_2 \mu_1 L D_{\widehat{\nu}}) =
 \left( \begin{array}{l} 15 =11+4\\ 10 = 6 + 2  +2 \\8 \
 =  3 + 2 +2 +1 \\ 6 \ =2 + 1 +1+1+1
 \end{array} \right) = \lambda. \]

\begin{proof}
We shall prove the results for right fillings.  Since $(D_{\mu}L,
D_{\widehat{\nu}})$ is in the orbit of $(D_{\mu}, L
D_{\widehat{\nu}})$, the result for left fillings will follow by an
appropriate interchange of rows and columns.  Suppose $L$ is some
$\mu$-$\widehat{\nu}$-generic matrix such that $(D_{\mu}, L
D_{\widehat{\nu}})$ is in the orbit of $(M,N)$. By
Proposition~\ref{mu=mu-nu} we may calculate the right filling of the
pair $(M,N)$ using the row indices of $L$ appearing in
Equation~\ref{seq-def}, and the leftmost columns.

Let us fix $i$ and $j$, where $i < j$, and choose some index
$\sigma$ such that $i \leq \sigma$.  Let $\nu^* = (\nu_1, \dots,
\nu_{\sigma}, 0, \dots, 0)$.  Then

\begin{eqnarray*}
k_{1j} + k_{2j} + \dots + k_{ij}
&= &\left\| LD_{\widehat{\nu}} \begin{pmatrix} 1\dots  (j-i-1),(j) \dots r\\
1, 2 \dots  ,r-i+1 \end{pmatrix} \right\| - \left\|
LD_{\widehat{\nu}}\begin{pmatrix} 1\dots
   (j-i),(j+1) \dots r\\
1, 2 \dots  ,r-i+1 \end{pmatrix} \right\| \\
&= &\left\| L \begin{pmatrix} 1\dots    (j-i-1),(j) \dots r\\
1, 2 \dots  ,r-i+1 \end{pmatrix} \right\| +(\nu_{r} + \dots +
\nu_{i+1}+ \nu_i)\\ & \qquad  & \quad \quad - \left\| L
\begin{pmatrix} 1\dots
   (j-i),(j+1) \dots r\\
1, 2 \dots  ,r-i+1 \end{pmatrix} \right\| - (\nu_{r} +
 \dots + \nu_{i+1}+ \nu_i) \\
& = &\left\| L \begin{pmatrix} 1\dots    (j-i-1),(j) \dots r\\
1, 2 \dots  ,r-i+1 \end{pmatrix} \right\|  - \left\| L
\begin{pmatrix} 1\dots
   (j-i),(j+1) \dots r\\
1, 2 \dots  ,r-i+1 \end{pmatrix} \right\|  \\
& = &\left\| (L,\nu_1, \dots , \nu_{\sigma}) \begin{pmatrix} 1\dots    (j-i-1),(j) \dots r\\
1, 2 \dots  ,r-i+1 \end{pmatrix} \right\|  \\
& \qquad & \qquad - \left\| (L,\nu_1, \dots , \nu_{\sigma})
\begin{pmatrix} 1\dots
   (j-i),(j+1) \dots r\\
1, 2 \dots  ,r-i+1 \end{pmatrix} \right\| .
\end{eqnarray*}

Let $\{ k_{ij}^*\}$ denote the right L${\mathbb R}$-filling of the
pair $(D_{\mu}, LD_{\widehat{\nu^*}})$.  To calculate $\{
k_{ij}^*\}$, by Corollary~\ref{col-ineq} we must first put $L$ into
$\mu$-$\widehat{\nu*}$-generic form by multiplying $L$ by a lower
triangular $\widehat{\nu^*}$-admissible matrix $T_{\nu^*}$ so that
$LT_{\nu^*}$ is $\mu$-$\widehat{\nu^*}$-generic. Having done this,
computing the $k_{ij}^*$ for $i<j$ can be accomplished using the
orders of minors $(LT_{\nu^*}D_{\widehat{\nu^*}})_{IJ_{k}}$ for
appropriate index sets $I$. But then
\begin{align*}
 \| (LT_{\nu^*}D_{\widehat{\nu^*}})_{IJ_{k}} \| & = \left\| \sum_{H \subseteq
 I} (L)_{IH} (T_{\nu^*})_{HJ_{k}} (D_{\widehat{\nu^*}})_{J_{k}J_{k}}
 \right\| \\
 \intertext{But by the generic nature of the factor
 $(T_{\nu^*})_{HJ_{k}}$, we have}
 &= \min_{H \subseteq I}\| (L)_{IH} (T_{\nu^*})_{HJ_{k}} (D_{\widehat{\nu^*}})_{J_{k}J_{k}}
 \| \\
 &= \| (L)_{IJ_{k}} (T_{\nu^*})_{J_{k}J_{k}}
 (D_{\widehat{\nu^*}})_{J_{k}J_{k}} \| \\
 & =\| (L)_{IJ_{k}}
 (D_{\widehat{\nu^*}})_{J_{k}J_{k}} \| = \|
 (LD_{\widehat{\nu^*}})_{IJ_{k}}\|. \label{fix-fill}
\end{align*}

Thus, computing orders of minors
$\|(LT_{\nu^*}D_{\widehat{\nu^*}})_{IJ_{k}} \|$, from which the
L${\mathbb R}$-filling of $(D_{\mu}, LD_{\widehat{\nu^*}})$ may be
determined, amounts to calculating the minors
$\|(LD_{\widehat{\nu^*}})_{IJ_{k}}\|$.  This, combined with the
above calculation, proves that $k_{ij} = k_{ij}^*$ for all $i$, $1
\leq i \leq \sigma$, and $j$ such that $i < j \leq r$. The parts of
the fillings $k_{ii}$ and $k_{ii}^*$ are shown equal for $1 \leq i
\leq \sigma$ by noting that by Theorem~\ref{symm-fill} the shapes of
$inv(D_{\mu}LD_{\widehat{\nu}})$ and
$inv(D_{\mu}LD_{\widehat{\nu^*}})$ are the same in rows $1$ through
$\sigma$, and since no parts $k_{ij}$, for $i > \sigma$ appear above
row $\sigma +1$.  So, having shown for the interior $(i,j)$ (where
$i < j$) that $k_{ij} = k_{ij}^*$, the equality of the shapes of
$inv(D_{\mu}LD_{\widehat{\nu}})$ and
$inv(D_{\mu}LD_{\widehat{\nu^*}})$  in rows $1$ through $\sigma$
forces $k_{ii} = k_{ii}^*$, for $1 \leq i \leq \sigma$.

Finally, the $j$th row of the invariant partition $inv(D_{\mu} L
D_{\widehat{\nu^*}})=inv(\mu L \nu_1, \dots, \nu_{\sigma})$ is
$\mu_j + k_{1j}^* + \dots + k_{\sigma j}^*$ since $\nu^* = (\nu_1,
\dots, \nu_{\sigma}, t^0, \dots, t^0)$.  But then this row also
equals $\mu_j + k_{1j} + \dots  + k_{\sigma j}$, which proves the
Theorem.
\end{proof}

An interesting consequence of the above proof is that invariants
$k_{ij}$ for $i< j$ are determined entirely by the
$\mu$-$\widehat{\nu}$-generic matrix $L$, even though this matrix is
invertible. By the above it is clear that any
$\mu$-$\widehat{\nu}$-generic matrix appearing in the orbit of
($M,N)$ will yield the same fillings.  As a corollary to the proof
of the above Theorem we record the following:

\begin{cor}Suppose $L$ is a $\mu$-$\widehat{\nu}$-generic matrix
with respect to ${\mathbb R}$-partitions $\mu$ and $\nu$ of length
$r$, where $\nu = (\nu_1, \dots , \nu_r)$.  Let $\nu^*$ be the
${\mathbb R}$-partition
\[ \nu^* = (\nu_1, \nu_2, \dots, \nu_{\sigma}, \alpha, 0, \dots, 0),
\]
where \label{same-same}
\[ 0 < \alpha \leq \nu_{\sigma+1}. \]
Let $\{ k_{ij}^* \}$ denote the right filling associated to the pair
$(D_{\mu}, LD_{\widehat{\nu^*}})$.  Then the pairs $(D_{\mu},
LD_{\widehat{\nu}})$ and $(D_{\mu}, LD_{\widehat{\nu^*}})$ have the
same fillings for $\nu_1$ through $\nu_{\sigma}$.  That is,
\[ k_{ij} = k_{ij}^*, \quad 1 \leq i \leq \sigma, \ \ i \leq j \leq
r. \] Analogous results hold for left fillings.
\end{cor}

So, the content of Theorem~\ref{symm-fill} is that given a
$\mu$-$\widehat{\nu}$-generic matrix $L$, we can calculate the right
filling of the pair $(D_{\mu}, LD_{\widehat{\nu}})$ by the
L${\mathbb R}$ sequence of partitions formed by the partial
diagonals on the right:
\[ inv(\mu L \nu_1,\dots ,\nu_r), inv(\mu L \nu_1,\dots ,\nu_{r-1}),
\dots, inv(\mu L \nu_1 \nu_2), inv(\mu L \nu_1) ,inv(\mu L) \] where
the $i$-strip of the filling is given by the skew-shape
\[ inv(\mu L \nu_1,\dots ,\nu_i) / inv(\mu L \nu_1,\dots
,\nu_{i-1}),\] and with the same matrix $L$ we may calculate an
associated left-filling of $(D_{\mu}, L D_{\widehat{\nu}})$ by the
sequence of partitions formed by the partial diagonals on the left:
\[ inv(\mu_{r}, \dots, \mu_1, L \nu), inv(\mu_{r-1}, \dots, \mu_1, L
\nu), \dots , inv(\mu_2 \mu_1, L \nu), inv(\mu_1, L \nu), inv(L
\nu).
\]

We shall prove below that this association between left and right
fillings depends only on the fillings, and not on the particular
matrix realization of the filling.

However, even with these results we can reach some interesting
conclusions regarding invariants for matrix pairs over our valuation
rings, and facts about L${\mathbb R}$-fillings (and classical
\LR-fillings) these results imply.  For example, the following is
easily proved:

\begin{prop}  Let $(M,N) \in M_{r}({\cal F})^2$ be a pair of
full-rank matrices such that $inv(M) = \mu$, $inv(N) =\nu$, and
$inv(MN) = \lambda$, and suppose $\{ k_{ij} \}$ is the associated
right L${\mathbb R}$-filling of the pair.  Let $D(\alpha) =
diag(\alpha, \alpha, \dots, \alpha)$ be a scalar matrix, for some
$\alpha \in {\cal F}$.  Let $\{ k_{ij}^* \}$ denote the right
filling for the pair $(M,N\cdot D(\alpha))$.  Then
\[ k_{ij}^* = k_{ij}, \ \ \hbox{for} \ \ i < j, \]
and
\[ k_{ii}^* = k_{ii} + \alpha. \]
An analogous result holds for the left filling of the pair.
\end{prop}
\begin{proof}
Let $\nu + (\alpha)$ denote the ${\mathbb R}$-partition
\[ \nu + (\alpha) = (\nu_1 + \alpha, \nu_2  + \alpha, \dots, \nu_r +
\alpha). \]  If $(D_{\mu} L D_{\widehat{\nu}})$ is a
$\mu$-$\widehat{\nu}$-generic pair in the orbit of $(M,N)$, then the
matrix $L$ is $\widehat{\nu + (\alpha)}$-generic.  This follows from
noting that any $\widehat{\nu}$-admissible matrix is $\widehat{\nu +
(\alpha)}$-admissible, and $D(\alpha) D_{\widehat{\nu}} =
D_{\widehat{\nu + (\alpha)}}$. Therefore, since we may now compute
the interior parts $k_{ij}^*$ for the pair $(D_{\mu} L
D_{\widehat{\nu}})$ using the same matrix $L$ that we may use for
the pair $(M,N)$, we see both pairs have the same interior parts
(which are determined entirely by the matrix $L$).  The result on
edge parts follows as:
\begin{align*} k_{ii}^* & = \nu_{i} +\alpha - (k_{i,i+1}^* + \dots k_{i,r}^* )\\
 & =\nu_{i} +\alpha - (k_{i,i+1} + \dots k_{i,r} ) = k_{ii} +
 \alpha. \end{align*}
\end{proof}

This result has a number of simple corollaries regarding L${\mathbb
R}$-fillings and matrices over valuation rings:

\begin{cor} Let $\nu$ be an ${\mathbb R}$-partition, and let $\{ k_{ij} \}$ be
a L${\mathbb R}$-filling of some skew shape $\lambda / \mu$, with
content $\nu$.  Then the mapping
\[ k_{ij} \mapsto k_{ij}', \ \ \hbox{for} \ i < j, \]
and
\[ k_{ii} + \alpha \mapsto k_{ii}',  \ \ \hbox{for all $i$,} \ 1 \leq i \leq r, \]
is a bijection from the set L${\mathbb R}(\mu, \nu; \lambda)$ of
L${\mathbb R}$-fillings of type $(\mu,\nu; \lambda)$, into the set
L${\mathbb R}(\mu, \nu+(\alpha); \lambda + (\alpha))$ of L${\mathbb
R}$-fillings of type $(\mu, \nu+(\alpha); \lambda + (\alpha))$.  An
analogous result holds for left fillings. \label{fill-biject}
\end{cor}

Interpreting the above result in terms of classical \LR-fillings
results in the following.

\begin{cor} Let $(\mu, \nu ; \lambda)$ be a triple of partitions of
non-negative integers.  Let $\alpha$ be some positive integer.  Then
\[ c_{\mu, \nu}^{\lambda} = c_{\mu, \nu+ (\alpha)}^{\lambda +
(\alpha)} = c_{\mu + (\alpha), \nu}^{\lambda + (\alpha)}, \] where
$c_{\mu, \nu}^{\lambda}$ is the \LR\ coefficient of the triple
$(\mu, \nu; \lambda)$, denoting the number of (non-negative integer
valued) \LR-fillings of $\lambda / \mu$ with content $\nu$.
\end{cor}

For future use, we also record the following observation:

\begin{cor}
Given any matrix pair $(M,N) \in M_{r}({\cal F})^2$, both of full
rank, there exist real numbers $\alpha, \beta \in {\mathbb R}$ such
that the left and right fillings of the pair
\[ (D(\beta)M,ND(\alpha)) \] \label{diag-shift}
are all non-negative. \end{cor}

In other words, multiplying $M$ and $N$ by appropriate scalars will
``shift'' the filling so that even the edge parts of the filling are
non-negative.  Pictorially, we imagine starting with a matrix pair
$(M,N)$ whose associated L${\mathbb R}$ diagram has the form:

\vspace{0.3in} \psset{unit=0.7cm} 
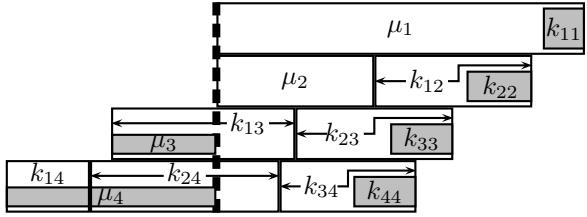
\hspace{2in}
\begin{pspicture}(11,4)
\psline[linewidth=3pt, linestyle=dashed](0,0)(0,4)
\psframe(0,3)(7,4) \rput{*0}(3.5,3.5){$\mu_1$} \psframe(0,2)(3,3)
\rput{*0}(1.5,2.5){$\mu_2$}
\psframe[fillstyle=solid,fillcolor=lightgray](-2,1.1)(0,1.5)
\rput{*0}(-1,1.3){$\mu_3$}
\psframe[fillstyle=solid,fillcolor=lightgray](-4,0.1)(0,0.5)
\rput{*0}(-2,0.3){$\mu_4$}
\psframe(3,2)(6,3)\psline{->}(0.2,1.7)(-2,1.7)
\psline{->}(1,1.7)(1.5,1.7) \rput{*0}(4,2.5){$k_{12}$}
\psframe(-2,1)(1.5,2) \rput{*0}(0.6,1.7){$k_{13}$}
\psframe(1.5,1)(4.5,2) \rput{*0}(2.4,1.5){$k_{23}$}
\psline{->}(2.8,1.5)(3, 1.5)(3,1.8)(4.5,1.8)
\psline{->}(2.1,1.5)(1.5, 1.5) \psframe(-4,0)(-2.4,1)
\rput{*0}(-3.2,0.75){$k_{14}$}
\psframe(-2.4,0)(1.2,1)\rput{*0}(-0.6,0.75){$k_{24}$}
\psline{->}(-1,0.7)(-2.4,0.7)\psline{->}(-0.2,0.7)(1.2,0.7)
\psframe(1.2,0)(3.8,1) \rput{*0}(2,0.5){$k_{34}$}
\psframe[fillstyle=solid,fillcolor=lightgray](6.2,3.1)(7,3.9)
\psline{->}(4.4,2.5)(4.6, 2.5)(4.6,2.8)(6,2.8)
\rput{*0}(6.6,3.5){$k_{11}$}
\psframe[fillstyle=solid,fillcolor=lightgray](4.75,2.1)(6,2.7)
\rput{*0}(5.4,2.4){$k_{22}$} \psline{->}(3.6,2.5)(3,2.5)
\psframe[fillstyle=solid,fillcolor=lightgray](3.3,1.1)(4.5,1.7)
\rput{*0}(3.9,1.4){$k_{33}$}
\psframe[fillstyle=solid,fillcolor=lightgray](2.6,0.1)(3.8,0.7)
\rput{*0}(3.2,0.4){$k_{44}$} \psline{->}(2.3,0.5)(2.5,
0.5)(2.5,0.8)(3.8,0.8) \psline{->}(1.6,0.5)(1.2, 0.5)
\end{pspicture}

\vspace{0.3in}

Multiplying $M$ by a scalar matrix $D(\beta) = diag(\beta, \beta,
\dots, \beta)$, for $\beta > 0$, will have the effect of moving the
origin to the {\em left}.  The above diagram came from our earlier
example in which $\mu = (\mu_1, \mu_2, \mu_3, \mu_4) = (7,3,-2,
-4)$.  Consequently, so that $\mu+ (\beta)$ will have all positive
values, we will set $\beta = 5$:

\vspace{0.3in} \psset{unit=0.7cm} 
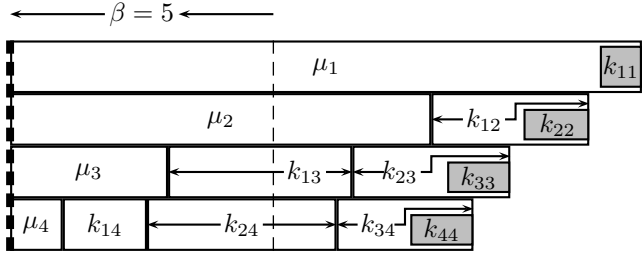
\hspace{2in}
\begin{pspicture}(11,4)
\psline[linewidth=3pt, linestyle=dashed](-5,0)(-5,4)
\psline[linewidth=0.5pt, linestyle=dashed](0,0)(0,4)
\psframe(-5,3)(7,4) \rput{*0}(1,3.5){$\mu_1$} \psframe(-5,2)(3,3)
\rput{*0}(-1,2.5){$\mu_2$}
\psframe(-5,1)(-2,2) \rput{*0}(-3.5,1.5){$\mu_3$}
\rput{*0}(-4.5,0.5){$\mu_4$} \psframe(-5,0)(-4,1)
\psframe(3,2)(6,3)\psline{->}(0.2,1.5)(-2,1.5)
\psline{->}(1,1.5)(1.5,1.5) \rput{*0}(4,2.5){$k_{12}$}
\psframe(-2,1)(1.5,2) \rput{*0}(0.6,1.5){$k_{13}$}
\psframe(1.5,1)(4.5,2) \rput{*0}(2.4,1.5){$k_{23}$}
\psline{->}(2.8,1.5)(3, 1.5)(3,1.8)(4.5,1.8)
\psline{->}(2.1,1.5)(1.5, 1.5) \psframe(-4,0)(-2.4,1)
\rput{*0}(-3.2,0.5){$k_{14}$}
\psframe(-2.4,0)(1.2,1)\rput{*0}(-0.6,0.5){$k_{24}$}
\psline{->}(-1,0.5)(-2.4,0.5)\psline{->}(-0.2,0.5)(1.2,0.5)
\psframe(1.2,0)(3.8,1) \rput{*0}(2,0.5){$k_{34}$}
\psframe[fillstyle=solid,fillcolor=lightgray](6.2,3.1)(7,3.9)
\psline{->}(4.4,2.5)(4.6, 2.5)(4.6,2.8)(6,2.8)
\rput{*0}(6.6,3.5){$k_{11}$}
\psframe[fillstyle=solid,fillcolor=lightgray](4.75,2.1)(6,2.7)
\rput{*0}(5.4,2.4){$k_{22}$} \psline{->}(3.6,2.5)(3,2.5)
\psframe[fillstyle=solid,fillcolor=lightgray](3.3,1.1)(4.5,1.7)
\rput{*0}(3.9,1.4){$k_{33}$}
\psframe[fillstyle=solid,fillcolor=lightgray](2.6,0.1)(3.8,0.7)
\rput{*0}(3.2,0.4){$k_{44}$} \psline{->}(2.3,0.5)(2.5,
0.5)(2.5,0.8)(3.8,0.8) \psline{->}(1.6,0.5)(1.2, 0.5)
\psline{->}(-3.2,4.5)(-5,4.5) \psline{->}(0,4.5)(-1.8,4.5)
\rput{*0}(-2.5,4.5){$\beta =5$}
\end{pspicture}

\vspace{0.3in}

Note that, under this transformation, all values of $\mu_{i} -
\mu_{j}$ remain the same.  In order to remove the negative values of
$k_{ii}$ in the above, we must multiply $LD_{\widehat{\nu}}$ by an
appropriate scalar matrix $D(\alpha)$.  Here, we set $\alpha = 3.2$:

\vspace{0.3in} \psset{unit=0.7cm} 
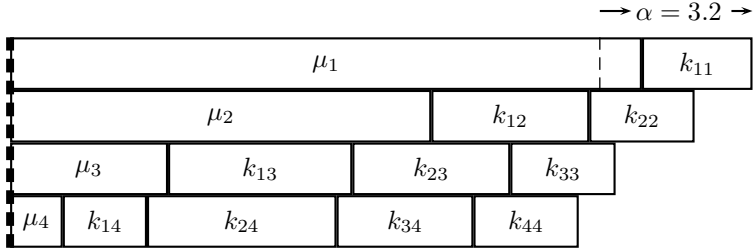
\hspace{2in}
\begin{pspicture}(11,4)
\psline[linewidth=3pt, linestyle=dashed](-5,0)(-5,4)
\psline[linewidth=0.5pt, linestyle=dashed](6.2,3)(6.2,4)
\psline{->}(8.7,4.5)(9.1,4.5) \psline{->}(6.2,4.5)(6.8,4.5)
\rput{*0}(7.7,4.5){$\alpha =3.2$} \psframe(-5,3)(7,4)
\rput{*0}(1,3.5){$\mu_1$} \psframe(-5,2)(3,3)
\rput{*0}(-1,2.5){$\mu_2$}
\psframe(-5,1)(-2,2) \rput{*0}(-3.5,1.5){$\mu_3$}
\rput{*0}(-4.5,0.5){$\mu_4$} \psframe(-5,0)(-4,1) \psframe(3,2)(6,3)
\rput{*0}(4.5,2.5){$k_{12}$} \psframe(-2,1)(1.5,2)
\rput{*0}(-0.25,1.5){$k_{13}$} \psframe(1.5,1)(4.5,2)
\rput{*0}(3,1.5){$k_{23}$}
 \psframe(-4,0)(-2.4,1)
\rput{*0}(-3.2,0.5){$k_{14}$}
\psframe(-2.4,0)(1.2,1)\rput{*0}(-0.6,0.5){$k_{24}$}
\psframe(1.2,0)(3.8,1) \rput{*0}(2.5,0.5){$k_{34}$}
\psframe(7,3)(9.1,4) \rput{*0}(8.05,3.5){$k_{11}$}
\psframe(6,2)(8,3) \rput{*0}(7,2.5){$k_{22}$} \psframe(4.5,1)(6.5,2)
\rput{*0}(5.5,1.5){$k_{33}$} \psframe(3.8,0)(5.8,1)
\rput{*0}(4.8,0.5){$k_{44}$}
\end{pspicture}

\vspace{0.3in}
 Important here is the observation that these shifts preserve the
 interior parts of the filling, so that, in fact, any L${\mathbb
 R}$-filling is a scalar shift away from one with all non-negative
 parts.

 Before proceeding, we will need the idea of an {\em initial segment}
of an $i$-strip of a L${\mathbb R}$-filling. Let us suppose, given
some L${\mathbb R}$-filling $\{ k_{ij} \}$ of a matrix pair that for
some $i$, the interior {\em and also} the edge parts are
non-negative. Given such an $i$-strip ${\cal S}= \lambda /
\lambda'$, say, we say the skew shape ${\cal S}'$ is an {\em initial
segment} of ${\cal S}$ if there is a row $j$ of the shape such that
the parts of ${\cal S}$ and ${\cal S'}$ are identical in all rows
below $j$, that ${\cal S'}$ is empty above row $j$, and the part of
row $j$ that appears in ${\cal S}'$ has length less than or equal to
that appearing in ${\cal S}$ in row $j$.

 \begin{lem}[Ordering Lemma]
Let $D_{\mu}$ and $D_{\widehat{\nu}}$ be as above, and let $L$ be a
$\mu$-$\widehat{\nu}$-generic matrix.  Let $\{ k_{ij} \}$ be the
associated right L${\mathbb R}$ filling of $\lambda / \mu$ with
content $\nu$ (where $\lambda = inv(D_{\mu}LD_{\widehat{\nu}}$).  We
assume all $k_{ij} \geq 0$.  Then, for each $i$, $1 \leq i \leq r$,
if $ 0 < \alpha < \beta \leq \nu_{i}$, the horizontal strip
\[ inv(\mu L \nu_1 \dots , \nu_{i-1}, \alpha ) / inv(\mu L \nu_1 ,\dots , \nu_{i-1}) \] is an initial
segment of the horizontal strip
\[ inv(\mu L \nu_1 , \dots , \nu_{i-1}, \beta) / inv(\mu L \nu_{1} ,
 \dots , \nu_{i-1}). \]
 That is, each horizontal $i$-strip of an L${\mathbb R}$ filling
grows from the bottom row to the top, and from the left-most to the
right, as $\alpha$ grows from $0$ to $\nu_{i}$.  The same will be
true for the left filling $\{ m_{ij} \}$ of $(D_{\mu},
LD_{\widehat{\nu}})$ of $\lambda / \nu$ with content $\mu$.

\end{lem}

The Ordering Lemma, to be proved below, says something
 even more precise about the dynamics of L${\mathbb R}$-fillings
 when we alter specific values of $\nu_{i}$ appearing in
 $D_{\widehat{\nu}}$ (or values of $\mu_j$ in $D_{\mu}$).   Here, we will dispense with presenting the
 matrix calculations, and display our results in terms of the
 L${\mathbb R}$-diagrams generated by our matrix results.

The \LR\ diagram associated to the pair $(D_{\mu}, L
D_{\widehat{\nu}})$ above is just:

\vspace{0.3in} \psset{unit=0.7cm} \hspace{1in}
\begin{pspicture}(5,4) \psframe(0,0)(1,1) \psframe(0,1)(1,2)
\psframe(0,2)(1,3) \psframe(0,3)(1,4)

\psframe(1,0)(2,1) \psframe(1,1)(2,2)
\rput{*0}(1.5,0.5){$1$}\psframe(1,2)(2,3) \psframe(1,3)(2,4)

\psframe(2,0)(3,1) \psframe(2,1)(3,2)\rput{*0}(2.5,0.5){$2$}
\psframe(2,2)(3,3)\rput{*0}(2.5,1.5){$1$} \psframe(2,3)(3,4)

\psframe(3,0)(4,1) \psframe(3,1)(4,2)\rput{*0}(3.5,0.5){$3$}
\psframe(3,2)(4,3) \rput{*0}(3.5,1.5){$1$} \psframe(3,3)(4,4)

\psframe(4,0)(5,1) \psframe(4,1)(5,2)
\rput{*0}(4.5,0.5){$4$}\psframe(4,2)(5,3) \rput{*0}(4.5,
1.5){$2$}\psframe(4,3)(5,4)

\psframe(5,0)(6,1) \psframe(5,1)(6,2)
\rput{*0}(5.5,0.5){$4$}\psframe(5,2)(6,3)
\psframe(5,3)(6,4)\rput{*0}(5.5,1.5){$2$} \rput{*0}(5.5,2.5){$1$}

\psframe(6,1)(7,2) \psframe(6,2)(7,3) \psframe(6,3)(7,4)
\rput{*0}(6.5,1.5){$3$} \rput{*0}(6.5,2.5){$1$}

\psframe(7,1)(8,2) \psframe(7,2)(8,3) \psframe(7,3)(8,4)
\rput{*0}(7.5,1.5){$3$} \rput{*0}(7.5,2.5){$2$}

\psframe(8,2)(9,3) \psframe(8,3)(9,4) \rput{*0}(8.5,2.5){$2$}

\psframe(9,2)(10,3) \psframe(9,3)(10,4)  \rput{*0}(9.5,2.5){$2$}
\rput{*0}(9.5,3.5){$1$}

 \psframe(10,3)(11,4) \rput{*0}(10.5,3.5){$1$}

  \psframe(11,3)(12,4)\rput{*0}(11.5,3.5){$1$}

   \psframe(12,3)(13,4)\rput{*0}(12.5,3.5){$1$}

    \psframe(13,3)(14,4)\rput{*0}(13.5,3.5){$1$}

     \psframe(14,3)(15,4)\rput{*0}(14.5,3.5){$1$}
\end{pspicture}

\vspace{0.2in} which, as shown above, we now represent with
L${\mathbb R}$-diagram:

\vspace{0.3in} \psset{unit=0.7cm} \hspace{1in}
\begin{pspicture}(5,4)\psframe(0,0)(1,1)
\psframe(1,0)(2,1)\rput{*0}(1.5,0.5){$k_{14}$}
\rput{*0}(2.5,0.5){$k_{24}$} \rput{*0}(3.5,0.5){$k_{34}$}
\psframe(2,0)(3,1) \psframe(3,0)(4,1) \psframe(2,1)(4,2)
\rput{*0}(3,1.5){$k_{13}$} \rput{*0}(0.5,0.5){$\mu_4$}
 \psframe(0,1)(2,2) \psframe(0,2)(5,3)
\rput{*0}(1,1.5){$\mu_{3}$} \rput{*0}(2.5,2.5){$\mu_2$}
\psframe(0,3)(9,4) \rput{*0}(4.5,3.5){$\mu_1$} \psframe(9,3)(15,4)
\rput{*0}(12,3.5){$k_{11}$} \psframe(5,2)(7,3)
\rput{*0}(6,2.5){$k_{12}$} \psframe(7,2)(10,3)
\rput{*0}(8.5,2.5){$k_{22}$} \psframe(4,0)(6,1)
\rput{*0}(5,0.5){$k_{44}$} \psframe(4,1)(6,2)
\rput{*0}(5,1.5){$k_{23}$} \psframe(6,1)(8,2)
\rput{*0}(7,1.5){$k_{33}$}
\end{pspicture}

\vspace{0.2in}

By Corollary~\ref{same-same}, we know that the parts $\{ k_{ij} \}$
of the right filling of the pair $(D_{\mu} , LD_{\widehat{\nu}}) =
(\mu L, 11, 6, 3,2)$ are the same, for $1 \leq i \leq \kappa$, as
that of $(\mu L \nu_1, \dots, \nu_{\kappa}, 0 \dots 0).$  Thus, the
diagram, using $L$ as above, for $(\mu L \nu_1 \nu_2)= (\mu L,  11 ,
6, 0 , 0)$ would be:

\vspace{0.3in} \psset{unit=0.7cm} \hspace{1in}
\begin{pspicture}(5,4)\psframe(0,0)(1,1)
\psframe(1,0)(2,1)\rput{*0}(1.5,0.5){$k_{14}$}
\rput{*0}(2.5,0.5){$k_{24}$} \psframe(2,0)(3,1)
 \psframe(2,1)(4,2)
\rput{*0}(3,1.5){$k_{13}$} \rput{*0}(0.5,0.5){$\mu_4$}
 \psframe(0,1)(2,2) \psframe(0,2)(5,3)
\rput{*0}(1,1.5){$\mu_{3}$} \rput{*0}(2.5,2.5){$\mu_2$}
\psframe(0,3)(9,4) \rput{*0}(4.5,3.5){$\mu_1$} \psframe(9,3)(15,4)
\rput{*0}(12,3.5){$k_{11}$} \psframe(5,2)(7,3)
\rput{*0}(6,2.5){$k_{12}$} \psframe(7,2)(10,3)
\rput{*0}(8.5,2.5){$k_{22}$}
 \psframe(4,1)(6,2)
\rput{*0}(5,1.5){$k_{23}$}
\end{pspicture}
\vspace{0.2in}

Let us now decrease $\nu_2 = 6$.  By the Ordering Lemma we obtain:

\noindent $(\mu L , 11, \underline{5}, 0 , 0)$:

 \psset{unit=0.7cm} 
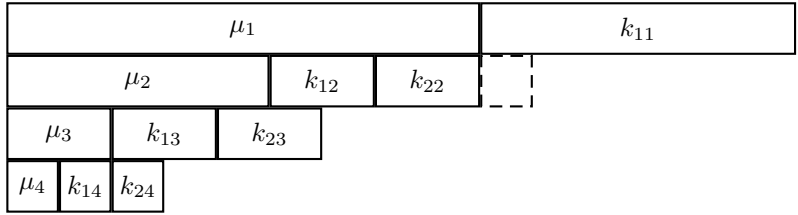
\hspace{1in}
\begin{pspicture}(5,4)\psframe(0,0)(1,1)
\psframe(1,0)(2,1)\rput{*0}(1.5,0.5){$k_{14}$}
\rput{*0}(2.5,0.5){$k_{24}$} \psframe(2,0)(3,1)
 \psframe(2,1)(4,2)
\rput{*0}(3,1.5){$k_{13}$} \rput{*0}(0.5,0.5){$\mu_4$}
 \psframe(0,1)(2,2) \psframe(0,2)(5,3)
\rput{*0}(1,1.5){$\mu_{3}$} \rput{*0}(2.5,2.5){$\mu_2$}
\psframe(0,3)(9,4) \rput{*0}(4.5,3.5){$\mu_1$} \psframe(9,3)(15,4)
\rput{*0}(12,3.5){$k_{11}$} \psframe(5,2)(7,3)
\rput{*0}(6,2.5){$k_{12}$}
\psframe(7,2)(9,3)\psframe[linestyle=dashed](9,2)(10,3)
\rput{*0}(8,2.5){$k_{22}$}
 \psframe(4,1)(6,2)
\rput{*0}(5,1.5){$k_{23}$}
\end{pspicture}
\vspace{0.2in}

\noindent $(\mu L , 11, \underline{4.5}, 0 , 0)$:

 \psset{unit=0.7cm} 
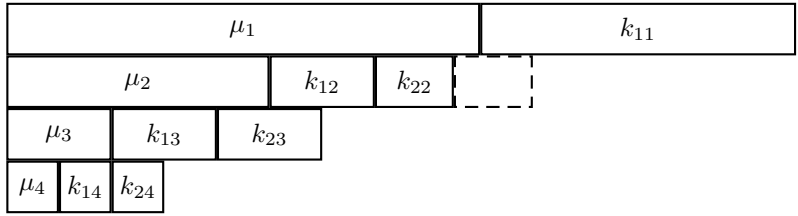
\hspace{1in}
\begin{pspicture}(5,4)\psframe(0,0)(1,1)
\psframe(1,0)(2,1)\rput{*0}(1.5,0.5){$k_{14}$}
\rput{*0}(2.5,0.5){$k_{24}$} \psframe(2,0)(3,1) \psframe(2,1)(4,2)
\rput{*0}(3,1.5){$k_{13}$} \rput{*0}(0.5,0.5){$\mu_4$}
 \psframe(0,1)(2,2) \psframe(0,2)(5,3)
\rput{*0}(1,1.5){$\mu_{3}$} \rput{*0}(2.5,2.5){$\mu_2$}
\psframe(0,3)(9,4) \rput{*0}(4.5,3.5){$\mu_1$} \psframe(9,3)(15,4)
\rput{*0}(12,3.5){$k_{11}$} \psframe(5,2)(7,3)
\rput{*0}(6,2.5){$k_{12}$}
\psframe(7,2)(8.5,3)\psframe[linestyle=dashed](8.5,2)(10,3)
\rput{*0}(7.75,2.5){$k_{22}$}
 \psframe(4,1)(6,2)
\rput{*0}(5,1.5){$k_{23}$}
\end{pspicture}
\vspace{0.2in}

\noindent $(\mu L , 11, \underline{3.2}, 0 , 0)$:

 \psset{unit=0.7cm} 
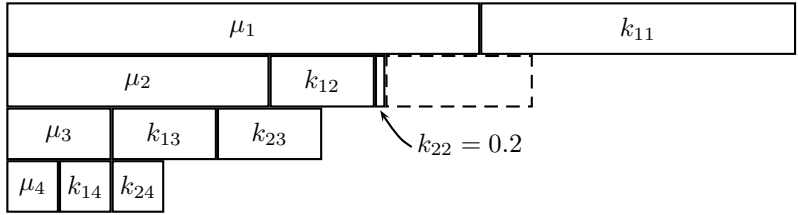
\hspace{1in}
\begin{pspicture}(5,4)\psframe(0,0)(1,1)
\psframe(1,0)(2,1)\rput{*0}(1.5,0.5){$k_{14}$}
\rput{*0}(2.5,0.5){$k_{24}$} \psframe(2,0)(3,1) \psframe(2,1)(4,2)
\rput{*0}(3,1.5){$k_{13}$} \rput{*0}(0.5,0.5){$\mu_4$}
 \psframe(0,1)(2,2) \psframe(0,2)(5,3)
\rput{*0}(1,1.5){$\mu_{3}$} \rput{*0}(2.5,2.5){$\mu_2$}
\psframe(0,3)(9,4) \rput{*0}(4.5,3.5){$\mu_1$} \psframe(9,3)(15,4)
\rput{*0}(12,3.5){$k_{11}$} \psframe(5,2)(7,3)
\rput{*0}(6,2.5){$k_{12}$}
\psframe(7,2)(7.2,3)\psframe[linestyle=dashed](7.2,2)(10,3)
\rput{*0}(8.8,1.3){$k_{22}=
0.2$}\pscurve{->}(7.7,1.25)(7.5,1.4)(7.1,1.95)
 \psframe(4,1)(6,2)
\rput{*0}(5,1.5){$k_{23}$}
\end{pspicture}
\vspace{0.2in}

\noindent $(\mu L , 11, \underline{1.7}, 0 , 0)$:

 \psset{unit=0.7cm} 
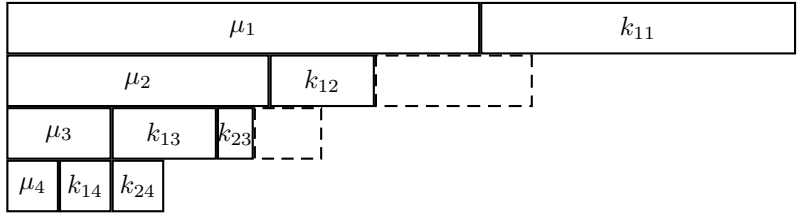
\hspace{1in}
\begin{pspicture}(5,4)\psframe(0,0)(1,1)
\psframe(1,0)(2,1)\rput{*0}(1.5,0.5){$k_{14}$}
\rput{*0}(2.5,0.5){$k_{24}$} \psframe(2,0)(3,1) \psframe(2,1)(4,2)
\rput{*0}(3,1.5){$k_{13}$} \rput{*0}(0.5,0.5){$\mu_4$}
 \psframe(0,1)(2,2) \psframe(0,2)(5,3)
\rput{*0}(1,1.5){$\mu_{3}$} \rput{*0}(2.5,2.5){$\mu_2$}
\psframe(0,3)(9,4) \rput{*0}(4.5,3.5){$\mu_1$} \psframe(9,3)(15,4)
\rput{*0}(12,3.5){$k_{11}$} \psframe(5,2)(7,3)
\rput{*0}(6,2.5){$k_{12}$} \psframe[linestyle=dashed](7,2)(10,3)
 \psframe(4,1)(4.7,2) \psframe[linestyle=dashed](4.7,1)(6,2)
\rput{*0}(4.35,1.5){$k_{23}$}
\end{pspicture}
\vspace{0.2in}

The stability of the filling for parts $k_{ij}$ for $i < \sigma$ as
we decrease the value of $\nu_{\sigma}$ not only relates a
continuously varying family of fillings to each other, but will be a
key component in our proof constructing a bijection between left and
right fillings.  For now, let us continue with the proof of the
Ordering Lemma:

\begin{proof}
We shall prove this result for left fillings, that is, for fillings
of $\lambda / \nu$ with content $\mu$ of the pair $(D_{\mu},
LD_{\widehat{\nu}})$.  The result for right fillings will follow
analogously.  We may assume $\mu = (\mu_1, \mu_2, \dots, \mu_{\ell},
\mu_{\ell + 1}, 0 , \dots, 0)$, and we shall see that as we decrease
the value of $\mu_{\ell+1}$, keeping $L$ and $D_{\widehat{\nu}}$
above fixed, that the $(\ell+1)$-strip of the left filling deforms
according to the statement of the Lemma.

Suppose that the left filling of the pair $(D_{\mu},
LD_{\widehat{\nu}})$ has parts $\{ m_{ij}\}$.   Let $\mu(\ell+1;
\beta) = (\mu_1, \mu_2, \dots , \mu_{\ell}, \beta, 0 , \dots, 0)$,
where $\beta$ satisfies $0 < \beta \leq \mu_{\ell + 1}$.  Let $\{
m_{ij}^* \}$ be the left filling of the pair $(D_{\mu(\ell+1;
\beta)},LD_{\widehat{\nu}})$. By Corollary~\ref{same-same} we know
that $m_{ij} = m_{ij}^*$ for all $i$, $1 \leq i \leq \ell$, and all
$j$. Thus, to track the changes to the $(\ell +1)$-strip, it will be
sufficient to compute the invariant partition $inv(D_{\mu(\ell+1;
\beta)}LD_{\widehat{\nu}})$ for different choices of $\beta$, since
the rest of the left filling remains fixed during the deformation.

Let us begin by noting that
\[ inv( D_{\mu}LD_{\widehat{\nu}}) = inv(D_{\mu}N),
\] where we may assume $N$ is a $\mu$-generic matrix in the same orbit as
the pair $(D_{\mu}, LD_{\widehat{\nu}})$.  In particular, $N$ is
upper triangular, and by Lemma~\ref{row-ineq}, applied to the
$\mu$-generic matrix $N$, the orders of entries in $N$ increase as
we proceed {\em down} any column, while the orders in the product
$D_{\mu}N$ increase as we proceed {\em up} any column. Further,
orders of entries increase as we proceed to the {\em left} in any
row. The upshot of this is that the product:
\begin{align*} D_{\mu}N
& =
\begin{bmatrix} t^{\mu_1} a_{11} & t^{\mu_1} a_{12} &
\dots & \dots&\dots& \dots& \dots & t^{\mu_1}a_{1r} \\
0 & t^{\mu_2}a_{22} & t^{\mu_2}a_{23} & &  & & & \vdots \\
\vdots & 0& \ddots & \ddots & & & & \vdots \\
\vdots & &\ddots & t^{\mu_{\ell}} a_{\ell, \ell} &
t^{\mu_{\ell}}a_{\ell, \ell + 1} & \dots & \dots &
t^{\mu_{\ell}}a_{\ell, r} \\
\vdots & & & \ddots & t^{\mu_{\ell + 1}}a_{\ell + 1, \ell + 1} &
t^{\mu_{\ell + 1}}a_{\ell + 1, \ell + 2} & \dots & t^{\mu_{\ell +
1}}a_{\ell + 1, r}
\\
\vdots & &&& \ddots & a_{\ell + 2, \ell + 2} & \dots & a_{\ell + 2,
r} \\
\vdots & & & & & \ddots & \ddots & \vdots \\
\vdots & & & & & \ddots & a_{r-1, r-1} & a_{r-1,r}\\
 0 & \dots & \dots    &\dots  &\dots  &\dots  &0 & a_{rr} \end{bmatrix}
 \begin{array}{c}    \Uparrow \\ \hbox{Orders} \\
 \hbox{Increase} \\
 \Uparrow \end{array} \\ & \hspace{2.4in} \hbox{$\Leftarrow$ Orders Increase
 $\Leftarrow$} \end{align*}

 is row-equivalent to the diagonal matrix:
\[ D_{\mu}N = \begin{bmatrix} t^{\mu_1} a_{11} & 0 &
\dots & \dots&\dots& \dots& \dots & \dots & 0 \\
0 & t^{\mu_2}a_{22} & 0 & &  & & && \vdots \\
\vdots & 0& \ddots & \ddots & & & && \vdots \\
\vdots & &\ddots & t^{\mu_{\ell}} a_{\ell, \ell} & 0 & \dots & \dots
& \dots & 0 \\
\vdots & & & \ddots & t^{\mu_{\ell + 1}}a_{\ell + 1, \ell + 1} & 0&
\dots & \dots & 0
\\
\vdots & &&& \ddots & a_{\ell + 2, \ell + 2} & 0 & \dots &0 \\
\vdots & & & & & \ddots & \ddots &\ddots & \vdots \\
\vdots & & & & & \ddots & &a_{r-1, r-1} &0\\
 0 & \dots & \dots    &\dots  &\dots  &\dots & &0 & a_{rr} \end{bmatrix}. \]

It is a consequence of Corollary~\ref{row-inv} that $\| a_{ii} \|
+\mu_{i}= \lambda_{i}$, so orders of entries in the above decrease
as we proceed down the diagonal.

Let us calculate, then, the invariant partition of the product
$D_{\mu(\ell+1; \beta)}N$:
\[ D_{\mu(\ell+1; \beta)}N = \begin{bmatrix} t^{\mu_1} a_{11} &
t^{\mu_1} a_{12} &
\dots & \dots&\dots& \dots& \dots & \dots & t^{\mu_1}a_{1r} \\
0 & t^{\mu_2}a_{22} & t^{\mu_2}a_{23} & &  & & & &\vdots \\
\vdots & 0& \ddots & \ddots & & & & & \vdots \\
\vdots & &\ddots & t^{\mu_{\ell}} a_{\ell, \ell} &
t^{\mu_{\ell}}a_{\ell, \ell + 1} & \dots & \dots & \dots &
t^{\mu_{\ell}}a_{\ell, r} \\
\vdots & & & \ddots &  \boldsymbol{t^{\boldsymbol\beta}}a_{\ell + 1,
\ell + 1} &  \boldsymbol{t^{\boldsymbol\beta}}a_{\ell + 1, \ell + 2}
& \dots &  \dots & \boldsymbol{t^{\boldsymbol\beta}}a_{\ell + 1, r}
\\
\vdots & &&& \ddots & a_{\ell + 2, \ell + 2} & a_{\ell + 2, \ell +
3} & \dots & a_{\ell + 2,
r} \\
\vdots & & & & & \ddots & \ddots &  & \vdots \\
\vdots & & & & & \ddots && a_{r-1, r-1} & a_{r-1,r}\\
 0 & \dots & \dots    &\dots  &\dots  &\dots  &&0 & a_{rr} \end{bmatrix}. \]

Since $D_{\mu(\ell+1; \beta)}N$ differs from $D_{\mu}N$ only in row
$\ell + 1$, and since $\beta \leq \mu_{\ell + 1}$, we may still
perform most of the row operations to $D_{\mu(\ell+1; \beta)}N$ that
we used to simplify $D_{\mu}N$.  We may conclude $D_{\mu(\ell+1;
\beta)}N$ is row-equivalent to the matrix $P(\beta)$, where:

\[ P(\beta) = \begin{bmatrix} t^{\mu_1} a_{11} &
0 &
\dots & \dots&\dots& \dots& \dots & \dots &0 \\
0 & t^{\mu_2}a_{22} & 0 & &  & & & &\vdots \\
\vdots & 0& \ddots & \ddots & & & & & \vdots \\
\vdots & &\ddots & t^{\mu_{\ell}} a_{\ell, \ell} & 0 & 0 & \dots &
\dots &
0 \\
\vdots & & & \ddots &  \boldsymbol{t^{\boldsymbol\beta}}a_{\ell + 1,
\ell + 1} &  \boldsymbol{t^{\boldsymbol\beta}}a_{\ell + 1, \ell + 2}
& \dots &  \dots & \boldsymbol{t^{\boldsymbol\beta}}a_{\ell + 1, r}
\\
\vdots & &&& \ddots & a_{\ell + 2, \ell + 2} & 0 & \dots & 0 \\
\vdots & & & & & \ddots & \ddots & \ddots & \vdots \\
\vdots & & & & & \ddots && a_{r-1, r-1} & 0\\
 0 & \dots & \dots    &\dots  &\dots  &\dots  &&0 & a_{rr} \end{bmatrix}. \]
Consequently, $inv(D_{\mu(\ell+1; \beta)}N) = inv(P(\beta))$.  The
proof of the lemma is thus reduced to calculating the invariant
partitions of the matrices $P(\beta)$ as $\beta$ decreases from
$\mu_{\ell +1}$ to $0$. That is, we must prove there is an
increasing sequence of row indices $j_0 < j_1 < \cdots $ such that
as $\beta$ decreases from $\mu_{\ell + 1}$ to $0$, the shape of
$inv(P(\beta))$ decreases first only in row $j_0$, and then only in
row $j_1$, etc., until we reach $inv(P(0))=inv(\mu_{\ell}, \dots ,
\mu_1 N)$.

To accomplish this, let us first set $\beta_{0} = \mu_{\ell + 1}$
and $j_{0} = \ell + 1$. Then, let us define, iteratively, for $i =
1, 2, \dots$
\[\beta_i = \max_{j > j_{i-1}}\{ \beta: \beta = \| a_{j, j} \| - \|
a_{\ell + 1, j} \|\},
\] and
\[ j_i = \max_{j > j_{i-1}} \{ j :  \beta_i = \|
a_{j,j} \| -\| a_{\ell + 1,j} \| \}. \]

The meaning of the above definitions is the following.  When $\beta
= \mu_{\ell +1}$, the orders of the entries $t^{\beta}a_{\ell + 1,
j}$ are all greater than or equal to the orders of the diagonal
entries $a_{j,j}$ lying below them on the diagonal of $P(\beta)$, by
Lemma~\ref{row-ineq}, since $N$ is $\mu$-generic. If we decrease
$\beta$, so long as the off-diagonal entries in row $\ell + 1$ have
orders greater than the diagonal entries below them, we may clear
these entries in row $\ell +1$ using the diagonal entries below
them, and we see the invariant partition of $P(\beta)$ only changes
in row $\ell + 1$. As we further decrease $\beta$, we will set
$\beta_1$ to be the first value at which two entries in the same
column, $t^{\beta_1}a_{\ell+1, j}$ and $a_{j,j}$,have the same
order. We then set $j_1$ to be the right-most column at which this
occurs when $\beta = \beta_1$. This means, when $\beta = \beta_1$,
in all columns to the right of $j_1$ the order of the diagonal entry
in that column is {\em strictly less than} the order of the entry
lying above it in row $\ell + 1$.  We then imagine decreasing
$\beta$ smaller than $\beta_1$, and set $\beta_2$ as the next value
at which, in some column to the right of $j_1$, we have $\beta_2+ \|
a_{\ell + 1, j} \| = \|a_{j,j} \|$, and set $j_2$ as the right-most
such column. We then continue until $\beta$ has reached 0 or we have
run out of columns.

Since, by hypothesis, $\mu_{i} = 0$ for $i \geq \ell + 1$, by
Lemma~\ref{row-ineq} (applied to the $\mu$-generic matrix $N$) we
have $\| a_{j,j} \| = \| a_{\ell + 1 + \kappa,j} \|$ for all $\kappa
\geq 1$ and $1 \leq j \leq \ell + \kappa$. Thus
\[ \| a_{j,j} \| - \| a_{\ell + 1, j } \| \leq \| a_{\ell + 2, j} \|
- \| a_{\ell + 1, j} \| \leq \mu_{\ell + 1} = \beta_{0}. \]

Thus, it is possible that $\beta_{1} = \beta_{0}$, but from then on,
as the column indices move to the right, we must have
\[ \beta_{i+1} < \beta_i, \quad i \geq 1. \]

Note that, in particular, for any integer $\kappa
>0$ we have $j_{i} < j_{i + \kappa}$, but
\[ \beta_i + \|a_{\ell + 1}, j_{i+ \kappa} \| >
\| a_{j_{i+ \kappa},j_{i+ \kappa}} \|, \] and
\[ \beta_{i + \kappa} + \|a_{\ell + 1}, j_{i} \| <
\| a_{j_{i},j_{i}} \|. \] In other words, as we decrease $\beta$,
when we first reach $\beta_i$, the gap between $\beta_i + \|a_{\ell
+ 1}, j_{i} \|$ and $\| a_{j_{i},j_{i}} \|$ has decreased to $0$, by
the definitions of $j_i$ and $\beta_i$, but  $\beta_i + \|a_{\ell +
1}, j_{i+ \kappa} \|$ is still {\em greater} than $\| a_{j_{i+
\kappa},j_{i+ \kappa}} \|$.  On the other hand, once $\beta$ has
decreased past $\beta_i$ and has reached $\beta_{i + \kappa}$, then
$\beta_i + \|a_{\ell + 1}, j_{i} \|$ has become less than $\|
a_{j_{i},j_{i}} \|$.

{\bf Claim:}  Let us set $j_{0} = \ell + 1$ and $\beta_0 = \mu_{\ell
+ 1}$.  Then for all $\kappa \geq 0$, if we decrease $\beta$ on the
interval $\beta_{\kappa+1} < \beta \leq \beta_{\kappa}$, the
$\mu_{\ell + 1}$ strip decreases {\em only} in row $j_{\kappa}$.

Since $\ell + 1 = j_0 \geq j_1 > j_2 > \dots$, proving the claim
will complete the proof of the lemma.

{\bf Proof of the Claim:} As shown above, for $\beta$ between
$\mu_{\ell + 1}$ and $\beta_1$, the $\mu_{\ell +1}$-strip only
decreases in row $\ell + 1$.  Let us fix an index $\kappa$, and a
choice of $\beta$ such that $\beta_{\kappa+1} < \beta \leq
\beta_{\kappa}$.  We shall calculate the invariant partition of
$P(\beta)$ in this case. To do so we will reduce the matrix
$P(\beta)$, inductively, on the column indices $j_1, j_2, \dots,
j_{\kappa}$, and then perform a simple operation on the remaining
columns, so that the final result is a matrix in diagonal form. We
shall suppose that, for some $i$, where $j_0 \leq j_i \leq
j_{\kappa}$ that $P(\beta)$ is equivalent to an upper-triangular
matrix $P(\beta)^{(i)}$ such that:

\begin{enumerate}
\item  In $P(\beta)^{(i)}$, above row $j_{i}$, the matrix is diagonal,
where each entry above row $j_{i}$ is either: \begin{enumerate}
\item The {\em same} as in $P(\beta)$, if it occurs in some row $s$, where
$s \neq j_{p}$ for any $p < i$, or \item where the
$(j_{p},j_{p})$-entry, for $j_{p} < j_{i}$, has order \[ \| a_{\ell
+ 1, j_{p}} \| + \beta_{j_{(p+1)}}. \] \end{enumerate} \item In rows
below $j_{i}$, the matrix is the same as $P(\beta)$. \item Lastly,
the $(j_{i}, s)$ entry, for $ s \geq j_{i}$, is $t^{\beta}a_{\ell +
1, s}$.  In particular, in rows $j_{i}$ through $j_{i+1}$, and
columns $j_{i}$ through $j_{\kappa}$, the matrix $P(\beta)^{(i)}$
has the form:

\[ \begin{matrix} t^{\beta}a_{\ell +1, j_{i}} & t^{\beta}
a_{\ell + 1, j_{i}+1} & t^{\beta} a_{\ell + 1, j_{i}+2}& \dots &
t^{\beta} a_{\ell + 1, (j_{(i+1)}-1)} & t^{\beta}a_{\ell + 1,
j_{(i+1)}} &t^{\beta}a_{\ell + 1, j_{(i+1)+1}}   & \dots &
t^{\beta}a_{\ell + 1, j_{\kappa}}\\
0 & a_{(j_{i}+1), (j_{i}+1)} & 0 & \dots & 0 & 0 & 0 &\dots & 0 \\
\vdots  & \ddots & a_{(j_{i}+2), (j_{i}+2)} & \ddots & \vdots &
\vdots  & && \vdots  \\
\vdots  && \ddots & \ddots & 0 & \vdots& & &\vdots \\
 \vdots & & &\ddots  &a_{(j_{(i+1)}-1, j_{(i+1)}-1)} & 0  & \dots &
  \dots &0\\
0 & \dots & \dots  & \dots &0& a_{(j_{(i+1)},j_{(i+1)})} & 0 & \dots
& 0
\end{matrix}
\]

\end{enumerate}

In the above, the top row entries are indeed correctly labeled
$t^{\beta}a_{\ell + 1, j}$, though we are assuming, for purposes of
induction, that they are located in row $j_{i}$ of the matrix
$P(\beta)^{(i)}$. In particular, $P(\beta)^{(0)} = P(\beta)$.  Our
goal is to reduce this matrix to produce the form
$P(\beta)^{(i+1)}$.

  Note that, by hypothesis on each $\beta_{p}$, that
\[ \| a_{j_{p}, j_p} \|  = \| a_{\ell + 1, j_p} \| + \beta_p, \]
but for all $p$ such that $1 \leq p < \kappa$ we have
\[ \| a_{j_{p}, j_p} \|  = \| a_{\ell + 1, j_p} \| + \beta_p >
\| a_{\ell + 1, j_p} \| + \beta_{\kappa} \geq \| a_{\ell + 1, j_p}
\| + \beta. \]

Since the orders of entries $a_{ij}$ in the $\mu$-generic matrix $N$
increase as we proceed to the left in any row, $P(\beta)^{(i)}$ is
equivalent to one whose entries in columns $j_{i}$ through
$j_{\kappa}$ look like:

\[ \begin{matrix} 0 &0 & \dots & \dots &
0 & t^{\beta}a_{\ell + 1, j_{(i+1)}} &  t^{\beta}a_{\ell + 1,
j_{(i+1)+1}}  & \dots &
t^{\beta}a_{\ell + 1, j_{\kappa}}\\
0 & a_{(j_{i}+1), (j_{i}+1)} & 0 & \dots & 0 & 0 &0 &\dots & 0 \\
\vdots  & \ddots & a_{(j_{i}+2), (j_{i}+2)} & \ddots & \vdots &
\vdots  & && \vdots  \\
\vdots  && \ddots & \ddots & 0 & \vdots& & &\vdots \\
 \vdots & & &\ddots  &a_{(j_{(i+1)}-1, j_{(i+1)}-1)} & 0  & \dots &
  \dots &0\\
q_{j_{(i+1)},j_{i}} & q_{j_{(i+1)},j_{i}+1} & \dots  & \dots
&q_{j_{(i+1)},j_{(i+1)}-1}& 0 & q_{j_{(i+1)},j_{(i+1)}+1} & \dots &
q_{j_{(i+1)},j_{\kappa}}
\end{matrix}
\]
where this is obtained by first adding a multiple of row $j_{i}$
(whose entries were of the form: $t^{\beta}a_{\ell + 1, s}$ to row
$j_{(i+1)}$ to put a $0$ in the $(j_{(i+1)}, j_{(i+1)})$ entry, and
then adding multiples of column $j_{(i+1)}$ to the columns to the
left to put zeros in row $j_{i}$.

Now, for each column $p$, for $j_{i}+1 \leq p \leq (j_{(i+1)} -1)$,
we must have (by the definition of $\beta_{(i+1)}$) either:
\[ \| a_{p, p}  \| \leq \| t^{\beta} a_{\ell + 1, p}
\|, \] in which case, $q_{j_{(i+1)}, p}$, which is a multiple of
$t^{\beta} a_{\ell + 1, p}$,  will have order greater than that of
$a_{p, p}$.  Thus in this case, we may add a multiple of row $p$
(whose only non-zero entry is in column $p$) to row $j_{(i+1)}$ to
cancel the entry $q_{j_{(i+1)}, p}$.  Otherwise, since by
construction there are no columns $j_{s}$ between $j_i$ and
$j_{(i+1)}$, and since $\beta_{\kappa +1} < \beta \leq
\beta_{\kappa} < \beta_{(i+1)}$, the only way an entry in row $\ell
+ 1$ between columns $j_{i}$ and $j_{(i+1)}$ can have its order less
than the diagonal entry below it is if it reaches that value {\em
simultaneously} with column $j_{(i+1)}$ (and $j_{(i+1)}$ is the {\em
right-most} column at which the values equaled).  That is, we must
have had
\[ \|a_{p,p} \| - \| t^{\beta_{(i+1)} }a_{\ell + 1, p} \| =
\|a_{j_{(i+1)},j_{(i+1)}} \| - \| t^{\beta_{(i+1)} }a_{\ell + 1,
j_{(i+1)}} \|. \] But then
\[ \| a_{p,p} \| - \| t^{\beta}a_{\ell + 1, p} \| = \|
a_{j_{(i+1)},j_{(i+1)}} \| - \| t^{\beta} a_{\ell +1, j_{(i+1)}} \|.
\]
When clearing the entry $a_{j_{(i+1)},j_{(i+1)}}$, we multiplied row
$\ell +1$ by a multiple of order
\[ \|a_{j_{(i+1)},j_{(i+1)}} \| - \| a_{\ell + 1, j_{(i+1)}} \|. \]
Consequently, the entry $q_{j_{(i+1)},p}$, which is the image of the
operation applied to column $p$, must have order: \begin{align*}
 \|q_{j_{(i+1)},p} \| & = \| t^{\beta}a_{\ell + 1, p} \| + \|
a_{j_{(i+1)},j_{(i+1)}} \| - \| t^{\beta} a_{\ell + 1, j_{(i+1)}} \|
\\
& = \| t^{\beta}a_{\ell + 1, p} \| + \| a_{p,p} \| - \|
t^{\beta}a_{\ell + 1, p} \|  \\
 & =\| a_{p,p} \|, \end{align*}
 so that again we may clear the entry $q_{j_{(i+1)},p}$ using the
diagonal entry $a_{pp}$ in row $p$.

In columns $p > j_{(i+1)}$, by the definition of $\beta_{(i+1)}$, we
again have either
\[ \| a_{p,p} \| \leq t^{\beta} a_{\ell + 1, p}, \]
in which case we may clear entries $q_{j_{(i+1)}, p}$ below this
term.

If, however,
\[ \| a_{pp} \| > \| t^{\beta} a_{\ell +1, p} \|, \]
where column $p$ is to the {\em right} of column $j_{(i+1)}$, then
we must have, by the definition of $\beta_{(i+1)}$,
\[ \| a_{j_{(i+1)},j_{(i+1)}} \| - \| t^{\beta} a_{\ell + 1,
j_{(i+1)}} \| > \| a_{p,p} \| - \| t^{\beta} a_{\ell + 1, p} \| , \]
that is,
\[ \| a_{j_{(i+1)},j_{(i+1)}} \|  + \| a_{\ell + 1, p} \|- \| a_{\ell + 1,
j_{(i+1)}} \| > \| a_{p,p} \| .\]

 But then, since \begin{align*} \| q_{j_{(i+1)}, p} \| & =
\|a_{j_{(i+1)},j_{(i+1)}} | - \beta - \| a_{\ell + 1, j_{(i+1)}} \|
+
\beta + \| a_{\ell + 1, p} \| \\
& =\|a_{j_{(i+1)},j_{(i+1)}} |  - \| a_{\ell + 1, j_{(i+1)}} \| +
 \| a_{\ell + 1, p} \| \\
 & > \| a_{p,p} \|,
\end{align*}
we may, again, add a multiple of row $p$ to row $j_{(i+1)}$ to clear
the entry in column $p$ of that row.

In column $j_{i}$, we see \begin{align*} \| q_{ j_{(i+1)},j_{1}} \|
& = \beta + \| a_{\ell + 1, j_{i}} \| + \| a_{j_{(i+1)}. j_{(i+1)}}
\| - \beta - \| a_{\ell + 1, j_{(i+1)}} \| \\ & =  \| a_{\ell + 1,
j_{i}} \| + \| a_{j_{(i+1)}. j_{(i+1)}} \| - \| a_{\ell + 1,
j_{(i+1)}} \| \\  & = \| a_{\ell + 1, j_{i}} \| + \beta_{(i+1)}.
\end{align*} which is independent of $\beta$. After switching rows
$j_{i}$ and $j_{(i+1)}$ we see we have constructed a matrix,
equivalent to $P(\beta)^{(i)}$, which satisfies the conditions
(1),(2), and (3) above, so we may re-name it $P(\beta)^{(i+1)}$.

We repeat this reduction in all columns until we reach
$P(\beta)^{(\kappa)}$.  Since $\beta_{\kappa + 1} < \beta \leq
\beta_{\kappa}$, the entries $a_{\ell + 1, p}$ for $p
> j_{\kappa}$ must have order {\em greater than} or equal to the diagonal
entries below them, so that $P(\beta)^{(\kappa)}$ is row equivalent
to a diagonal matrix.  Thus, for all $\beta$ such that
$\beta_{\kappa + 1} < \beta \leq \beta_{\kappa}$, $P(\beta)$ is
equivalent to a diagonal matrix whose entries are:

\begin{enumerate}
\item $t^{\mu_{i}}a_{ii}$, in rows $1 \leq i \leq \ell$.
\item $s_{\ell+1,\ell+1}$ in row $\ell +1$, whose order is
$\|a_{\ell + 1,\ell + 1} \| + \beta_1 $.
\item In rows $\ell + 2$ through $j_{\kappa} -1$, the entry is
either \begin{enumerate}
\item $a_{\tau, \tau}$, if $\tau \neq j_{p}$ for any $p$, $1 \leq p < \kappa$,
or
\item $s_{j_{p},j_{p}}$, whose order is $\| a_{\ell + 1, j_{p}} \| +
\beta_{j_{(p+1)}} $.
\end{enumerate}
\item In row $j_{\kappa}$, the entry is $t^{\beta} a_{\ell + 1,
j_{\kappa}}$.
\item In all rows $\tau$ below $j_{\kappa}$, the entry is $a_{\tau,
\tau}$.
\end{enumerate}

 We claim that the orders
of these entries are in decreasing order along the diagonal. Note
that when $\beta = \beta_{i}$, by the definition of $\beta_{i}$ we
have
\[ \| t^{\beta} a_{\ell + 1, j_{i}} \| = \| t^{\beta_{i}} a_{\ell + 1, j_{i}} \| =
\| a_{j_{i}j_{i} } \|, \] so that, in the diagonal matrix above,
initially the size of the invariant partition of $P(\beta)$ in row
$j_{i}$ is unchanged.  It then decreases until $\beta =
\beta_{i+1}$.  Note that we must have, for all $\beta > \beta_{i+1},
$
\[ \| t^{\beta} a_{\ell + 1, j_{i}} \| \geq \| a_{(j_{i} +1),
(j_{i}+1)} \|, \] else
\[ \| a_{(j_{i} +1), (j_{i}+1)} \| >  \| t^{\beta} a_{\ell + 1, j_{i}} \| \geq  \|
 t^{\beta} a_{\ell + 1, (j_{i}+ 1)} \|, \] which would imply there
existed some value of $\beta > \beta_{i+1}$ with $\| a_{(j_{i} +1),
(j_{i}+1)} \| = \| t^{\beta} a_{\ell + 1, (j_{i}+ 1)} \|$,
contradicting the definition of $\beta_{i+1}$.  Thus, for all values
of $\beta$ such that $\beta_{i+1} < \beta \leq \beta_{i}$ we have
\[ \| a_{(j_{i} +1),
(j_{i}+1)} \| \leq \| t^{\beta} a_{\ell + 1, j_{i}} \|  \leq \|
a_{j_{i}j_{i} } \|. \]

 Hence, for $\beta$ such that $\beta_{i+1} <
\beta \leq \beta_{i}$, the size of the invariant partition
$inv(P(\beta))$ decreases only in row $j_{i}$, with all other
constant. Consequently, for $\beta$ between $\beta_{0} = \mu_{\ell +
1}$ and $\beta_1$, the invariant partition only decreases in row
$j_{0} = \ell +1$.  After this, the size in row $\ell + 1$ is fixed
until $\beta = \beta_{1}$, and the partition now decreases in row
$j_{1}$ until $\beta = \beta_{2}$, whereupon the partition decreases
in row $j_{2}$, etc.  Since for differing values of $\beta$ we
calculate the filling in the $\mu_{\ell + 1}$-strip, this completes
the proof of the lemma.
\end{proof}

The Ordering Lemma states that, once a pair is in
$\mu$-$\widehat{\nu}$-generic form $(D_{\mu},L D_{\widehat{\nu}})$,
decreasing the last non-zero entry $\mu_{\ell +1}$ in $\mu$ will
{\em not} change any parts of the left filling $m_{ij}$, for $1 \leq
i \leq \ell$.   The Corollary below explains what can be said about
the effect on the {\em right} filling, upon decreasing the value of
$\mu_{\ell + 1}$;

\begin{cor}Suppose $\mu = (\mu_{1}, \dots, \mu_{\ell} ,
\mu_{\ell + 1}, 0, \dots 0)$, and $\nu$ are ${\mathbb
R}$-partitions, with $L$ a $\mu$-$\widehat{\nu}$-generic matrix. Let
$\mu^* =(\mu_{1}, \dots, \mu_{\ell} , \beta, 0, \dots 0)$, where
$\beta$ is chosen so that, according to the Ordering Lemma, the
shape of the invariant partition is changed only in rows $\ell +1$
through row $\kappa$ (by reducing the $\mu_{\ell + 1}$-strip), and
reducing the size of row $\kappa$ so that it is greater than or
equal to the length of row $\kappa +1$ below it in the original
shape.  Then, no part of the $\nu$-filling over this reduced $\mu^*$
that appears in rows below $\kappa$ will be changed (they will
occupy the same row and column location in the new L${\mathbb
R}$-filling). That is, if $\{ k_{ij} \}$ is the $\nu$ filling of the
pair over $\mu$, and $\{ k_{ij}^* \}$ is the $\nu$-filling over
$\mu^*$ (using the same $\mu$-$\widehat{\nu}$-generic matrix $L$),
then $k_{ij} = k_{ij}^*$ for all $j > \kappa$.  \label{don't-move}
\end{cor}

\begin{proof} We will prove this by induction on $\sigma$, the number of non-zero parts
in the ${\mathbb R}$-partition $\nu$.  When $\sigma = 1$, computing
the filling is, given $\mu$, just computing the $1$-strip of the
right filling, which may be obtained by calculating $inv(\mu L
\nu_1)$. Suppose $N_{1}$ is a $\mu$-generic matrix in the same orbit
as this pair, so $inv(\mu L \nu_1) = inv(D_{\mu}N_1)$. Then we may
assume
\[  D_{\mu}N_1 = \begin{bmatrix} t^{\mu_1} a_{11} & t^{\mu_1} a_{12} &
\dots & \dots&\dots& \dots& \dots & \dots & t^{\mu_1}a_{1r} \\
0 & t^{\mu_2}a_{22} & t^{\mu_2}a_{23} & &  & & & &\vdots \\
\vdots & 0& \ddots & \ddots & & & & & \vdots \\
\vdots & &\ddots & t^{\mu_{\ell}} a_{\ell, \ell} &
t^{\mu_{\ell}}a_{\ell, \ell + 1} & \dots & \dots &\dots &
t^{\mu_{\ell}}a_{\ell, r} \\
\vdots & & & \ddots & t^{\mu_{\ell + 1}}a_{\ell + 1, \ell + 1} &
t^{\mu_{\ell + 1}}a_{\ell + 1, \ell + 2} & \dots & \dots &
t^{\mu_{\ell + 1}}a_{\ell + 1, r}
\\
\vdots & &&& \ddots & a_{\ell + 2, \ell + 2} & \dots & \dots &
a_{\ell + 2,
r} \\
\vdots & & & & & \ddots & \ddots & & \vdots \\
\vdots & & & & & &\ddots  & a_{r-1, r-1} &a_{r-1,r}\\
 0 & \dots & \dots    &\dots  &\dots  &\dots  &\dots &0 & a_{rr} \end{bmatrix}. \]
As in the proof of the Ordering Lemma, this matrix is equivalent to:
\[ \begin{bmatrix} t^{\mu_1} a_{11} & 0 &
\dots & \dots&\dots& \dots& \dots &\dots & 0\\
0 & t^{\mu_2}a_{22} & 0 & &  & & & &\vdots \\
\vdots & 0& \ddots & \ddots & & & & &\vdots \\
\vdots & &\ddots & t^{\mu_{\ell}} a_{\ell, \ell} & 0& \dots & \dots &\dots &0 \\
\vdots & & & \ddots & t^{\mu_{\ell + 1}}a_{\ell + 1, \ell + 1} &
t^{\mu_{\ell + 1}}a_{\ell + 1, \ell + 2} & \dots &\dots &
t^{\mu_{\ell + 1}}a_{\ell + 1, r}
\\
\vdots & &&& \ddots & a_{\ell + 2, \ell + 2} & 0 & \dots & 0
 \\
\vdots & & & & & \ddots & \ddots &\ddots & \vdots \\
\vdots & & & & & &\ddots & a_{r-1, r-1} &0\\
 0 & \dots & \dots    &\dots  &\dots &\dots  &\dots  &0 & a_{rr} \end{bmatrix}. \]
In fact, since we know $inv(D_{\mu}N_1)$ will determine the
$1$-strip of a right L${\mathbb R}$-filling, there can be no
non-zero parts of the $1$-strip at or below row $\ell + 3$, and
therefore $\| a_{jj} \| = 0$ for $j \geq \ell + 3$, and so we can
even conclude $D_{\mu}N_1$ is equivalent to:
\[  \begin{bmatrix} t^{\mu_1} a_{11} & 0 &
\dots & \dots&\dots& \dots& \dots &\dots & 0\\
0 & t^{\mu_2}a_{22} & 0 & &  & & & &\vdots \\
\vdots & 0& \ddots & \ddots & & & & &\vdots \\
\vdots & &\ddots & t^{\mu_{\ell}} a_{\ell, \ell} & 0& \dots & \dots &\dots &0 \\
\vdots & & & \ddots & t^{\mu_{\ell + 1}}a_{\ell + 1, \ell + 1} &
t^{\mu_{\ell + 1}}a_{\ell + 1, \ell + 2} &0&\dots & 0
\\
\vdots & &&& \ddots & a_{\ell + 2, \ell + 2} & 0 & \dots & 0
 \\
\vdots & & & & & \ddots & \ddots &\ddots & \vdots \\
\vdots & & & & & &\ddots & a_{r-1, r-1} &0\\
 0 & \dots & \dots    &\dots  &\dots &\dots  &\dots  &0 & a_{rr} \end{bmatrix}. \]
If we replace $\mu_{\ell + 1}$ above with a parameter $\beta$,
calling the resulting matrix $P(\beta)$ as in the Ordering Lemma, we
see that as long as $ \beta  \geq \| a_{\ell + 2, \ell + 2} \|$,
then $\beta + \| a_{\ell + 1, \ell + 2} \| \geq \| a_{\ell + 2, \ell
+ 2} \|$, and so the invariant partition of $P(\beta)$ will remain
unchanged in row $\ell + 2$. Since in this case $\| a_{\ell + 2,
\ell + 2} \|$ is the size of the $1$-strip in row $\ell + 2$, we see
the claim is proved in this case.

Inductively, we will assume the result holds for all parts $\{
k_{ij} \}$ of the $\nu$-filling, for $1 \leq i \leq \sigma$, and
prove this is also true for $\sigma + 1$.

We will calculate the filling of $\mu L \nu_1, \dots, \nu_{\sigma},
\nu_{\sigma + 1}$.  Again, following the proof of the Ordering
Lemma, if we reduce $\mu_{\ell +1}$ so that we decrease the size of
rows of the invariant partition of $D_{\mu}N$ in rows $\ell +1$
through $\kappa$, then overall shape of the rows below row $\kappa$
is unchanged, and the parts $k_{ij}$ for $1 \leq i \leq \sigma$,
$\kappa \leq j \leq r$ are unchanged as well.  Thus, the parts
$k_{\sigma+1, j}$ for $\kappa \leq j \leq r$ cannot change, either.
\end{proof}

We have shown that, given a pair $(M,N)$, we can find in its orbit a
pair $(D_{\mu},N^{*})$, where $D_{\mu}$ is diagonal and $N^{*}$ is
$\mu$-generic from which we can determine a right filling of
$\lambda / \mu$ with content $\nu$.  We are also able to apply the
same matrix operations to $M$ (instead of $N$), yielding a pair
$(M^{*},D_{\nu})$ from which we could determine a left filling
(using $M^{*}$) of $\lambda / \nu$ with content $\mu$. Thus, at the
matrix level, we have a possible mapping from fillings with content
$\nu$ to fillings with content $\mu$ (in their respective skew
shapes). Due to the lack of uniqueness between orbits and fillings,
it is not even clear that we have defined a function, much less a
bijection.

We shall show below that given a matrix pair $(M,N)$ from which we
determine a filling of $\lambda / \mu$ with content $\nu$, the
associated filling (determined by the same pair) of $\lambda / \nu$
with content $\mu$ is, in, fact, independent of the pair, but
depends only on the filling.   In fact, we will show that the
bijection determined by matrices over our valuation rings is the
same as the combinatorially defined bijection between fillings found
in \cite{kerber}.

\begin{thm} Suppose that $(M,N)$ and $(M',N')$ are two matrix pairs
such that $inv(M)=inv(M') = \mu$, $inv(N) = inv(N') = \nu$ and
$inv(MN) = inv(M'N') = \lambda$.  Suppose further that both pairs
yield the same right-filling $\{ k_{ij} \}$ of $\lambda / \mu$ with
content $\nu$.  Then, both pairs yield the same left-filling $\{
m_{ij} \}$ of $\lambda / \nu$ with content $\mu$.
\end{thm}

\begin{proof}   Let us fix a matrix pair, which we may assume is in the form
$(D_{\mu}, LD_{\widehat{\nu}})$ for some ${\mathbb R}$-partitions
$\mu$ and $\nu$, where $L$ is $\mu$-$\widehat{\nu}$-generic.  Let us
suppose $\lambda = inv(D_{\mu} LD_{\widehat{\nu}})$.  We may
calculate from this pair a L${\mathbb R}$-filling of $\lambda / \mu$
with content $\nu$, and also a filling of $\lambda / \nu$ with
content $\mu$.  Since the tableau for a L${\mathbb R}$ filling
determines a skew shape in which a partition is distributed over
another, we shall also refer to a left filling of $\lambda / \nu$
with content $\mu$, determined by the pair $(D_{\mu}, L
D_{\widehat{\nu}})$, as a filling of $\mu$ ``over" $\nu$, since the
content of $\mu$ is to be spread around the fixed partition $\nu$.
Similarly we shall call a right filling of $\lambda / \mu$ with
content $\nu$ a filling of $\nu$ over $\mu$. We will denote the
parts of the right filling of $(M,N)$ of $\nu$ over $\mu$ by $\{
k_{ij} \}$ and the parts of the left filling of $\mu$ over $\nu$ by
$\{ m_{ij} \}$.  In all cases, we will refer to fillings determined
by a fixed $\mu$-$\widehat{\nu}$-generic matrix $L$.

Let us first note that it is sufficient to prove the result for
partitions $\mu$ and $\nu$ such that, given some
$\mu$-$\widehat{\nu}$-generic matrix $L$, {\em all} $k_{ij}$ and
$m_{ij}$ in these fillings are non-negative. This follows since, as
noted in Corollary~\ref{diag-shift}, given any matrix pair $(M,N)$,
there are scalar matrices $D(\beta)$ and $D(\beta)$ such at all
parts of the left and right fillings of the pair
$(D(\beta)M,ND(\alpha))$ are positive. By
Corollary~\ref{fill-biject}, then, the theorem will follow once we
prove it for fillings with non-negative parts.

So, let us suppose that all parts of the left and right fillings for
both pairs $(M,N)$ and $(M',N')$ are non-negative, so that the
Ordering Lemma applies.

Assume, for purposes of induction, that given any ${\mathbb
R}$-partition $\nu$ with $\sigma$-many non-zero parts, and given a
$\nu$-filling $\{ k_{ij} \}$ over some ${\mathbb R}$-partition
$\mu$, that there is a {\em unique} $\mu$-filling $\{ m_{ij}^* \}$
over this $\nu$. Note that this is trivially true when $\sigma = 0$.
We shall prove this result holds for partitions $\nu$ with
$\sigma+1$ many non-zero parts.  We will fix a right filling $\{
k_{ij} \}$ of $(D_{\mu}, L D_{\widehat{\nu}})$, and show that its
left filling $\{ m_{ij} \}$ is uniquely determined.

Given matrices as above, we can represent the shapes of the
$\nu$-filling over $\mu$ below, highlighting the $(\sigma +
1)$-strip (note, the $(\sigma + 1)$-strip need not start in the
bottom row of the diagram):

\vspace{0.2in}

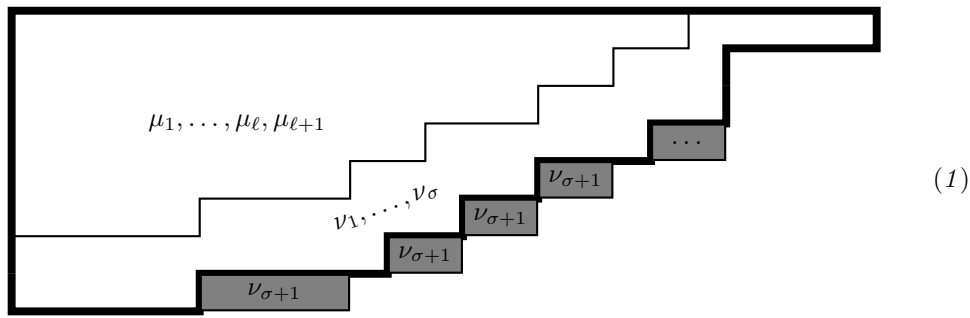
\psset{unit=0.5cm}
\begin{pspicture}(14,8)
\newgray{llgray}{0.92}
\psline[linewidth=3pt](19,6)(19,7)(20,7)(23,7)(23,8)(0,8)(0,1)
\psline(0,2)(5,2)(5,3)(9,3)(9,4)(11,4)(11,5)(14,5)(14,6)(16,6)(16,7)(18,7)(18,8)
\psline[linewidth=3pt](0,1)(0,0)(5,0)(5,1)(10,1)(10,2)(12,2)(12,3)(14,3)(14,4)(17,4)(17,5)(19,5)(19,6)
\rput{*0}(6,5){$\mu_{1}, \dots , \mu_{\ell}, \mu_{\ell +1}$}
\rput{*20}(10,2.75){$\nu_1, \dots , \nu_{\sigma}$}
\psframe[fillstyle=solid,fillcolor=gray](5,0)(9,1)
\rput{*0}(7,0.5){$\nu_{\sigma + 1}$}
\psframe[fillstyle=solid,fillcolor=gray](10,1)(12,2)
\rput{*0}(11,1.5){$\nu_{\sigma + 1}$}
\psframe[fillstyle=solid,fillcolor=gray](12,2)(14,3)
\rput{*0}(13,2.5){$\nu_{\sigma + 1}$}
\psframe[fillstyle=solid,fillcolor=gray](14,3)(16,4)
\rput{*0}(15,3.5){$\nu_{\sigma + 1}$}
\psframe[fillstyle=solid,fillcolor=gray](17,4)(19,5)
\rput{*0}(18,4.5){$\dots$} \rput{*0}(25,3.5){$\mathit{(1)}$}
\end{pspicture}

\vspace{0.2in}

By Corollary~\ref{same-same}, the parts of the right filling of
$\nu_1, \dots, \nu_{\sigma}, \nu_{\sigma + 1}$ over $\mu$ extends
the filling of $\nu_1, \dots, \nu_{\sigma}$ over $\mu$.  In other
words, we can represent the filling of $\nu_1, \dots, \nu_{\sigma}$
over $\mu_1, \dots, \mu_{\ell}, \mu_{\ell +1 }$ using the {\em same}
shape (and the same filling) as above (outlined in bold), after
removing the $\sigma+1$-strip:

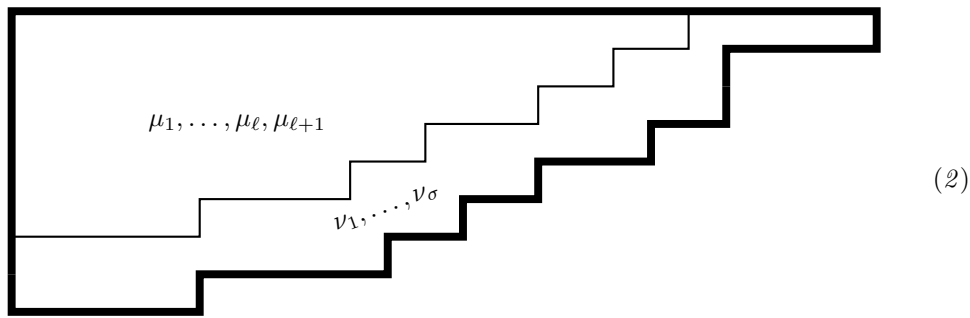
\vspace{0.3in}

\begin{pspicture}(14,8)
\psline[linewidth=3pt](19,6)(19,7)(20,7)(23,7)(23,8)(0,8)(0,1)
\psline(0,2)(5,2)(5,3)(9,3)(9,4)(11,4)(11,5)(14,5)(14,6)(16,6)(16,7)(18,7)(18,8)
\psline[linewidth=3pt](0,1)(0,0)(5,0)(5,1)(10,1)(10,2)(12,2)(12,3)(14,3)(14,4)(17,4)(17,5)(19,5)(19,6)
\rput{*0}(6,5){$\mu_{1}, \dots , \mu_{\ell}, \mu_{\ell +1}$}
\rput{*20}(10,2.75){$\nu_1, \dots , \nu_{\sigma}$}
\rput{*0}(25,3.5){$\mathit{(2)}$}
\end{pspicture}

\vspace{0.2in}

\noindent where the lower portion labeled $\nu_1, \dots
\nu_{\sigma}$ denotes the L${\mathbb R}$ filling of $(\nu_1, \dots ,
\nu_{\sigma})$ over $\mu_1, \dots, \mu_{\ell}, \mu_{\ell + 1}$.

By the inductive hypothesis on $\sigma$, we know the associated left
filling of $\mu_1, \dots, \mu_{\ell}, \mu_{\ell +1}$ over $\nu_1,
\dots, \nu_{\sigma}$ is {\em uniquely} determined by the above
filling of $\nu_{1}, \dots , \nu_{\sigma}$ over $\mu_{1}, \dots,
\mu_{\ell}, \mu_{\ell + 1}$.  We denote this below, while
highlighting the $\mu_{\ell + 1}$-strip:

 \vspace{0.3in}

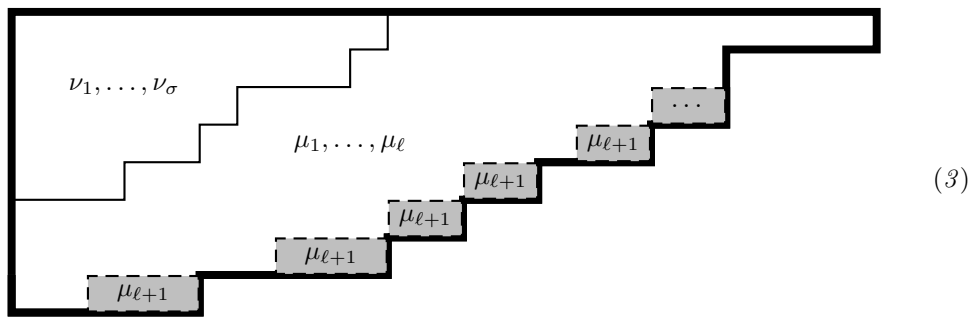
\psset{unit=0.5cm}

\begin{pspicture}(14,8)
\psline[linewidth=3pt](19,6)(19,7)(20,7)(23,7)(23,8)(0,8)(0,1)
\psline[linewidth=3pt](0,1)(0,0)(5,0)(5,1)(10,1)(10,2)(12,2)(12,3)(14,3)(14,4)(17,4)(17,5)(19,5)(19,6)
\psline(0,3)(3,3)(3,4)(5,4)(5,5)(6,5)(6,6)(9,6)(9,7)(10,7)(10,8)
\psline[linewidth=3pt](0,1)(0,0)(5,0)(5,1)(10,1)(10,2)(12,2)(12,3)(14,3)(14,4)(17,4)(17,5)(19,5)(19,6)
\psframe[linestyle=dashed,fillstyle=solid,fillcolor=lightgray](2,0)(5,1)\rput{*0}(3.5,0.5){$\mu_{\ell+1}$}
\psframe[linestyle=dashed,fillstyle=solid,fillcolor=lightgray](7,1)(10,2)\rput{*0}(8.5,1.5){$\mu_{\ell+1}$}
\psframe[linestyle=dashed,fillstyle=solid,fillcolor=lightgray](10,2)(12,3)\rput{*0}(11,2.5){$\mu_{\ell+1}$}
\psframe[linestyle=dashed,fillstyle=solid,fillcolor=lightgray](12,3)(14,4)\rput{*0}(13,3.5){$\mu_{\ell+1}$}
\psframe[linestyle=dashed,fillstyle=solid,fillcolor=lightgray](15,4)(17,5)\rput{*0}(16,4.5){$\mu_{\ell+1}$}
\psframe[linestyle=dashed,fillstyle=solid,fillcolor=lightgray](17,5)(19,6)\rput{*0}(18,5.5){$\dots$}
\rput{*0}(3,6){$\nu_1, \dots , \nu_{\sigma}$}
\rput{*0}(9,4.5){$\mu_{1}, \dots ,
\mu_{\ell}$}\rput{*0}(25,3.5){$\mathit{(3)}$}
\end{pspicture}

 \vspace{0.2in}

The shapes outlined in bold diagrams $\mathit{(1)}$, $\mathit{(2)}$
and $\mathit{(3)}$ above are the same (diagram $\mathit{(2)}$ is the
shape of $\nu_1, \dots, \nu_{\sigma}$ over $\mu_1, \dots, \mu_{\ell
+ 1}$, and diagram $\mathit{(3)}$ is the shape of $\mu_1, \dots,
\mu_{\ell + 1}$ over $\nu_1, \dots, \nu_{\sigma}$). This observation
allows us to draw a ``mixed tableau", which does not have an
immediate matrix interpretation. We will take the tableau depicting
the $\mu$-filling over $(\nu_{1}, \dots , \nu_{\sigma}, 0, \dots )$
outlined in bold in diagram $\mathit{(1)}$ and insert it inside the
tableau of diagram $\mathit{(3)}$ above depicting the filling of
$(\nu_1, \dots , \nu_{\sigma}, \nu_{\sigma + 1}, 0, \dots )$ over
$\mu$ (labeled $\bigstar$ below):

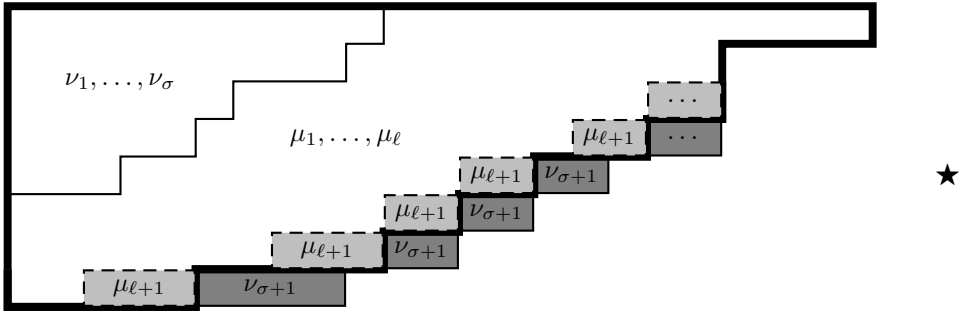
\vspace{0.5in}
\begin{pspicture}(14,6)
\psframe[fillstyle=solid,fillcolor=gray](5,0)(9,1)
\rput{*0}(7,0.5){$\nu_{\sigma + 1}$}
\psframe[fillstyle=solid,fillcolor=gray](10,1)(12,2)
\rput{*0}(11,1.5){$\nu_{\sigma + 1}$}
\psframe[fillstyle=solid,fillcolor=gray](12,2)(14,3)
\rput{*0}(13,2.5){$\nu_{\sigma + 1}$}
\psframe[fillstyle=solid,fillcolor=gray](14,3)(16,4)
\rput{*0}(15,3.5){$\nu_{\sigma + 1}$}
\psframe[fillstyle=solid,fillcolor=gray](17,4)(19,5)
\rput{*0}(18,4.5){$\dots$}
\psline[linewidth=3pt](19,6)(19,7)(20,7)(23,7)(23,8)(0,8)(0,1)
\psline[linewidth=3pt](0,1)(0,0)(5,0)(5,1)(10,1)(10,2)(12,2)(12,3)(14,3)(14,4)(17,4)(17,5)(19,5)(19,6)
\psline(0,3)(3,3)(3,4)(5,4)(5,5)(6,5)(6,6)(9,6)(9,7)(10,7)(10,8)
\psline[linewidth=3pt](0,1)(0,0)(5,0)(5,1)(10,1)(10,2)(12,2)(12,3)(14,3)(14,4)(17,4)(17,5)(19,5)(19,6)
\psframe[linestyle=dashed,fillstyle=solid,fillcolor=lightgray](2,0)(5,1)\rput{*0}(3.5,0.5){$\mu_{\ell+1}$}
\psframe[linestyle=dashed,fillstyle=solid,fillcolor=lightgray](7,1)(10,2)\rput{*0}(8.5,1.5){$\mu_{\ell+1}$}
\psframe[linestyle=dashed,fillstyle=solid,fillcolor=lightgray](10,2)(12,3)\rput{*0}(11,2.5){$\mu_{\ell+1}$}
\psframe[linestyle=dashed,fillstyle=solid,fillcolor=lightgray](12,3)(14,4)\rput{*0}(13,3.5){$\mu_{\ell+1}$}
\psframe[linestyle=dashed,fillstyle=solid,fillcolor=lightgray](15,4)(17,5)\rput{*0}(16,4.5){$\mu_{\ell+1}$}
\psframe[linestyle=dashed,fillstyle=solid,fillcolor=lightgray](17,5)(19,6)\rput{*0}(18,5.5){$\dots$}
\rput{*0}(3,6){$\nu_1, \dots , \nu_{\sigma}$}
\rput{*0}(9,4.5){$\mu_{1}, \dots , \mu_{\ell}$}
\rput{*0}(25,3.5){$\bigstar$}
\end{pspicture}

\vspace{0.3in} The shape of $\bigstar$ above is the same as the
shape of $inv(\mu_{\ell +1}, \mu_{\ell}, \dots, \mu_{1}, L, \nu_{1},
\dots, \nu_{\sigma}, \nu_{\sigma + 1})$.  In particular, this is the
same shape obtained from the tableau depicting the $\mu$-filling of
$(D_{\mu}, LD_{(0,\dots , 0, \nu_{\sigma + 1}, \nu_{\sigma},\dots ,
\nu_{1})})$, denoted by $\bigstar \bigstar$ below:

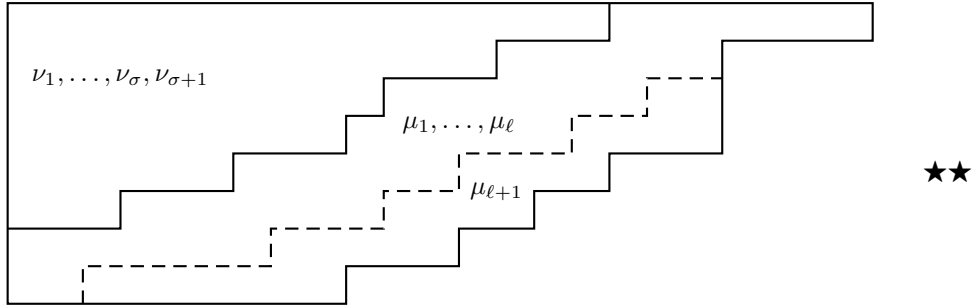
\vspace{0.2in}

\begin{pspicture}(14,8)
\newgray{llgray}{0.92}
\pspolygon(0,8)(0,0)
(4,0)(9,0)(9,1)(12,1)(12,2)(14,2)(14,3)(16,3)(16,4)(19,4)(19,7)(20,7)(23,7)(23,8)
\psline(0,2)(3,2)(3,3)(6,3)(6,4)(9,4)(9,5)(10,5)(10,6)(13,6)(13,7)(16,7)(16,8)
\rput{*0}(25,3.5){$\bigstar \bigstar$}\rput{*0}(3,6){$\nu_1, \dots ,
\nu_{\sigma}, \nu_{\sigma + 1}$} \rput{*0}(12,4.75){$\mu_{1}, \dots
, \mu_{\ell}$} \rput{*0}(13,3){$\mu_{\ell + 1}$}
\psline[linestyle=dashed](2,0)(2,1)(7,1)(7,2)(10,2)(10,3)(12,3)(12,4)(15,4)(15,5)(17,5)(17,6)(19,6)
\end{pspicture}

\hspace{1in}

The $\mu_{\ell +1}$-strip in $\bigstar \bigstar$ lies {\em
somewhere} below the dashed line, which represents the boundary of
the shape of $inv(0, \dots , 0 , \mu_{\ell}, \dots, \mu_1, L,
\nu_1,\dots , \nu_{\sigma})$, since, as described above, the shape
of $\mu_1, \dots, \mu_{\ell}$ over $\nu_{1} , \dots, \nu_{\sigma}$
lies inside the final shape of $\bigstar$, so that extending to
$\mu_{\ell +1}$, whether we do this before extending by $\nu_{\sigma
+ 1}$ (as depicted in $\bigstar$), of after (as depicted in
$\bigstar \bigstar$), must result in the $\mu_{\ell + 1}$-strip
lying below the dashed line of $\bigstar \bigstar$. In particular,
the number of boxes in $\bigstar \bigstar$ below the dashed line is
$\mu_{\ell + 1} + \nu_{\sigma + 1}$.

Note that in $\bigstar$, the shape of the outer $\nu_{\sigma
+1}$-strip is given, by hypothesis, from the right filling of
$(\nu_{1}, \dots , \nu_{\sigma}, \nu_{\sigma + 1}, 0, \dots )$ over
$(\mu_{1}, \dots, \mu_{\ell}, \mu_{\ell + 1}, 0, \dots )$. Further,
the shape of the $\mu_{\ell + 1}$-strip in $\bigstar$ is given by
the induction hypothesis on $\sigma$, where we assume the left
filling of $(\mu_{1}, \dots, \mu_{\ell}, \mu_{\ell + 1}, 0, \dots )$
over $(\nu_{1}, \dots , \nu_{\sigma}, 0, \dots )$ is uniquely
determined by the right filling of $(\nu_{1}, \dots , \nu_{\sigma},
0, \dots )$ over $(\mu_{1}, \dots, \mu_{\ell}, \mu_{\ell + 1}, 0,
\dots )$.

As noted above, the skew shape below the dashed line in $\bigstar
\bigstar$ is the same as the shape formed by the $\mu_{\ell + 1}$
and $\nu_{\sigma + 1}$-strips in $\bigstar$.  However, the (as yet
undetermined) $\mu_{\ell + 1}$ -strip in $\bigstar \bigstar$
actually depicts the terms $m_{\ell + 1, j}$ in the left filling of
$\mu = (\mu_1, \dots, \mu_{\ell + 1}, 0 \dots 0)$ over $(\nu_1,
\dots, \nu_{\sigma}, \nu_{\sigma + 1}, 0, \dots, 0)$.

Let us call the region formed by the $\mu_{\ell + 1}$ and
$\nu_{\sigma + 1}$-strips in $\bigstar$ the ``switching region''.
This name is chosen since we will show that from the location of the
$\nu_{\sigma + 1}$ strip in the switching region of $\bigstar$
(which is determined by the right filling of $\nu$ by hypothesis),
and the overall shape of the switching region (which is determined
by the inductive hypothesis applied to $\sigma$), that the shape of
the $\mu_{\ell + 1}$ strip in the switching region of $\bigstar
\bigstar$ is {\em uniquely} determined.  By the Ordering Lemma and
Corollary~\ref{same-same}, once we have determined the shape of the
$\mu_{\ell + 1}$-strip by the above argument, we may reduce
$\mu_{\ell + 1}$ to zero without changing the parts of the filling
of $\mu_{1}, \dots, \mu_{\ell}, 0, \dots 0$ over $\nu_{1}, \dots,
\nu_{\sigma}, \nu_{\sigma + 1}$.  Thus, by downward induction on
$\ell$, we may conclude that the right filling of $\mu$ is also
uniquely determined for ${\mathbb R}$ partitions $\nu$ of length
$\sigma + 1$.

Thus, it remains to show that the $\mu_{\ell + 1}$ strip in
$\bigstar$ is uniquely determined by the $\nu_{\sigma + 1}$ strip
and the shape of the switching region.

Let us first break up the $\nu_{\sigma + 1}$-strip into finitely
many pieces $P_{1}, \dots , P_{k}$ such that any piece is either
wholly below a block of the $\mu_{\ell +1}$-strip in $\bigstar$, or
entirely not under such a block. Similarly, we may break up the
$\mu_{\ell + 1}$-strip in $\bigstar$ into pieces $Q_1, \dots , Q_m$
such that each $Q_i$ is over (and equal in size) to some $P_j$, or
not supported by any of the $\nu_{\sigma + 1}$-strip at all.  That
is, we may re-draw the switching region in $\bigstar$:

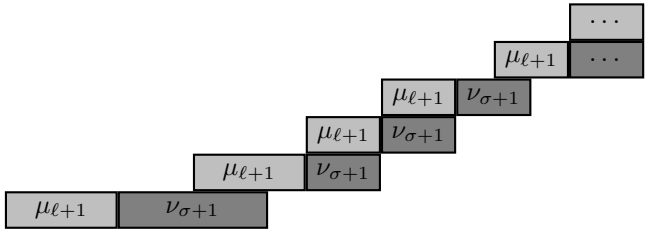
\vspace{0.3in}

\begin{pspicture}(14,6)
\newgray{llgray}{0.92}
\psframe[fillstyle=solid,fillcolor=gray](5,0)(9,1)
\rput{*0}(7,0.5){$\nu_{\sigma + 1}$}
\psframe[fillstyle=solid,fillcolor=gray](10,1)(12,2)
\rput{*0}(11,1.5){$\nu_{\sigma + 1}$}
\psframe[fillstyle=solid,fillcolor=gray](12,2)(14,3)
\rput{*0}(13,2.5){$\nu_{\sigma + 1}$}
\psframe[fillstyle=solid,fillcolor=gray](14,3)(16,4)
\rput{*0}(15,3.5){$\nu_{\sigma + 1}$}
\psframe[fillstyle=solid,fillcolor=gray](17,4)(19,5)
\rput{*0}(18,4.5){$\dots$}
\psframe[fillstyle=solid,fillcolor=lightgray](2,0)(5,1)\rput{*0}(3.5,0.5){$\mu_{\ell+1}$}
\psframe[fillstyle=solid,fillcolor=lightgray](7,1)(10,2)\rput{*0}(8.5,1.5){$\mu_{\ell+1}$}
\psframe[fillstyle=solid,fillcolor=lightgray](10,2)(12,3)\rput{*0}(11,2.5){$\mu_{\ell+1}$}
\psframe[fillstyle=solid,fillcolor=lightgray](12,3)(14,4)\rput{*0}(13,3.5){$\mu_{\ell+1}$}
\psframe[fillstyle=solid,fillcolor=lightgray](15,4)(17,5)\rput{*0}(16,4.5){$\mu_{\ell+1}$}
\psframe[fillstyle=solid,fillcolor=lightgray](17,5)(19,6)\rput{*0}(18,5.5){$\dots$}
\end{pspicture}

\vspace{0.3in}

so that it appears in the form:

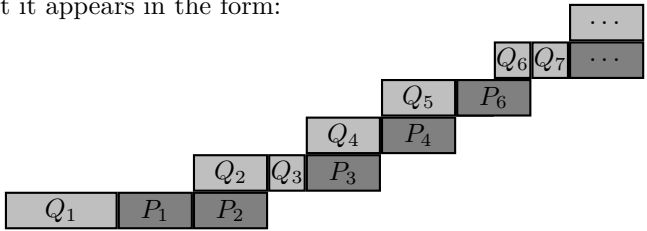
\vspace{0.3in}

\begin{pspicture}(14,4)
\newgray{llgray}{0.92}
\psframe[fillstyle=solid,fillcolor=gray](5,0)(7,1)
\rput{*0}(6,0.5){$P_1$}
\psframe[fillstyle=solid,fillcolor=gray](7,0)(9,1)
\rput{*0}(8,0.5){$P_2$}
\psframe[fillstyle=solid,fillcolor=gray](10,1)(12,2)
\rput{*0}(11,1.5){$P_3$}
\psframe[fillstyle=solid,fillcolor=gray](12,2)(14,3)
\rput{*0}(13,2.5){$P_4$}
\psframe[fillstyle=solid,fillcolor=gray](14,3)(15,4)
\rput{*0}(14.5,3.5){$P_5$}
\psframe[fillstyle=solid,fillcolor=gray](14,3)(16,4)
\rput{*0}(15,3.5){$P_6$}
\psframe[fillstyle=solid,fillcolor=gray](17,4)(19,5)
\rput{*0}(18,4.5){$\dots$}
\psframe[fillstyle=solid,fillcolor=lightgray](2,0)(5,1)
\rput{*0}(3.5,0.5){$Q_1$}
\psframe[fillstyle=solid,fillcolor=lightgray](7,1)(9,2)
\rput{*0}(8,1.5){$Q_2$}
\psframe[fillstyle=solid,fillcolor=lightgray](9,1)(10,2)
\rput{*0}(9.5,1.5){$Q_3$}
\psframe[fillstyle=solid,fillcolor=lightgray](10,2)(12,3)
\rput{*0}(11,2.5){$Q_4$}
\psframe[fillstyle=solid,fillcolor=lightgray](12,3)(14,4)
\rput{*0}(13,3.5){$Q_5$}
\psframe[fillstyle=solid,fillcolor=lightgray](15,4)(16,5)
\rput{*0}(15.5,4.5){$Q_6$}
\psframe[fillstyle=solid,fillcolor=lightgray](16,4)(17,5)
\rput{*0}(16.5,4.5){$Q_7$}
\psframe[fillstyle=solid,fillcolor=lightgray](17,5)(19,6)
\rput{*0}(18,5.5){$\dots$}
\end{pspicture}

\vspace{0.3in}
 Since the length of the $\mu_{\ell +1}$-strip in $\bigstar \bigstar$ is equal
 to the sum of the lengths of the $Q_{i}$, it will suffice to prove that the
 $\mu_{\ell  +1}$-strip in $\bigstar \bigstar$ is formed by shifting the
 $Q_{i}$ blocks appearing in $\bigstar$, and that this shifting is uniquely
 determined by the location of the $P_{j}$ in the switching region.  In fact,
 we will show that this ``switching", though exhibited in a purely matrix
 setting, corresponds exactly to the row-switching algorithm appearing in \cite{kerber}
 and generalized in \cite{Sottile}.

 To obtain this, we first note that, by the ordering lemma, the filling of $\nu_1, \dots, \
 \nu_{\sigma}$ over $\mu$ is unchanged by decreasing the value of
 $\nu_{\sigma + 1}$.  That is, if we replaced $(D_{\mu}, L diag(1,
 \dots, 1, t^{\nu_{\sigma + 1}}, t^{\nu_{\sigma}}, \dots ,
 t^{\nu_1})$ with $(D_{\mu}, L \, diag(1,
 \dots, 1, t^{s}, t^{\nu_{\sigma}}, \dots ,
 t^{\nu_1})$, for some $s$ with $0 < s < \nu_{\sigma + 1}$, the
 parts of the filling of $\nu_1, \dots, \nu_{\sigma}$ remain unchanged.  Consequently, we
 can reduce the value of $\nu_{\sigma + 1}$ (temporarily) so that only the initial
block $P_{1}$ appears in the $\sigma + 1$ strip, and
 re-draw the switching region in $\bigstar$:

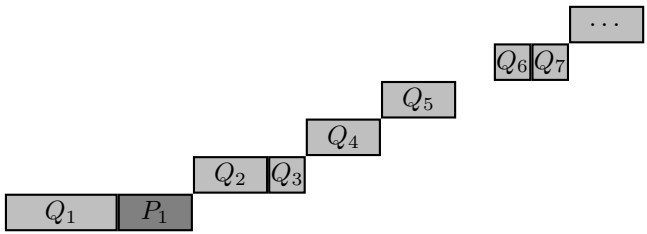
\begin{pspicture}(14,7)
\newgray{llgray}{0.92}
\psframe[fillstyle=solid,fillcolor=gray](5,0)(7,1)
\rput{*0}(6,0.5){$P_1$}
\psframe[fillstyle=solid,fillcolor=lightgray](2,0)(5,1)
\rput{*0}(3.5,0.5){$Q_1$}
\psframe[fillstyle=solid,fillcolor=lightgray](7,1)(9,2)
\rput{*0}(8,1.5){$Q_2$}
\psframe[fillstyle=solid,fillcolor=lightgray](9,1)(10,2)
\rput{*0}(9.5,1.5){$Q_3$}
\psframe[fillstyle=solid,fillcolor=lightgray](10,2)(12,3)
\rput{*0}(11,2.5){$Q_4$}
\psframe[fillstyle=solid,fillcolor=lightgray](12,3)(14,4)
\rput{*0}(13,3.5){$Q_5$}
\psframe[fillstyle=solid,fillcolor=lightgray](15,4)(16,5)
\rput{*0}(15.5,4.5){$Q_6$}
\psframe[fillstyle=solid,fillcolor=lightgray](16,4)(17,5)
\rput{*0}(16.5,4.5){$Q_7$}
\psframe[fillstyle=solid,fillcolor=lightgray](17,5)(19,6)
\rput{*0}(18,5.5){$\dots$}
\end{pspicture}

\vspace{0.3in}

In fact, there are two possible cases to consider in this reduction,
depending on the shape of the $\nu_{\sigma+1}$-strip. Case One,
which will look like the case above, where the initial block $P_1$
of the $\nu_{\sigma + 1}$-strip lies directly to the right of some
block $Q_{i}$ of the $\mu_{\ell + 1}$-strip of the mixed picture,
with no block of the $\mu_{\ell + 1}$-strip lying above it. In Case
Two, we might see a picture of the form:

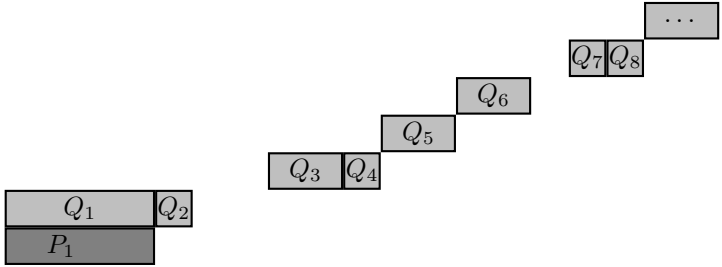
\vspace{0.3in}

\begin{pspicture}(14,6)
\newgray{llgray}{0.92}
\psframe[fillstyle=solid,fillcolor=gray](0,-1)(4,0)
\rput{*0}(1.5,-0.5){$P_1$}
\psframe[fillstyle=solid,fillcolor=lightgray](0,0)(4,1)
\rput{*0}(2,0.5){$Q_1$}
\psframe[fillstyle=solid,fillcolor=lightgray](4,0)(5,1)
\rput{*0}(4.5,0.5){$Q_2$}
\psframe[fillstyle=solid,fillcolor=lightgray](7,1)(9,2)
\rput{*0}(8,1.5){$Q_3$}
\psframe[fillstyle=solid,fillcolor=lightgray](9,1)(10,2)
\rput{*0}(9.5,1.5){$Q_4$}
\psframe[fillstyle=solid,fillcolor=lightgray](10,2)(12,3)
\rput{*0}(11,2.5){$Q_5$}
\psframe[fillstyle=solid,fillcolor=lightgray](12,3)(14,4)
\rput{*0}(13,3.5){$Q_6$}
\psframe[fillstyle=solid,fillcolor=lightgray](15,4)(16,5)
\rput{*0}(15.5,4.5){$Q_7$}
\psframe[fillstyle=solid,fillcolor=lightgray](16,4)(17,5)
\rput{*0}(16.5,4.5){$Q_8$}
\psframe[fillstyle=solid,fillcolor=lightgray](17,5)(19,6)
\rput{*0}(18,5.5){$\dots$}
\end{pspicture}

\vspace{0.3in}

By Corollary~\ref{don't-move}, if we reduce $\mu_{\ell + 1}$ so that
only $Q_{1}$ remains (in either of the cases), then the size and
location of the part labeled $P_{1}$ will not move, either.

{\bf Case 1:}  Applying Corollary~\ref{don't-move} and reducing the
size of the $\mu_{\ell + 1}$-strip we may depict this case with the
picture:

\begin{pspicture}(14,2)
\newgray{llgray}{0.92}
\psframe[fillstyle=solid,fillcolor=gray](5,0)(7,1)
\rput{*0}(6,0.5){$P_1$}
\psframe[fillstyle=solid,fillcolor=lightgray](2,0)(5,1)
\rput{*0}(3.5,0.5){$Q_j$}
\end{pspicture}

\vspace{0.3in}

$P_1$ does {\em not} lie below any block of the $\mu_{\ell
+1}$-strip. Then, all blocks $Q_h$ lying in a row {\em lower} than
the row in which $P_1$ appears must remain in the {\em same}
location in the $\mu_{\ell +1}$-strip of $\bigstar \bigstar$ as they
do in the switching region of $\bigstar$, since this must be a
horizontal strip and, by the ordering lemma, can be filled in only
one way.  The only way that the block $Q_{j}$ can appear in the
$\mu_{\ell + 1}$-strip would be to the right of the block $P_{1}$,
since by the Ordering Lemma, the block $P_{1}$ can only appear in
this row, so for the $\mu_{\ell + 1}$-strip to be a horizontal
strip, the block $Q_{j}$ must appear after this, to the right.

\medskip {\bf Case 2:} $P_1$ lies directly under some $Q_j$. That is, the
shapes look like:

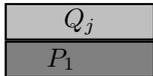
\begin{pspicture}(6,2)
\newgray{llgray}{0.92}
\psframe[fillstyle=solid,fillcolor=gray](0,-1)(4,0)
\rput{*0}(1.5,-0.5){$P_1$}
\psframe[fillstyle=solid,fillcolor=lightgray](0,0)(4,1)
\rput{*0}(2,0.5){$Q_j$}
\end{pspicture}
\vspace{0.3in}

By the Ordering Lemma applied to the growth of the $\mu_{\ell +
1}$-strip in $\bigstar \bigstar$, we can conclude that all blocks
$Q_{h}$, for $1 \leq h < j$, must appear in the same location as
they do in the switching region of $\bigstar$ (by considering the
case when $P_{1}$ is of zero length), since in this case the
switching region is a horizontal strip, so there this no choice for
the growth of the $\mu_{\ell + 1}$-strip. However, we claim that
$Q_j$ must ``switch'' with $P_1$.  That is, the block $Q_j$ in the
$\mu_{\ell +1}$-strip must occupy the space of $P_1$. Why?  By
construction, the stack of blocks $Q_j$ over $P_1$ in the switching
region of $\bigstar$ is the only part of the region that is has some
part of a row over another.  If $Q_j$ does not occupy the space
taken by $Q_{1}$ in the $\mu_{\ell + 1}$-strip of $\bigstar
\bigstar$, the same space that $P_1$ occupies in the switching
region of $\bigstar$, then there is no way that the growth of the
$\mu_{\ell + 1}$-strip can be a horizontal strip, as it must.  We
conclude that in this case, the locations of all blocks $Q_h$ of the
$\mu_{\ell + 1}$-strip, lying above or to the left of the first
block $P_1$ of the $\nu_{\sigma + 1}$-strip in $\bigstar$, are
completely determined by the $P_1$ and the shape of the switching
region.

\medskip

All that remains is to remark that, by increasing the $\nu_{\sigma +
1}$-strip in $\bigstar \bigstar$, block by block, and arguing in the
two cases as above, each successive block $Q_{j}$ of the $\mu_{\ell
+ 1}$-strip will be determined.  The blocks $Q_{j}$ that appear in
rows above the $\nu_{\sigma +1}$-strip in the switching region are
uniquely determined, of course, once all of the $\nu_{\sigma +
1}$-strip has been located, along with the parts of the $\mu_{\ell +
1}$-strip that appear in lower rows.
\end{proof}

\end{document}